\documentclass{article}

\usepackage{graphicx}
\usepackage{amssymb}
\usepackage[pdftex]{hyperref}
\usepackage{xcolor}
\hypersetup{
	colorlinks,
	linkcolor={red!50!black},
	citecolor={blue!90!black},
	urlcolor={blue!80!black}
}
\usepackage{cite}
\usepackage[hmargin={2.5cm,2.5cm}, vmargin={3cm,3cm}, dvips]{geometry}

\usepackage{fancyhdr}
\usepackage{amsmath} 
\usepackage{amssymb} 
\usepackage{amsthm} 
\usepackage{array}
\usepackage{placeins}
\usepackage{enumerate}
\usepackage{esint}
\usepackage{graphics}
\usepackage{graphicx}
\usepackage{tikz}
\usetikzlibrary{patterns}
\usetikzlibrary{perspective}
\usetikzlibrary{3d}
\usepackage{pgfplots}
\usepackage{caption}
\usepackage{subcaption}
\usepackage{listings}
\usepackage[shortlabels]{enumitem}
\usepackage{multirow}
\usepackage{multicol}
\usepackage{wrapfig}

\usepackage{afterpage}
\usepackage{mathrsfs}
\usepackage{mathtools}

\newtheoremstyle{droit}
{}
{}
{\upshape}
{}
{\bfseries}
{}
{ }
{}
\newtheoremstyle{italique}
{}
{}
{\itshape}
{}
{\bfseries}
{}
{ }
{}
\theoremstyle{italique}
\newtheorem{theorem}{Theorem}[section]

\theoremstyle{droit}

\newtheorem{remark}[theorem]{Remark}

\newtheorem{definition}[theorem]{Definition}

\usepackage{geometry}
\usetikzlibrary{shapes,backgrounds}
\newcommand{\tstar}[5]{
\pgfmathsetmacro{\starangle}{360/#3}
\draw[#5] (#4:#1)
\foreach \x in {1,...,#3}
{ -- (#4+\x*\starangle-\starangle/2:#2) -- (#4+\x*\starangle:#1)
}
-- cycle;
}

\newcommand{\setminussign}{{\mathrm r}}
\newcommand{\intersign}{{\mathrm s}}

\newtheorem{mylemma}{Appendix}[section]

\begin{document}



\title{\textbf{Analysis-aware defeaturing: \\problem setting and \textit{a posteriori} estimation}}


\author{{A. Buffa$^{1,2}$, O. Chanon$^1$, R. V\'azquez$^{1,2}$}\\ \\
	\footnotesize{$^1$ MNS, Institute of Mathematics, \'Ecole Polytechnique F\'ed\'erale de Lausanne, Switzerland}\\
	\footnotesize{$^2$ Istituto di Matematica Applicata e Tecnologie Informatiche `E. Magenes' (CNR), Pavia, Italy}
}
 \date{November 23, 2021}

\maketitle
\vspace{-0.8cm}
\noindent\rule{\linewidth}{0.4pt}
\thispagestyle{fancy}
\begin{abstract}
Defeaturing consists in simplifying geometrical models by removing the geometrical features that are considered not relevant for a given simulation. Feature removal and simplification of computer-aided design models enables faster simulations for engineering analysis problems, and simplifies the meshing problem that is otherwise often unfeasible. 
The effects of defeaturing on the analysis 
are then neglected and, as of today, there are basically very few strategies to quantitatively evaluate such an impact. 
Understanding well the effects of this process is an important step for automatic integration of design and analysis.
We formalize the process of defeaturing by understanding its effect on the solution of {Poisson} equation defined on the geometrical model of interest containing a single feature, with Neumann boundary conditions on the feature itself. 
We derive an \textit{a posteriori} estimator of the energy error between the solutions of the exact and the defeatured geometries in $\mathbb R^n$, $n\in\{2,3\}$, that is simple, {reliable and efficient} up to oscillations. The dependence of the estimator upon the size of the features is explicit.
\end{abstract}


\section{Introduction} \label{s:intro}
Complex geometrical models are created and processed using computer-aided design tools (CAD) in the context of computer-aided engineering. The automatic integration of design and analysis tools in a single workflow has been an important topic of research for many years \cite{farouki,riesenfield}. One of the methodologies that emerged in the last $15$ years is the one based on isogeometric analysis (IGA) \cite{igabook,igabasis}, a method to solve partial differential equations (PDEs) using smooth B-splines, NURBS or variances thereof as basis functions for the solution field. IGA has proven to be a valid simulation method in a wide range of applications \cite{igaappli}, and a sound mathematical theory \cite{igahref,igaanalysis}, including strategies for adaptive refinement \cite{tsplines,buffagarau,buffagiannelli1,buffagiannelli2}, is now available.

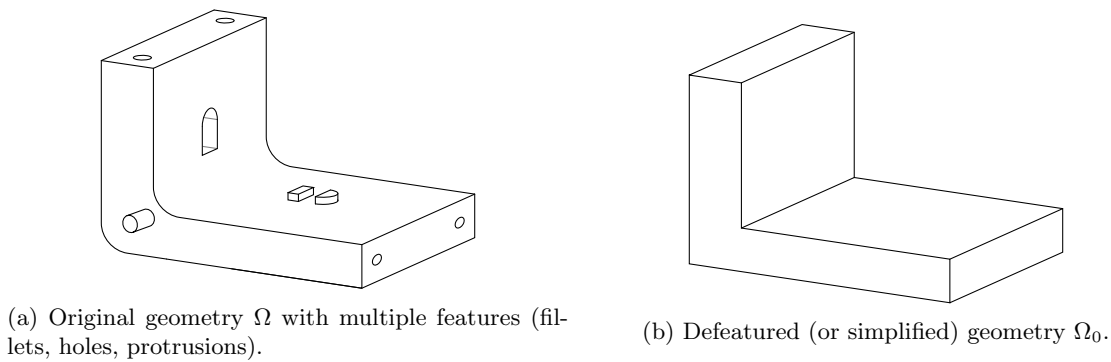
\begin{figure}
	\centering
	\begin{subfigure}{0.45\textwidth}
		\begin{center}
			\begin{tikzpicture}[scale=2,3d view={30}{15},line cap=round,
			declare function={ax=2;ay=1.5;az=1.3;bx=0.4;bz=0.3;corn=0.35;rr=0.1;}]
			\draw (0,ay,az) -- (bx,ay,az)[rounded corners=.35cm] -- (bx,ay,bz) -- (ax,ay,bz); 
			\draw (ax,ay,bz) -- (ax,ay,0) -- (ax,0,0) -- (ax/2,0,0); 
			\draw (bx,0,az) -- (bx,ay,az);
			\draw (0.5,0,0) -- (ax,0,0) -- (ax,0,bz);
			\draw[rounded corners=.35cm] (ax,0,bz) -- (bx,0,bz) -- (bx,0,az);
			\draw (bx,0,az) -- (0,0,az);
			\draw[rounded corners=.35cm] (0,0,az) -- (0,0,0) -- (0.5,0,0);
			\draw (ax,0,bz) -- (ax,ay,bz);
			\draw (0,0,az) -- (0,ay,az);
			
			\draw(bx/2,.2,az)ellipse[x radius=.06,y radius=.06]; 
			\draw(bx/2,ay-.2,az)ellipse[x radius=.06,y radius=.06];	
			
			\begin{scope}[transform canvas={rotate=30}]
			\node [draw, cylinder, shape aspect=0.8, rotate=170, minimum height=1mm, minimum width=1mm,shift={(-2*bx, 0, -bz)}] {};
			\end{scope}
			
			\draw (ax, 0.2, bz/2) [y={(0,1,1)}] circle (0.03); 
			\draw (ax, ay-0.2, bz/2) [y={(0,1,1)}] circle (0.03); 
			
			\draw (ax/2, ay/2, bz+0.05) -- (ax/2+0.08, ay/2, bz+0.05) -- (ax/2+0.08, ay/2+0.2, bz+0.05) -- (ax/2, ay/2+0.2, bz+0.05) -- (ax/2, ay/2, bz+0.05); 
			\draw (ax/2, ay/2, bz+0.05) -- (ax/2, ay/2, bz) -- (ax/2+0.08, ay/2, bz) -- (ax/2+0.08, ay/2+0.2, bz) -- (ax/2+0.08, ay/2+0.2, bz+0.05);
			\draw (ax/2+0.08, ay/2, bz) -- (ax/2+0.08, ay/2, bz+0.05);
			
			\draw (ax/2+0.21, ay/2, bz) -- (ax/2+0.21, ay/2, bz+0.05) -- (ax/2+0.21, ay/2+0.2, bz+0.05);
			\begin{scope}[canvas is xy plane at z=bz+0.05]
			\draw (ax/2+0.21, ay/2+0.2) arc(90:-90:0.1); 
			\end{scope}
			\begin{scope}[canvas is xy plane at z=bz]
			\draw (ax/2+0.285, ay/2+0.168) arc(40:-90:0.1); 
			\end{scope}
			\draw (ax/2+0.21+0.06, ay/2+0.195, bz-0.015) --  (ax/2+0.21+0.06, ay/2+0.195, bz+0.035);
			
			\draw[gray] (bx, ay/2+rr, az/2+rr) -- (bx-0.1, ay/2+rr, az/2+rr);
			\draw[gray] (bx, ay/2+rr, az/2-rr) -- (bx-0.11, ay/2+rr, az/2-rr);
			\draw (bx, ay/2-rr, az/2+rr) -- (bx, ay/2-rr, az/2-rr) -- (bx, ay/2+rr, az/2-rr) -- (bx, ay/2+rr, az/2+rr);
			\begin{scope}[canvas is zy plane at x=bx]
			\draw (az/2+rr, ay/2+rr) arc(90:-90:rr) ; 
			\end{scope}
			\end{tikzpicture}
			\caption{Original geometry $\Omega$ with multiple features (fillets, holes, protrusions).}
		\end{center}
	\end{subfigure}
	~ 
	\begin{subfigure}{0.45\textwidth}
		\begin{center}
			\begin{tikzpicture}[scale=2,3d view={30}{15},line cap=round,
			declare function={ax=2;ay=1.5;az=1.3;bx=0.4;bz=0.3;corn=0.35;}]
			\draw (0,0,0) -- (ax,0,0) -- (ax,0,bz) -- (bx,0,bz) -- (bx,0,az)
			-- (0,0,az) -- cycle 
			(0,0,az) -- (0,ay,az) 
			-- (bx,ay,az) edge ++ (0,-ay,0)
			-- (bx,ay,bz) edge ++ (0,-ay,0) 
			-- (ax,ay,bz) edge ++ (0,-ay,0)
			-- (ax,ay,0) -- (ax,0,0);
			\end{tikzpicture}
			\caption{Defeatured (or simplified) geometry $\Omega_0$.}
		\end{center}
	\end{subfigure}
	\caption{Illustration of defeaturing.} \label{fig:illustrationdefeat}
\end{figure}

However, a major challenge remains in the usability of complex CAD geometries in the analysis phase. While the first CAD models used in IGA were relatively simple geometries defined by multiple patches \cite{multipatch,igabook}, in recent years, more effort is being dedicated to the analysis on complex geometries defined via Boolean operations such as intersections (trimming)\cite{antolinvreps,trimming,trimmingbletzinger} and unions\cite{antolinwei,unionkargaran,unionzuo}. The related engineering literature includes in particular the shell analysis on models with B-reps\cite{bletzinger1,trimmedshells,bletzinger2}, and the finite cell method combined with IGA on complex geometries\cite{rankcomplexgeom,rank2,rank3}. Before even doing any simulation on complex geometries, defining them may already require a very large number of degrees of freedom, that are not necessarily needed --and potentially too costly to be taken into account-- to perform an accurate analysis. Moreover, repeated design changes is part of a typical process in simulation-based design for manufacturing, and it involves adding or removing geometrical features to the design, as well as adjusting geometric parameters in order to meet functionality, manufacturability and aesthetic requirements. Therefore, to be able to consider complex geometries and to accelerate the process of analysis-aware geometric design, it is essential to be able to simplify the geometrical model, process also called defeaturing, while understanding its effect on the solution of the problem in hand. The idea of defeaturing is illustrated in Fig.~\ref{fig:illustrationdefeat}, where we show a complex geometry and its simplified version, with all the features removed.

For a long time, the defeaturing problem has been approached using subjective \textit{a priori} criteria, relying mostly on the engineers' expertise or based on geometrical considerations such as variations in volume or area of the domain\cite{surveymodelsimpl}. More objective criteria have then been considered, still based on some \textit{a priori} knowledge of the mechanical problem at hand such as the verification of constitutive or conservation laws\cite{fineremondinileon2000,rahimi2018}. However, in order to automatize the simulation-based design process, the interest is to have an \textit{a posteriori} criterion to assess the error introduced by defeaturing from the result of the analysis in the defeatured geometric model.
Following this direction, an \textit{a posteriori} criterion is given in Ref.~\cite{ferrandesgiannini2009}: it evaluates an approximation of the energy norm between the exact solution of the problem at hand and the solution on the defeatured geometry. It is intuitively based on the fact that the energy error due to defeaturing is concentrated in the modified boundaries of the geometry, and this boundary error is estimated by solving local problems around each feature. Nevertheless, this approach does not give a demonstrated certification that the proposed criterion is indeed a good estimator of the defeaturing error. 

A different approach is based on the concept of feature sensitivity analysis (FSA)\cite{gopalakrishnansuresh2008,turevsky2008defeaturing}, which relies on topological sensitivity analysis\cite{tsa1,tsa2}, a method used in design optimization that studies the impact of infinitesimal (topological) geometrical changes on the solution of a given PDE. The works on FSA study the defeaturing in geometries with a single arbitrarily-shaped feature. First order changes of quantities of interest are analyzed when a small internal or boundary hole is removed from the geometry. However, besides the underlying assumption of infinitesimal features, this technique cannot be generalized to more complex features. 

An alternative approach, still based on \textit{a posteriori} error estimators, is proposed in Ref.~\cite{ligaomartin2011} for internal holes. The idea behind this estimator is to reformulate the geometrical defeaturing error as a modeling error, by rewriting the PDE solved in two different geometries as two different PDEs on a unique geometry. The modeling error is then estimated using the dual weighted residual method introduced in Ref.~\cite{beckerrannacher} and Ref.~\cite{odenprudhomme}, following the lines of Refs.~\cite{carstensensauter2004,odenvemaganti2000,vegamanti2004} that study heterogeneous and perforated materials, and Ref.~\cite{repinsautersmolianski} that studies the error introduced by the approximation of boundary conditions, two problems that can be easily related to defeaturing. 
This a posteriori approach has then been generalized to different linear and non-linear problems, and to other types of features, in Refs.~\cite{ligao2011,ligaozhang2012,ligaomartin2013,zhanglili2016}. However, some heuristic remains in all these contributions, and a precise mathematical study of the estimator with regards to its efficiency and stability is lacking. In particular, it is assumed that the difference between the solutions of the PDE in the exact and defeatured geometries is small, and it relies on the heuristic estimation of constants that depend on the size of the features, but are not explicit with respect to it. 

Consequently, the first aim of this paper is to give a solid mathematical framework for analysis-based defeaturing, and to precisely define the defeaturing error. We consider geometries that contain a single feature, the case of a geometry with multiple features being the subject of our subsequent work in preparation. We work in the context of {Poisson} equation for which Neumann boundary conditions are imposed on the features, and we introduce an \textit{a posteriori} estimator of the defeaturing error that explicitly depends on the size of the features. 
Our estimator is easy to compute, very cheap, and it is proven to be reliable and efficient {up to oscillations}. Moreover, the considered features are very general, they can either be negative (internal or boundary holes), positive (protrusions), or more complex with both positive and negative components. As mentioned earlier, the estimator is very cheap to compute{: }after the computation of the solution in the defeatured domain, it only requires the solution of a local problem in a simplified feature (as, e.g., its bounding box) if the feature is positive, and the computation of local boundary integrals. 
{Indeed, the proposed estimator is derived from a representation of the defeaturing error that only involves differences between boundary terms on the feature, as already observed in Refs.~\cite{gopalakrishnansuresh2008,turevsky2008defeaturing} and in Refs.~\cite{ligao2011,ligaozhang2012,ligaomartin2013}.}

After introducing the notation used throughout the article in Sec.~\ref{sec:notation}, we precisely define the defeaturing problem in Sec.~\ref{sec:defeatpb} in the simpler setting in which the feature is either negative or positive. Then, in Sec.~\ref{sec:negativeest}, the defeaturing error estimator is derived and analyzed in the case in which the geometry contains a negative feature. In Sec.~\ref{sec:complexfeat}, the defeaturing problem and error estimator are generalized to the case of a geometry with a complex feature, required by complex geometric models. The study of the defeaturing problem when the feature is positive can be deduced from Sec.~\ref{sec:complexfeat} as a special case. 
Subsequently, in Sec.~\ref{sec:numexp}, we present a validation of the results presented previously. Our validation is obtained by comparing errors and defeaturing estimators for numerical solutions on very fine meshes. We finally draw conclusions in Sec.~\ref{sec:ccl}, and some mathematical results used throughout the paper are stated and proven in Appendix \ref{sec:appendix}. 

\section{Notation} \label{sec:notation}
We start by introducing the notation that will be used throughout the paper. Let $n=2$ or $n=3$, let $\omega$ be any open $k$-dimensional manifold in $\mathbb{R}^n$, $k\leq n$, and let $\varphi\subset \partial \omega$. 

We denote by $|\omega|$, $\overline{\omega}$, $\mathrm{int}(\omega)$ and $\text{conn}(\omega)$, respectively, the measure of $\omega$, its closure, its interior, and the set of its connected components. We also denote by $\mathrm{diam}(\varphi)$ and $\mathrm{hull}(\varphi)$, respectively, the manifold diameter and convex hull. More precisely, if $\varphi$ is connected, we let
$\mathrm{diam}(\varphi):= \max_{\xi, \eta \in \varphi} \rho(\xi, \eta)$, where $\rho(\xi, \eta)$ is the infimum of lengths of continuous piecewise $C^1$-paths between $\xi$ and $\eta$ in $\varphi$. If $\varphi$ is not connected, by abuse of notation, we denote by $\mathrm{diam}(\varphi)$ the diameter of $\mathrm{hull}(\varphi)$, where $\mathrm{hull}(\varphi)$ is the smallest geodesically convex subset of $\partial \omega$ containing $\varphi$, that is, given any two points in $\mathrm{hull}(\varphi)$, there is a unique minimizing geodesic contained within $\mathrm{hull}(\varphi)$ that joins those two points.

Moreover, for any function $z$ defined on $\omega$, we denote $\overline{z}^\omega$ its average over $\omega$. Furthermore, for $1\leq p \leq \infty$, let $\|\cdot\|_{L^p(\omega)}$ be the norm in $L^p(\omega)$, and let $H^s(\omega)$ denote the Sobolev space of order $s\in\mathbb R$ whose classical norm and semi-norm are written $\|\cdot\|_{s,\omega}$ and $|\cdot|_{s,\omega}$, respectively. We recall from Ref.~\cite[Definition~1.3.2.1]{grisvard}, that for all $z\in H^s(\omega)$ with $\theta := s-\lfloor s \rfloor$, 
$$\|z\|^2_{s,\omega} :=\|z\|^2_{\lfloor s\rfloor,\omega} + |z|^2_{\theta,\omega}; \qquad |z|^2_{\theta,\omega} := \int_{\omega}\int_{\omega}\frac{\big(z(x)-z(y)\big)^2}{|x-y|^{k+2\theta}}\,\mathrm{d}x\,\mathrm{d}y.$$
We also write $L^2(\omega):=H^0(\omega)$ so that the norm in $L^2(\omega)$ will be written $\|\cdot\|_{0,\omega}$. And to deal with boundary conditions, for $z\in H^\frac{1}{2}(\varphi)$, we denote $$H^1_{z,\varphi}(\omega) := \left\{y\in H^1(\omega): \mathrm{tr}_\varphi(y) = z \right\},$$ where $\mathrm{tr}_\varphi(y)$ denotes the trace of $y$ on $\varphi\subset\partial \omega$. 
Moreover, we consider the Sobolev space $$H^{\frac{1}{2}}_{00}(\varphi):=\left\{z\in L^2(\varphi) : z^\star\in H^\frac{1}{2}(\partial\omega)\right\},$$ where $z^\star$ is the extension of $z$ by $0$ on $\partial \omega$, with its norm and semi-norm that we respectively denote $\|\cdot\|_{H^{1/2}_{00}(\varphi)}$ and $|\cdot|_{H^{1/2}_{00}(\varphi)}$, where
$$\|z\|_{H_{00}^{1/2}(\varphi)}^2 := \|z\|_{\frac{1}{2}, \varphi}^2 + |z|_{H_{00}^{1/2}(\varphi)}^2 \quad \text{ and }\quad |z|^2_{H_{00}^{1/2}(\varphi)} := \int_\varphi \int_{\partial \omega\setminus \varphi} \frac{z^2(s)}{|s-t|^{n}} \,\mathrm dt \,\mathrm ds.$$
We recall from Ref.~\cite[Lemma~1.3.2.6]{grisvard}, that there are two constants $C\geq c>0$ (independent from the measure of $\varphi$) such that for all $z\in H_{00}^\frac{1}{2}\left( \varphi\right)$, 
\begin{align*}
c|z|_{H_{00}^{1/2}(\varphi)}^2 \leq \int_{\varphi} \frac{z^2(s)}{\text{dist}\left(s, \partial\varphi\right)} \,\mathrm ds \leq C |z|_{H_{00}^{1/2}(\varphi)}^2,
\end{align*} 
and from Ref.~\cite[Eq.~(1,3,2,7)]{grisvard}, $\|z\|_{H_{00}^{1/2}(\varphi)} = \left\|z^\star\right\|_{\frac{1}{2},\partial\omega}$. In particular, we have $|z|^2_{\frac{1}{2},\varphi} + |z|^2_{H_{00}^{1/2}(\varphi)} = \left|z^\star\right|^2_{\frac{1}{2},\partial\omega}$. When $\varphi$ is not a connected set, then we define
$$H_{00}^\frac{1}{2}(\varphi) := \left\{ z\in L^2(\varphi): z|_{\varphi_c}\in H_{00}^\frac{1}{2}(\varphi_c), \forall \varphi_c\in \text{conn}(\varphi) \right\},$$
and we equip this space with the norm $$\|\cdot\|_{H_{00}^{1/2}(\varphi)} := \left( \displaystyle\sum_{\varphi_c\in \text{conn}(\varphi)} \|\cdot\|_{H_{00}^{1/2}(\varphi_c)}^2 \right)^\frac{1}{2}.$$ 
Next, let $H^{-\frac{1}{2}}_{00}(\varphi)$ be the dual space of $H^{\frac{1}{2}}_{00}(\varphi)$ equipped with the dual norm written $\|\cdot\|_{H^{-1/2}_{00}(\varphi)}$. 
Furthermore, for $m\in\mathbb{N}$, let $\mathbb{Q}_m(\omega)$ be the set of polynomials on $\omega$ of degree at most $m$ on each variable.
And if $\left\{\omega_\ell\right\}_{\ell=1}^L$ is a given partition of $\omega$ such that each $\omega_\ell$ is a flat element, that is, a straight line if $k=1$ or a flat square or triangle if $k=2$, let $\mathbb{Q}_{m,0}^\mathrm{pw}(\omega)$ be the set of continuous functions $q$ such that $q\vert_{\partial \omega} \equiv 0$, $q\vert_{\omega_\ell}\in \mathbb Q_m\left(\omega_\ell\right)$ for all $\ell=1,\ldots,L$. 

Finally, we will need the following assumptions on different domains, {where the symbol $\lesssim$ is used to mean any inequality which does not depend on the size of considered domains, but which can depend on their shape.}
\begin{definition} \label{as:isotropy} We say that $\omega$ is \emph{isotropic} if $$\mathrm{diam}\left(\omega\right) \lesssim \displaystyle\max_{\omega_c\in\text{conn}(\omega)}\big(\mathrm{diam}(\omega_c)\big),$$ and if each connected component $\omega_c$ of $\omega$ is isotropic, that is if diam$\left(\omega_c\right)^k \lesssim \left|\omega_c\right|$, for all $\omega_c\in\text{conn}(\omega)$. 
\end{definition}
\begin{definition} \label{as:pwsmoothshapereg} We say that $\omega$ is \emph{regular} if $\omega$ is piecewise shape regular and composed of flat elements, that is, if there is $L_\omega\in\mathbb{N}$ such that $\omega = \text{int}\left(\displaystyle\bigcup_{\ell=1}^{L_\omega} \overline{\omega_\ell}\right)$, where for all $\ell,m=1,\ldots,L_\omega$, $\omega_\ell \cap \omega_m = \emptyset$, $|\omega|\lesssim \left|\omega_\ell\right|$ and $\omega_\ell$ is flat (for instance, if $k=1$, it is a straight line).
\end{definition}
If $\omega$ is regular, for all $m\in \mathbb{N}$, we define $\Pi_{{m},\omega} : L^2(\omega) \rightarrow \mathbb{Q}_{m,0}^\text{pw}(\omega)$ as the extension of the Cl\'ement operator\cite{clement} developed in Ref.~\cite{bernardigirault} on $\omega$. 

\section{Defeaturing model problem} \label{sec:defeatpb}
Let us consider a given open Lipschitz domain $\Omega\subset \mathbb{R}^n$ that can potentially be complex: in this paper, we assume that it contains a feature $F$, that is, one geometrical detail of smaller scale. There exist three kinds of such geometrical features: a feature $F\subset \mathbb{R}^n$ is said to be
\begin{itemize}
	\item negative if $\left(\overline{F}\cap\overline{\Omega}\right) \subset \partial \Omega$;
	\item positive if $F\subset \Omega$;
	\item complex if it is composed of both negative and positive components. 
\end{itemize}
A positive feature corresponds to the addition of some material, a negative feature corresponds to a part where some material has been removed, and a feature is complex in the most general situation that corresponds to both the addition and the removal of some material. 

In this section, we suppose that $F$ is either negative or positive, and the considered defeaturing problem is stated, together with the notation for the problem that will be used throughout the article. The general situation in which $F$ is a complex feature will then be studied in Sec.~\ref{sec:complexfeat}. \\

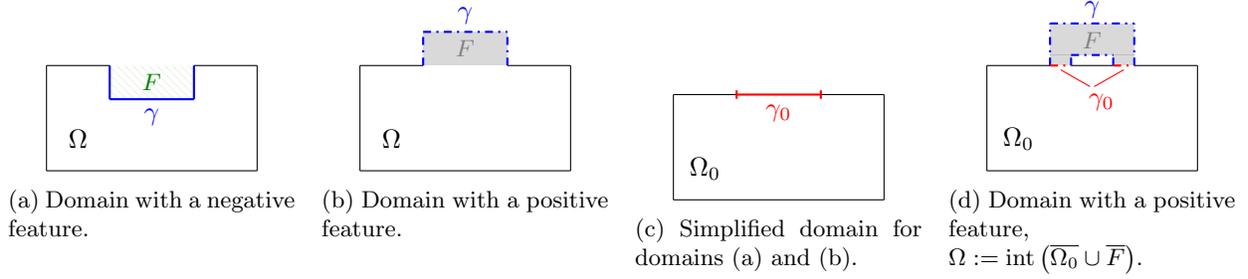
\begin{figure}
	\centering
	\begin{subfigure}[b]{0.23\textwidth}
		\begin{center}
			\begin{tikzpicture}[scale=2.8]
			\fill[pattern=north west lines, pattern color=green!50!black, opacity=0.3] (2.3,0.84) rectangle (2.7,1.00);
			\draw (2,0.5) -- (3,0.5) ;
			\draw (2,0.5) -- (2,1) ;
			\draw (3,0.5) -- (3,1) ;
			\draw (2,1) -- (2.3,1) ;
			\draw (2.7,1) -- (3,1) ;
			\draw[blue,thick] (2.3,1) -- (2.3,0.84) ;
			\draw[blue,thick] (2.7,1) -- (2.7,0.84) ;
			\draw[blue,thick] (2.3,0.84) -- (2.7,0.84) ;
			\draw (2.15,0.65) node{$\Omega$} ;
			\draw[green!50!black] (2.5,0.92)node{\small $F$};
			\draw[blue,thick] (2.5,0.84) node[below]{$\gamma$} ;
			\end{tikzpicture}
			\caption{Domain with a negative feature.\\}
			\label{fig:ex1a}
		\end{center}
	\end{subfigure}
	~
	\begin{subfigure}[b]{0.23\textwidth}
		\begin{center}
			\begin{tikzpicture}[scale=2.8]
			\fill[gray, fill, opacity=0.3] (0.3,1) rectangle (0.7,1.16);
			\draw (0,0.5) -- (1,0.5) ;
			\draw (0,0.5) -- (0,1) ;
			\draw (1,0.5) -- (1,1) ;
			\draw (0,1) -- (0.3,1) ;
			\draw (0.7,1) -- (1,1) ;
			\draw[blue,thick,dash dot] (0.3,1) -- (0.3,1.16) ;
			\draw[blue,thick,dash dot] (0.7,1) -- (0.7,1.16) ;
			\draw[blue,thick] (0.5,1.15) node[above]{$\gamma$} ;
			\draw[blue,thick,dash dot] (0.3,1.16) -- (0.7,1.16) ;
			\draw (0.15,0.65) node{$\Omega$} ;
			\draw[gray] (0.5,1.08)node{\small $F$};
			\end{tikzpicture}
			\caption{Domain with a positive feature.\\}
			\label{fig:ex1c}
		\end{center}
	\end{subfigure}
	~
	\begin{subfigure}[b]{0.23\textwidth}
		\begin{center}
			\begin{tikzpicture}[scale=2.8]
			\draw (0,0.5) -- (1,0.5) ;
			\draw (0,0.5) -- (0,1) ;
			\draw (1,0.5) -- (1,1) ;
			\draw (0,1) -- (0.3,1) ;
			\draw (0.7,1) -- (1,1) ;
			\draw[red,thick] (0.3,1) -- (0.7,1) ;
			\draw[red,thick] (0.3,0.98) -- (0.3, 1.02); 
			\draw[red,thick] (0.7,0.98) -- (0.7, 1.02); 
			\draw[red,thick] (0.5,1) node[below]{$\gamma_0$} ;
			\draw (0.15,0.65) node{$\Omega_0$} ;
			\end{tikzpicture}
			\caption{Simplified domain for domains (a) and (b).}
			\label{fig:ex1b}
		\end{center}
	\end{subfigure}
	~
	\begin{subfigure}[b]{0.23\textwidth}
		\begin{center}
			\begin{tikzpicture}[scale=2.8]
			\fill[gray, fill, opacity=0.3] (0.3,1.05) rectangle (0.7,1.2);
			\fill[gray, fill, opacity=0.3] (0.3,1) rectangle (0.4,1.05);
			\fill[gray, fill, opacity=0.3] (0.6,1) rectangle (0.7,1.05);
			\draw (0,0.5) -- (1,0.5) ;
			\draw (0,0.5) -- (0,1) ;
			\draw (1,0.5) -- (1,1) ;
			\draw (0,1) -- (0.3,1) ;
			\draw (0.7,1) -- (1,1) ;
			\draw (0.4,1) -- (0.6,1);
			\draw[blue,thick,dash dot] (0.3,1) -- (0.3, 1.2); 
			\draw[blue,thick,dash dot] (0.3,1.2) -- (0.7, 1.2); 
			\draw[blue,thick,dash dot] (0.7,1) -- (0.7, 1.2); 
			\draw[blue,thick,dash dot] (0.4,1) -- (0.4, 1.05); 
			\draw[blue,thick,dash dot] (0.4,1.05) -- (0.6, 1.05); 
			\draw[blue,thick,dash dot] (0.6,1) -- (0.6, 1.05); 
			\draw[blue,thick] (0.5,1.18) node[above]{$\gamma$} ;
			\draw (0.15,0.65) node{$\Omega_0$} ;
			\draw[gray] (0.5,1.12)node{\small $F$};
			\draw[red,thick,dash dot] (0.6,1) -- (0.7,1) ;
			\draw[red,thick,dash dot] (0.3,1) -- (0.4,1) ;
			\draw[red] (0.49,0.9) -- (0.35,0.98);
			\draw[red] (0.51,0.9) -- (0.65,0.98);
			\draw[red] (0.545,0.91) node[below]{$\gamma_0$} ;
			\end{tikzpicture}
			\caption{Domain with a positive feature,\\$\Omega := \text{int}\left(\overline{\Omega_0} \cup \overline{F}\right)$.}
			\label{fig:ex1h}
		\end{center}
	\end{subfigure}
	\caption{Different examples of geometries with a negative or a positive feature.} \label{fig:exfeat}
\end{figure}

So for now, let $F$ be an open Lipschitz feature which is either positive or negative, as in Fig.~\ref{fig:exfeat}. Then, let $\Omega_0\subset \mathbb R^n$ be the defeatured (or simplified) geometry, that is
\begin{itemize}
	\item if $F$ is negative, ${\Omega_0} := \mathrm{int}\left(\overline\Omega \cup \overline F\right)$;
	\item if $F$ is positive, $\Omega_0 := \Omega \setminus \overline F$,
\end{itemize}
and we also assume that $\Omega_0$ is an open Lipschitz domain. {In other words, if the feature $F$ is negative, then the exact domain $\Omega$ is embedded in the defeatured domain $\Omega_0$; if $F$ is positive instead, the exact domain $\Omega$ is the union of the defeatured domain $\Omega_0$ and the feature $F$.}

Let $\mathbf{n}$, $\mathbf{n}_0$ and $\mathbf{n}_F$ be the unitary outward normals of $\Omega$, $\Omega_0$ and $F$ respectively. Let $\partial \Omega = \overline{\Gamma_D} \cup \overline{\Gamma_N}$, ${\Gamma}_D \cap {\Gamma}_N = \emptyset$ {with $\Gamma_D \neq \emptyset$}, and we assume that ${\Gamma_D} \cap \partial F = \emptyset$. Finally, let $\gamma_0 := \partial F \setminus \overline{\Gamma_N}\subset \partial \Omega_0$ and $\gamma:=\partial F \setminus \overline{\gamma_0}\subset \partial \Omega$ so that $\partial F = \overline{\gamma} \cup \overline{\gamma_0}$ and $\gamma\cap\gamma_0 = \emptyset$ (see Fig.~\ref{fig:exfeat}). 
In particular, note that an internal feature $F$ is a negative feature for which $\gamma = \partial F$ and $\gamma_0 = \emptyset$. In the following, the defeaturing problem is stated, and the cases in which $F$ is either positive or negative are treated separately. \\

Let $h\in H^{\frac{3}{2}}(\Gamma_D)$, $g\in H^{\frac{1}{2}}(\Gamma_N)$ and $f\in L^2\left(\Omega\right)$. The considered problem is Poisson equation on the exact geometry $\Omega$: find $u\in H^1(\Omega)$, the weak solution of 
\begin{align}\label{eq:originalpb}
\begin{cases}
-\Delta u = f &\text{ in } \Omega \\
u = h &\text{ on } \Gamma_D \\
\displaystyle\frac{\partial u}{\partial \mathbf{n}} = g &\text{ on } \Gamma_N,\vspace{0.1cm}
\end{cases}
\end{align}
that is, $u\in H^1_{h,\Gamma_D}(\Omega)$ satisfies for all $v\in H^1_{0,\Gamma_D}(\Omega)$, 
\begin{equation} \label{eq:weakoriginalpb}
\int_\Omega \nabla u \cdot \nabla v \,\mathrm dx = \int_\Omega fv \,\mathrm dx + \int_{\Gamma_N} g v \,\mathrm ds.
\end{equation}

If feature $F$ is negative, consider any $L^2$-extension of $f\in L^2(\Omega)$ in $F$, that we still write $f\in L^2(\Omega_0)$ by abuse of notation. Note that such an extension is not needed for a positive feature. Then, instead of (\ref{eq:originalpb}), the following defeatured (or simplified) problem is solved: given $g_{0}\in H^{\frac{1}{2}}(\gamma_0)$, find $u_0\in H^1(\Omega_0)$, the weak solution of 
\begin{align} \label{eq:simplpb}
\begin{cases}
-\Delta u_0 = f &\text{ in } \Omega_0 \\
u_0 = h &\text{ on } \Gamma_D \\ \vspace{1mm}
\displaystyle\frac{\partial u_0}{\partial \mathbf{n}_0}  = g &\text{ on } \Gamma_N\setminus {\gamma}\\ 
\displaystyle\frac{\partial u_0}{\partial \mathbf{n}_0}  = g_{0} &\text{ on } \gamma_0,\vspace{0.1cm}
\end{cases}
\end{align}
that is, $u_0\in H^1_{h,\Gamma_D}(\Omega_0)$ satisfies for all $v\in H^1_{0,\Gamma_D}(\Omega_0)$, 
\begin{equation} \label{eq:weaksimplpb}
\int_{\Omega_0} \nabla u_0 \cdot \nabla v \,\mathrm dx = \int_{\Omega_0} fv \,\mathrm dx + \int_{\Gamma_N\setminus \gamma} g v \,\mathrm ds + \int_{\gamma_0} g_{0} v \,\mathrm ds.
\end{equation}
We are interested in controlling the energy norm of the defeaturing error, which we suitably define in what follows. \\

\textit{Negative feature case:} since $\Omega \subset \Omega_0$ in this case, we consider the restriction of $u_0$ to $\Omega$. Then we define the defeaturing error as $\big|u-u_0|_{\Omega}\big|_{1,\Omega}$. In this setting, we suppose that $\gamma$ is isotropic according to Definition \ref{as:isotropy}, where the diameter and the convex hull of $\gamma$ are considered in the manifold $\partial \Omega$ (see Sec.~\ref{sec:notation}). \\

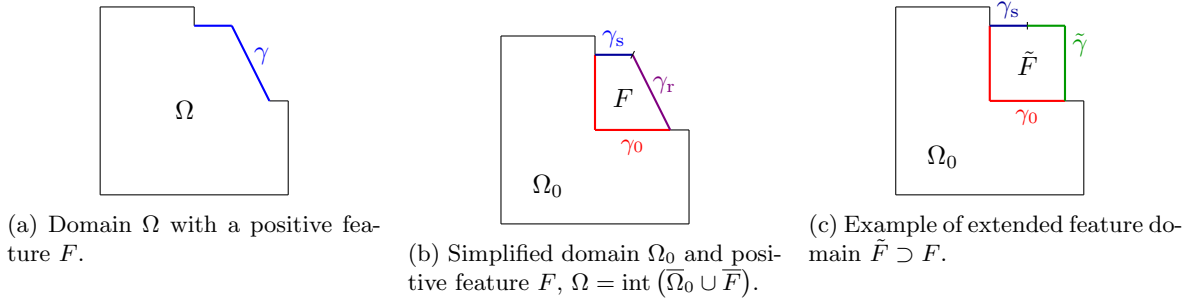
\begin{figure}[b]
	\centering
	\begin{subfigure}[b]{0.3\textwidth}
		\begin{center}
			\begin{tikzpicture}[scale=2.5]
			\draw (1,0) -- (2,0) ;
			\draw (1,0) -- (1,1) ;
			\draw (2,0) -- (2,0.5) ;
			\draw (1,1) -- (1.5,1) ;
			\draw (1.9,0.5) -- (2,0.5);
			\draw (1.5,0.9) -- (1.5,1);
			\draw[blue,thick] (1.5,0.9) -- (1.7,0.9);
			\draw[blue,thick] (1.7,0.9) -- (1.9,0.5);
			\draw[blue] (1.86,0.74) node{$\gamma$} ;
			\draw (1.45,0.45) node{$\Omega$} ;
			\end{tikzpicture}
			\caption{Domain $\Omega$ with a positive feature $F$.\\}
			\label{fig:badFfilletexgeom}
		\end{center}
	\end{subfigure}
	~
	\begin{subfigure}[b]{0.3\textwidth}
		\begin{center}
			\begin{tikzpicture}[scale=2.5]
			\draw (1,0) -- (2,0) ;
			\draw (1,0) -- (1,1) ;
			\draw (2,0) -- (2,0.5) ;
			\draw (1,1) -- (1.5,1) ;
			\draw[red,thick] (1.5,0.5) -- (1.9,0.5);
			\draw (1.9,0.5) -- (2,0.5);
			\draw (1.5,0.9) -- (1.5,1);
			\draw[red,thick] (1.5,0.5) -- (1.5,0.9);
			\draw[red] (1.7,0.5) node[below]{\small$\gamma_0$} ;
			\draw[blue!60!black!,thick] (1.5,0.9) -- (1.7,0.9);
			\draw[blue!50!red,thick] (1.7,0.9) -- (1.9,0.5);
			\draw[blue!50!red] (1.87,0.74) node{$\gamma_\setminussign$} ;
			\draw[blue] (1.6,0.89) node[above]{$\gamma_\intersign$} ;
			\draw (1.25,0.2) node{$\Omega_0$} ;
			\draw (1.65,0.67) node{$F$} ;
			\draw (1.71,0.92)--(1.69,0.88);
			\end{tikzpicture}
			\caption{Simplified domain $\Omega_0$ and positive feature $F$, $\Omega = \mathrm{int}\left(\overline\Omega_0 \cup \overline F\right)$.}
			\label{fig:badFfillet}
		\end{center}
	\end{subfigure}
	~
	\begin{subfigure}[b]{0.3\textwidth}
		\begin{center}
			\begin{tikzpicture}[scale=2.5]
			\draw (1,0) -- (2,0) ;
			\draw (1,0) -- (1,1) ;
			\draw (2,0) -- (2,0.5) ;
			\draw (1,1) -- (1.5,1) ;
			\draw[red,thick] (1.5,0.5) -- (1.9,0.5);
			\draw (1.9,0.5) -- (2,0.5);
			\draw (1.5,0.9) -- (1.5,1);
			\draw[red,thick] (1.5,0.5) -- (1.5,0.9);
			\draw[red] (1.7,0.5) node[below]{\small$\gamma_0$} ;
			\draw (1.25,0.2) node{$\Omega_0$} ;
			\draw[blue!60!black!,thick] (1.7,0.9) -- (1.5,0.9);
			\draw[blue!60!black!] (1.6,0.89) node[above]{$\gamma_\intersign$} ;
			\draw (1.7,0.92)--(1.7,0.88);
			\draw[green!60!black!,thick] (1.7,0.9) -- (1.9,0.9);
			\draw[green!60!black!,thick] (1.9,0.5) -- (1.9,0.9);
			\draw (1.7,0.7) node{$\tilde{F}$};
			\draw[green!60!black!] (1.89,0.8) node[above,right]{$\tilde{\gamma}$};
			\end{tikzpicture}
			\caption{Example of extended feature domain $\tilde{F}\supset F$.\\}
			\label{fig:badFfilletext}
		\end{center}
	\end{subfigure}
	\caption{Example of geometry with a positive feature.} \label{fig:badfilletex}
\end{figure}

\textit{Positive feature case:} the solution $u_0$ {is only defined in the defeatured geometry $\Omega_0$ which does not contain the feature $F$, since $F \subset \Omega$ but $F \not\subset \Omega_0$. That is, the solution $u_0$ is not defined everywhere on the exact geometry $\Omega= \text{int}\left(\overline{\Omega_0}\cup\overline{F}\right)$.} Therefore, to define the defeaturing error, {and since $\Omega$ is the union of $\Omega_0$ and $F$, one needs to solve a problem to extend $u_0$ to $F$}. The most natural extension would be the solution of
\begin{align} \label{eq:featurepbonF}
\begin{cases}
-\Delta \tilde{u}_0 = f &\text{ in } F \\
\tilde{u}_0 = u_0 & \text{ on } \gamma_{0} \\
\displaystyle\frac{\partial \tilde u_0}{\partial \mathbf{n}_F} = g & \text{ on } \gamma.\vspace{0.1cm}
\end{cases}
\end{align}

\noindent However, $F$ may be complex or even non-smooth (see the examples in Sec.~\ref{sec:nonlipschitz}), thus finding or computing the solution of (\ref{eq:featurepbonF}) may be cumbersome.
Therefore, let $\tilde{F}\subset \mathbb{R}^n$ be a Lipschitz domain that contains $F$ and such that $\gamma_{0} \subset \left(\partial \tilde{F} \cap \partial F\right)$, that is, $\tilde F$ is a suitable (simple) domain extension of $F$ such as the bounding box of $F$ for example. Note that it is possible to have $\tilde F \cap \Omega_0 \neq \emptyset$, but we also assume that $\tilde F \setminus \overline{F}$ is Lipschitz. Thus if we consider any $L^2$-extension of $f$ in $\tilde F$, that we still write $f$ by abuse of notation, then we can solve an extension problem in $\tilde F$ instead of $F$. {Note that one can look at $\tilde F$ as the defeatured geometry of the positive feature $F$, that is, as a geometry simplified from the exact geometry $F$, for which $\tilde F\setminus \overline F$ is a negative feature.}

This is illustrated in Fig.~\ref{fig:badfilletex}: instead of solving the extension problem (\ref{eq:featurepbonF}) in $F$, we can choose to solve an extension problem in $\tilde{F}$, a simpler domain which shares $\gamma_0$ as a boundary. Let $\tilde{\mathbf{n}}$ be the unitary outward normal of $\tilde F$, let $\tilde \gamma := \partial \tilde F \setminus \partial F$, and let $\gamma$ be decomposed as $\gamma = \text{int}(\overline{\gamma_\intersign}\cup\overline{\gamma_\setminussign})$, where $\gamma_\intersign$ and $\gamma_\setminussign$ are open, $\gamma_\intersign$ is the part of $\gamma$ that is shared with $\partial \tilde F$ while $\gamma_\setminussign$ is the remaining part of $\gamma$, that is, the part that does not belong to $\partial \tilde F$. 
Note that $\gamma_0$ and $\tilde \gamma$ are ``simple" boundaries since they are the boundaries of the chosen simplified geometry $\Omega_0$ and of the chosen extended feature domain $\tilde F$, respectively. 

Therefore, let us consider the following extension of the solution $u_0$ of (\ref{eq:simplpb}) on $\tilde F$: given $\tilde g\in H^{\frac{1}{2}}(\tilde \gamma)$, find $\tilde{u}_0\in H^1\left(\tilde{F}\right)$, the weak solution of 
\begin{align} \label{eq:featurepb}
\begin{cases}
-\Delta \tilde{u}_0 = f &\text{ in } \tilde{F} \\ \vspace{1mm}
\tilde{u}_0 = u_0 & \text{ on } \gamma_{0} \\ \vspace{1mm}
\displaystyle\frac{\partial \tilde u_0}{\partial \tilde{\mathbf{n}}}  = \tilde g & \text{ on } \tilde \gamma \\
\displaystyle\frac{\partial \tilde u_0}{\partial \tilde{\mathbf{n}}}  = g & \text{ on } \gamma_\intersign,
\end{cases}
\end{align}
that is, $\tilde{u}_0\in H^1_{u_0,\gamma_{0}}\left(\tilde F\right)$ satisfies for all $v\in H^1_{0,\gamma_{0}}\left(\tilde{F}\right)$, 
\begin{equation*} 
\int_{\tilde{F}} \nabla \tilde{u}_0 \cdot \nabla v \,\mathrm dx = \int_{\tilde{F}} fv \,\mathrm dx + \int_{\tilde \gamma} \tilde{g} v \,\mathrm ds + \int_{\gamma_\intersign} {g} v \,\mathrm ds.
\end{equation*}
Let $u_\mathrm d\in H^1_{h,\Gamma_D}\left(\Omega\right)$ be the extended defeatured solution, that is,
\begin{equation*}\label{eq:defud}
u_\mathrm d = u_0 \text{ in } \Omega_0 \quad \text{ and } \quad u_\mathrm d = \tilde{u}_0\vert_{F_\mathrm p} \text{ in } F_\mathrm p.
\end{equation*}
Then we define the defeaturing error as $\left|u-u_\mathrm d\right|_{1,\Omega}$. 

In this setting, we suppose that $\gamma_{0}$ and $\gamma_\setminussign$ are isotropic according to Definition \ref{as:isotropy}, where the diameter and the convex hull $\gamma_0$ are considered in the manifold $\partial \Omega_0$, and the diameter and the convex hull of $\gamma_\setminussign$ are considered in the manifold $\partial F$. 

\begin{remark}
	The problem is studied in the case in which all domains are Lipschitz, and under the isotropy conditions stated above. A finer analysis could be performed to take into account more general geometries, such as the non-Lipschitz fillet of Sec.~\ref{sec:nonlipschitz}, but this goes beyond the scope of this paper. Moreover, when used, the regularity condition defined in Definition \ref{as:pwsmoothshapereg} is taken for the sake of simplicity, but it can be relaxed by considering $\omega$ piecewise smooth and shape regular instead.
\end{remark}

Note that the boundaries $\gamma$, $\gamma_{0}$ and $\gamma_\setminussign$ can be non-connected sub-manifolds (see Fig.~\ref{fig:ex1h}). In the remaining part of this article, the symbol $\lesssim$ will be used to mean any inequality which does not depend on the size of the feature $F$ nor on the size of the positive extension $\tilde F$, but which can depend on their shape. Moreover, we will write $A \simeq B$ whenever $A\lesssim B$ and $B\lesssim A$. 

\section{Negative feature \textit{a posteriori }defeaturing error estimator} \label{sec:negativeest}
In this section, an optimal \textit{a posteriori} defeaturing error estimator is derived in the simplest setting of a negative feature. We show that the derived estimator is an upper bound and a lower bound (up to oscillations) of the energy norm of the defeaturing error. The key issue in the subsequent analysis is to track the dependence of all constants from the size of the feature. Although it would be possible to present the equivalent analysis for a positive feature, we have decided to omit it and to let the positive feature case be a consequence of the more general case of a complex feature, whose dedicated analysis is presented in Sec.~\ref{sec:complexfeat}. \\

Let $F$ be a negative feature of $\Omega$, and suppose that $\gamma$ is isotropic according to Definition \ref{as:isotropy}. Then, let 
\begin{align}
d_\gamma := g+\frac{\partial u_0}{\partial \mathbf n_F} \text{ on } \gamma \label{eq:dgamma}
\end{align}
{be the error term on the Neumann data on $\gamma$}, and we define the defeaturing error estimator as 
\begin{equation}\label{eq:negestimator}
\mathcal E_{\mathrm n}(u_0) := \left( |\gamma|^{\frac{1}{n-1}} \left\| d_\gamma - \overline{d_\gamma}^\gamma \right\|^2_{0,\gamma} + c_{\gamma}^2 |\gamma|^\frac{n}{n-1}\left| \overline{d_\gamma}^\gamma \right|^2 \, \right)^\frac{1}{2},
\end{equation}
where, if we define $\eta\in\mathbb{R}$ as the unique solution of $\eta = -\log(\eta)$, 
\begin{align} \label{eq:negconstant}
c_{\gamma} := \begin{cases}
\max\big(\hspace{-0.05cm}\left|\log\left(|\gamma|\right)\right|, \eta \big)^\frac{1}{2} & \text{ if } n = 2 \\
1 & \text{ if } n = 3. 
\end{cases}
\end{align}

We first show that the quantity $\mathcal{E}_{\mathrm n}(u_0)$ is a reliable estimator for the defeaturing error, i.e., it is an upper bound for the defeaturing error (see Theorem \ref{thm:upperbound}). Then, assuming that $\gamma$ is also regular according to Definition \ref{as:pwsmoothshapereg}, and under mild assumptions for the two-dimensional case, we show that it is also efficient (up to oscillations), i.e., it is a lower bound for the defeaturing error up to oscillations (see Theorem \ref{thm:lowerbound}). This means that the whole information on the error introduced by defeaturing a negative feature, in energy norm, is contained in the boundary $\gamma$, and can be accounted by suitably evaluating the error made on the normal derivative of the solution. 

\begin{remark} \label{rmk:compatcondneg}
	Consider the simplified problem (\ref{eq:simplpb}) restricted to $F$ with the natural Neumann boundary condition on $\gamma$, that is, $u_0\vert_{F}\in H^1(F)$ satisfies
	\begin{align*}
	\begin{cases}
	-\Delta \left(u_0\vert_{F}\right) = f &\text{ in } F \\ \vspace{1mm}
	\displaystyle\frac{\partial \left(u_0\vert_{F}\right)}{\partial \mathbf{n}_0}  = g_0 &\text{ on } \gamma_0 \\ 
	\displaystyle\frac{\partial \left(u_0\vert_{F}\right)}{\partial \mathbf{n}_F}  = \displaystyle\frac{\partial u_0}{\partial \mathbf{n}_F} &\text{ on } \gamma.
	\end{cases}
	\end{align*}
	By abuse of notation, we omit the explicit restriction of $u_0$ to $F$.  Then if we multiply the restricted problem by the constant function $1$ and integrate by parts, we obtain
	$$\int_F f\,\mathrm dx + \int_{\gamma_0} g_0\,\mathrm ds + \int_\gamma \frac{\partial u_0}{\partial \mathbf n_F} = 0.$$
	Consequently, 
	\begin{equation*} 
	\overline{d_\gamma}^\gamma = \overline{\left(g+\frac{\partial u_0}{\partial \mathbf n_F}\right)}^\gamma = \frac{1}{|\gamma|}\left( \int_\gamma g\,\mathrm ds - \int_{\gamma_0} g_0\,\mathrm ds - \int_F f\,\mathrm dx\right).
	\end{equation*}
	Therefore, the second term of the estimator $\mathcal{E}_\mathrm n(u_0)$ in (\ref{eq:negestimator}) only depends on the defeatured problem data, and more precisely on the choice of $g_0$ that one considers on $\gamma_0$, and on the choice of the extension of $f$ that one considers in the feature $F$. 
	As a consequence, if the second term of the estimator (\ref{eq:negestimator}) dominates, this means that the defeatured problem data should be better chosen. 
	Moreover, under the following reasonable flux conservation assumption
	\begin{equation} \label{eq:compatcondneg}
	\int_\gamma g\,\mathrm ds - \int_{\gamma_0} g_0\,\mathrm ds - \int_F f\,\mathrm dx = 0,
	\end{equation}
	the defeaturing error estimator (\ref{eq:negestimator}) rewrites
	$$\mathcal E_{\mathrm n}(u_0) = |\gamma|^{\frac{1}{2(n-1)}}\left\| d_\gamma \right\|_{0,\gamma}.$$
	Note that condition (\ref{eq:compatcondneg}) is easily met if the Neumann boundary condition $g$ and the source function $f$ are zero in the vicinity of the feature. 
\end{remark}

\begin{remark} \label{rmk:estnegtilde}
	Since $(4c_\gamma^2 - 1) > 0$ for all $\gamma$, remark that by Cauchy-Schwarz inequality, 
	\begin{align*}
	\mathcal{E}_\mathrm n(u_0) &\lesssim |\gamma|^\frac{1}{2(n-1)} \left[\| d_\gamma - \overline{d_\gamma}^\gamma \|_{0,\gamma}^2 + 4c_\gamma^2 |\gamma| \left(\overline{d_\gamma}^\gamma\right)^2 \right]^\frac{1}{2} \\
	&= |\gamma|^\frac{1}{2(n-1)} \left[ \|d_\gamma\|_{0,\gamma}^2 + \left( 4c_\gamma^2 -1 \right)|\gamma| \left(\overline{d_\gamma}^\gamma\right)^2 \right]^\frac{1}{2} \\
	& \lesssim \,c_\gamma |\gamma|^\frac{1}{2(n-1)} \left\| d_\gamma \right\|_{0,\gamma} =: \tilde{\mathcal{E}}_\mathrm n(u_0). 
	\end{align*}
	One could be tempted to use the simpler indicator $\tilde{\mathcal{E}}_\mathrm n(u_0)$, but when $n=2$ and under the flux conservation condition (\ref{eq:compatcondneg}), $\tilde{\mathcal{E}}_\mathrm n(u_0)$ is sub-optimal since in this case, $\tilde{\mathcal{E}}_\mathrm n(u_0) = c_\gamma \mathcal{E}_\mathrm n(u_0)$. Indeed, no lower bound can be provided for $\tilde{\mathcal{E}}_\mathrm n(u_0)$. 
\end{remark}

\subsection{Upper bound}
In this section, we prove that the error indicator defined in (\ref{eq:negestimator}) is reliable, that is, it is an upper bound for the defeaturing error.
\begin{theorem}\label{thm:upperbound}
	Let $u$ and $u_0$ be the weak solutions of problems (\ref{eq:originalpb}) and (\ref{eq:simplpb}), respectively. 
	If $\gamma$ is isotropic according to Definition \ref{as:isotropy}, then the defeaturing error in energy norm is bounded in terms of the estimator $\mathcal{E}_{\mathrm n}(u_0)$ introduced in (\ref{eq:negestimator}) as follows:
	\begin{equation*}
	\big|u-u_0|_{\Omega}\big|_{1,\Omega} \lesssim \mathcal{E}_{\mathrm n}(u_0).
	\end{equation*}
\end{theorem}

\begin{proof}
	Let us first consider the simplified problem (\ref{eq:simplpb}) restricted to $\Omega$ with the natural Neumann boundary condition on $\gamma$, that is, since $\mathbf{n}_F = -\mathbf{n}$ on $\gamma$, the restriction $u_0\vert_{\Omega}\in H^1_{h,\Gamma_D}(\Omega)$ is the weak solution of 
	\begin{align} \label{eq:simplpbrestricted}
	\begin{cases}
	-\Delta \left(u_0\vert_{\Omega}\right) = f &\text{ in } \Omega \\
	u_0\vert_{\Omega} = h &\text{ on } \Gamma_D \\ \vspace{1mm}
	\displaystyle\frac{\partial \left(u_0\vert_{\Omega}\right)}{\partial \mathbf{n}}  = g &\text{ on } \Gamma_N\setminus {\gamma} \\ 
	\displaystyle\frac{\partial \left(u_0\vert_{\Omega}\right)}{\partial \mathbf{n}}  = -\displaystyle\frac{\partial u_0}{\partial \mathbf{n}_F} &\text{ on } \gamma.
	\end{cases}
	\end{align}
	By abuse of notation, we omit the explicit restriction of $u_0$ to $\Omega$. 
	Then, for all $v\in H^1_{0,\Gamma_D}(\Omega)$, 
	\begin{equation}\label{eq:weaksimplpbrestricted}
	\int_{\Omega} \nabla u_0 \cdot \nabla v \,\mathrm dx = \int_{\Omega} fv \,\mathrm dx + \int_{\Gamma_N\setminus \gamma} g v \,\mathrm ds - \int_{\gamma} \frac{\partial u_0}{\partial \mathbf{n}_F} v \,\mathrm ds.
	\end{equation}
	Let $e:= u-u_0 \in H^1_{0,\Gamma_D}(\Omega)$. Then for all $v\in H^1_{0,\Gamma_D}(\Omega)$, it holds from equations (\ref{eq:weakoriginalpb}) and (\ref{eq:weaksimplpbrestricted}) that
	\begin{align}
	\int_{\Omega} \nabla e\cdot \nabla v \,\mathrm{d}x 
	&= 
	\int_{\Gamma_N} g v \,\mathrm ds 
	- \int_{\Gamma_N\setminus \gamma} g v \,\mathrm ds + \int_{\gamma} \frac{\partial u_0}{\partial \mathbf{n}_F} v \,\mathrm ds \nonumber\\
	&= \int_{\gamma} \left(g + \frac{\partial u_0}{\partial \mathbf{n}_F}\right)v\,\mathrm ds = \int_{\gamma} d_\gamma v \,\mathrm ds. \label{eq:negrewriteerror}
	\end{align}
	Now, if we take $v = e \in H^1_{0,\Gamma_D}(\Omega)$ in (\ref{eq:negrewriteerror}), 
	then
	\begin{equation} \label{eq:mainnegproof}
	|e|_{1,\Omega}^2 = \int_{\gamma} d_\gamma e \,\mathrm ds
	= \int_{\gamma} \left(d_\gamma - \overline{d_\gamma}^\gamma\right) e \,\mathrm ds + \overline{d_\gamma}^\gamma \int_{\gamma} e \,\mathrm ds.
	\end{equation}
	Let us first estimate the first term of (\ref{eq:mainnegproof}).
	Thanks to Poincar\'e inequality of Appendix \ref{lemma:poincarebd} and a trace inequality, 
	\begin{align}
	\int_{\gamma} \left(d_\gamma-\overline{d_\gamma}^\gamma\right) e \,\mathrm ds
	&= \int_{\gamma} \left(d_\gamma-\overline{d_\gamma}^\gamma\right) \left(e-\overline{e}^\gamma\right) \,\mathrm ds \nonumber \leq \left\|d_\gamma-\overline{d_\gamma}^\gamma\right\|_{0,\gamma} \left\|e-\overline e^\gamma\right\|_{0,\gamma} \nonumber \\
	& \lesssim \left\|d_\gamma-\overline{d_\gamma}^\gamma\right\|_{0,\gamma} |\gamma|^{\frac{1}{2(n-1)}} |e|_{\frac{1}{2},\gamma} \leq |\gamma|^{\frac{1}{2(n-1)}} \left\|d_\gamma-\overline{d_\gamma}^\gamma \right\|_{0,\gamma} |e|_{\frac{1}{2},\partial \Omega} \nonumber\\
	& \lesssim |\gamma|^{\frac{1}{2(n-1)}} \left\|d_\gamma-\overline{d_\gamma}^\gamma\right\|_{0,\gamma} |e|_{1,\Omega}. \label{eq:12}
	\end{align}
	Moreover, the second term of (\ref{eq:mainnegproof}) can be estimated thanks to Appendix \ref{lemma:savare} and a trace inequality, that is,
	\begin{align} 
	\overline{d_\gamma}^\gamma \int_{\gamma} e \,\mathrm ds \leq \left| \overline{d_\gamma}^\gamma \right| |\gamma|^\frac{1}{2}\|e\|_{0,\gamma} &\lesssim \left| \overline{d_\gamma}^\gamma \right| c_{\gamma} |\gamma|^{\frac{1}{2(n-1)}+\frac{1}{2}} \|e\|_{\frac{1}{2},\partial \Omega} \nonumber \\
	&\lesssim c_{\gamma} |\gamma|^{\frac{n}{2(n-1)}} \left| \overline{d_\gamma}^\gamma \right| |e|_{1,\Omega}. \label{eq:eq19}
	\end{align}
	Therefore, combining (\ref{eq:mainnegproof}), (\ref{eq:12}) and (\ref{eq:eq19}), and simplifying on both sides, we obtain the desired result.
\end{proof}

\subsection{Lower bound}
In this section, we prove that the error indicator defined in (\ref{eq:negestimator}) is efficient, that is, it is a lower bound for the defeaturing error, up to oscillations. In the case $n=2$, the flux conservation assumption (\ref{eq:compatcondneg}) is also required. 

\begin{theorem}\label{thm:lowerbound}
	Let $u$ and $u_0$ 
	be as in Theorem \ref{thm:upperbound}, and assume that $\gamma$ is isotropic and regular according to Definitions \ref{as:isotropy} and \ref{as:pwsmoothshapereg}. 
	Suppose that either $n=3$, or $n=2$ and the flux conservation condition (\ref{eq:compatcondneg}) is satisfied. Then the defeaturing error, in energy norm, bounds up to oscillations the estimator $\mathcal{E}_{\mathrm n}(u_0)$ introduced in (\ref{eq:negestimator}), that is
	\begin{equation*}
	\mathcal{E}_{\mathrm n}(u_0) \lesssim \big|u-u_0|_{\Omega}\big|_{1,\Omega} + \mathrm{osc}_{\mathrm n}(u_0),
	\end{equation*}
	where
	\begin{align} 
	\mathrm{osc}_{\mathrm n}(u_0) &:= \left| \gamma\right|^\frac{1}{2(n-1)} \left\| d_\gamma - \Pi_{{m}}\left(d_\gamma\right)\right\|_{0,\gamma} \label{eq:osc}
	\end{align}
	for any $m\in \mathbb{N}$, with $\Pi_m:=\Pi_{{m},\gamma}$ being the extension of the Cl\'ement operator defined in Sec.~\ref{sec:notation}. 
\end{theorem}

\begin{proof}
	To simplify the notation, we omit to explicitly write the restriction of $u_0$ to $\Omega$ when it would be necessary, since the context makes it clear. As before, let $e:=u-u_0\in H^1_{0,\Gamma_D}(\Omega)$. From equation (\ref{eq:negrewriteerror}), for all $v\in H_{0,\Gamma_D}^1(\Omega)$, 
	\begin{equation} \label{eq:fornegdual}
	\int_{\gamma} d_\gamma v \,\mathrm ds = \int_{\Omega} \nabla e \cdot \nabla v \,\mathrm dx \leq |e|_{1,\Omega} |v|_{1,\Omega}.	
	\end{equation}
	Now, for all $w\in H^\frac{1}{2}_{00}(\gamma)$, let $u_w\in H^1_{0,\partial\Omega\setminus\gamma}(\Omega)\subset H^1_{0,\Gamma_D}(\Omega)$ be the unique weak solution of
	\begin{align*}
	\begin{cases}
	-\Delta u_w = 0 & \text{in } \Omega \\
	u_w = w^\star & \text{on } \partial \Omega,
	\end{cases}
	\end{align*}
	where $w^\star$ is the extension of $w$ by $0$. Then $\left|u_w\right|_{1,\Omega} \lesssim \left\|w^\star\right\|_{\frac{1}{2},\partial \Omega} = \left\|w\right\|_{H_{00}^{1/2}(\gamma)}$ by continuity of the solution on the data. Therefore, using (\ref{eq:fornegdual}),
	\begin{align}
	\left\| d_\gamma\right\|_{H^{-1/2}_{00}(\gamma)} &= \sup_{\substack{w\in H_{00}^{1/2}(\gamma)\\ w\neq 0}} \frac{\displaystyle\int_{\gamma} d_\gamma w \,\mathrm ds}{\|w\|_{H_{00}^{1/2}(\gamma)}} \lesssim \sup_{\substack{w\in H_{00}^{1/2}(\gamma)\\ w\neq 0}} \frac{\displaystyle\int_{\gamma} d_\gamma u_w \,\mathrm ds}{|u_w|_{1,\Omega}} \nonumber \\
	&\leq \sup_{\substack{v\in H^1_{0,\Gamma_D}(\Omega)\\ v\neq 0}} \frac{\displaystyle\int_{\gamma} d_\gamma v \,\mathrm ds}{|v|_{1,\Omega}}\leq \sup_{\substack{v\in H^1_{0,\Gamma_D}(\Omega)\\ v\neq 0}} \frac{|e|_{1,\Omega}|v|_{1,\Omega}}{|v|_{1,\Omega}} = \,|e|_{1,\Omega}. \label{eq:dualerrneg}
	\end{align}
	Moreover, using Remark \ref{rmk:estnegtilde} if $n=3$, or Remark \ref{rmk:compatcondneg} if $n=2$ and if the flux conservation condition (\ref{eq:compatcondneg}) is satisfied, then 
	$$\mathcal E_{\mathrm n}(u_0) \lesssim |\gamma|^\frac{1}{2(n-1)}\|d_\gamma\|_{0,\gamma}.$$
	Therefore, using the triangle inequality and applying the inverse inequality of Appendix \ref{lemma:inverseineq}, we get
	\begin{align} 
	\mathcal E_{\mathrm n}(u_0)
	&\lesssim |\gamma|^\frac{1}{2(n-1)} \left( \left\| \Pi_{{m}}\left(d_\gamma\right)\right\|_{0,\gamma} + \left\| d_\gamma - \Pi_{{m}}\left(d_\gamma\right)\right\|_{0,\gamma} \right) \nonumber \\
	&\lesssim \left\| \Pi_{{m}}\left(d_\gamma\right)\right\|_{H^{-1/2}_{00}(\gamma)} + |\gamma|^\frac{1}{2(n-1)} \left\| d_\gamma - \Pi_{{m}}\left(d_\gamma\right)\right\|_{0,\gamma}. \label{eq:lowerboundneg1}
	\end{align}
	Finally, using another time the triangle inequality, Appendix \ref{lemma:inverseineq2} and (\ref{eq:dualerrneg}), we obtain
	\begin{align} 
	\left\| \Pi_{{m}}\left(d_\gamma\right)\right\|_{H^{-1/2}_{00}(\gamma)} &\leq \left\| d_\gamma\right\|_{H^{-1/2}_{00}(\gamma)} + \left\| \Pi_{{m}}\left(d_\gamma\right) - d_\gamma\right\|_{H^{-1/2}_{00}(\gamma)} \nonumber \\
	&\lesssim |e|_{1,\Omega} + |\gamma|^\frac{1}{2(n-1)} \left\| d_\gamma - \Pi_{{m}}\left(d_\gamma\right)\right\|_{0,\gamma}. \label{eq:lowerboundneg2}
	\end{align}
	Consequently, combining (\ref{eq:lowerboundneg1}) and (\ref{eq:lowerboundneg2}), and recalling the definition (\ref{eq:osc}) of the oscillations, then
	$$\mathcal E_{\mathrm n}(u_0) \lesssim |e|_{1,\Omega} + \text{osc}_{\mathrm n}(u_0). $$
\end{proof}

\begin{remark} \label{rmk:oscsmall}
	In some sense, the oscillations pollute the lower bound in Theorem \ref{thm:lowerbound}. It is therefore important to make sure that the oscillations are asymptotically smaller than the defeaturing error, with respect to the size of the feature. While there is a strong numerical evidence of it (see Sec.~\ref{sec:numexp}), an \textit{a priori} error analysis of the defeaturing problem is needed in order to obtain a rigorous proof, but this goes beyond the scope of this paper. 
	However, we are expecting the term $\left\|d_\gamma\right\|_{0,\gamma}$ to depend on the measure of $\gamma$. 
	When the data is regular, so is $u_0$, and it is then always possible to choose 
	$m$ large enough so that the asymptotic behavior of the oscillations is $\mathcal{O}\left(|\gamma|^{m+\frac{1}{2(n-1)}}\right)$. Therefore, upon a wise choice of $ m$, the oscillations converge faster than the defeaturing error with respect to the measure of $\gamma$. 
\end{remark}

\section{Defeaturing a geometry with a complex feature} \label{sec:complexfeat}
In this section, instead of discussing only a defeaturing error estimator for a geometry containing a positive feature, we directly generalize the previous study to a geometry containing a complex feature, that is, a feature containing both negative and positive components. More precisely, we first generalize the defeaturing problem of Sec.~\ref{sec:defeatpb} to this context, and then we derive a corresponding optimal \textit{a posteriori} defeaturing error estimator. Building upon the study of Sec.~\ref{sec:negativeest}, we show that the derived estimator is an upper bound and a lower bound (up to oscillations) of the energy norm of the defeaturing error, by accurately tracking the dependence of all constants from the size of the feature.

\subsection{Defeaturing model problem for a complex feature} \label{sec:gendefeatpb}
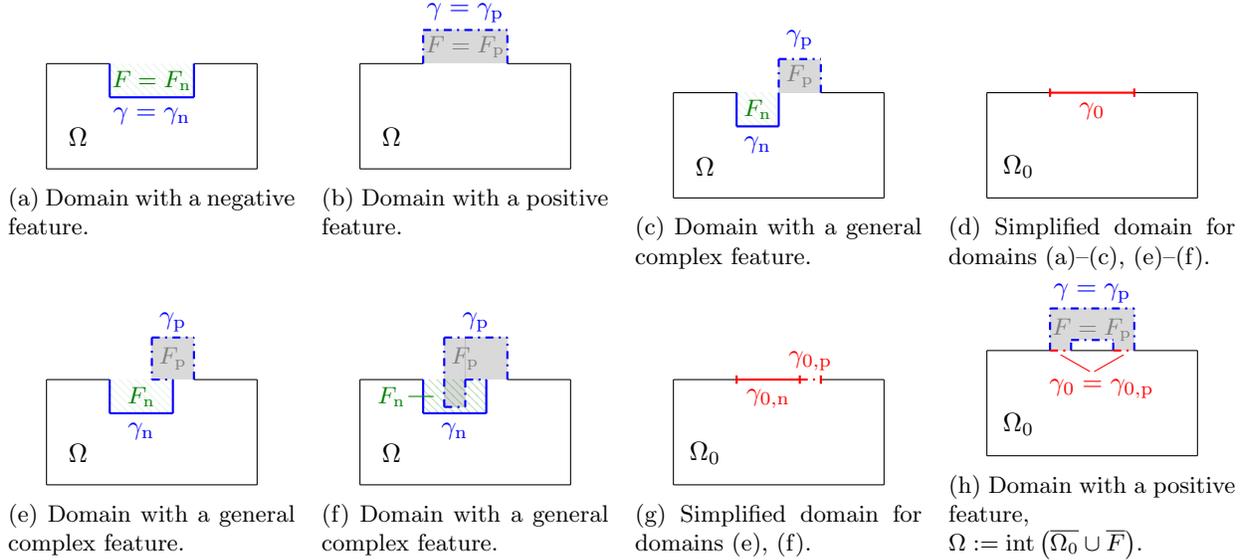
\begin{figure}
	\centering
	\begin{subfigure}[b]{0.23\textwidth}
		\begin{center}
			\begin{tikzpicture}[scale=2.8]
			\fill[pattern=north west lines, pattern color=green!50!black, opacity=0.3] (2.3,0.84) rectangle (2.7,1.00);
			\draw (2,0.5) -- (3,0.5) ;
			\draw (2,0.5) -- (2,1) ;
			\draw (3,0.5) -- (3,1) ;
			\draw (2,1) -- (2.3,1) ;
			\draw (2.7,1) -- (3,1) ;
			\draw[blue,thick] (2.3,1) -- (2.3,0.84) ;
			\draw[blue,thick] (2.7,1) -- (2.7,0.84) ;
			\draw[blue,thick] (2.3,0.84) -- (2.7,0.84) ;
			\draw (2.15,0.65) node{$\Omega$} ;
			\draw[green!50!black] (2.5,0.92)node{\small $F=F_\mathrm n$};
			\draw[blue,thick] (2.5,0.84) node[below]{$\gamma=\gamma_\mathrm n$} ;
			\end{tikzpicture}
			\caption{Domain with a negative feature.\\}
			\label{fig:ex1acomplex}
		\end{center}
	\end{subfigure}
	~
	\begin{subfigure}[b]{0.23\textwidth}
		\begin{center}
			\begin{tikzpicture}[scale=2.8]
			\fill[gray, fill, opacity=0.3] (0.3,1) rectangle (0.7,1.16);
			\draw (0,0.5) -- (1,0.5) ;
			\draw (0,0.5) -- (0,1) ;
			\draw (1,0.5) -- (1,1) ;
			\draw (0,1) -- (0.3,1) ;
			\draw (0.7,1) -- (1,1) ;
			\draw[blue,thick,dash dot] (0.3,1) -- (0.3,1.16) ;
			\draw[blue,thick,dash dot] (0.7,1) -- (0.7,1.16) ;
			\draw[blue,thick] (0.5,1.15) node[above]{$\gamma=\gamma_\mathrm p$} ;
			\draw[blue,thick,dash dot] (0.3,1.16) -- (0.7,1.16) ;
			\draw (0.15,0.65) node{$\Omega$} ;
			\draw[gray] (0.5,1.08)node{\small $F=F_\mathrm p$};
			\end{tikzpicture}
			\caption{Domain with a positive feature.\\}
			\label{fig:ex1ccomplex}
		\end{center}
	\end{subfigure}
	~
	\begin{subfigure}[b]{0.23\textwidth}
		\begin{center}
			\begin{tikzpicture}[scale=2.8]
			\fill[pattern=north west lines, pattern color=green!50!black, opacity=0.3] (0.3,0.84) rectangle (0.5,1);
			\fill[gray, fill, opacity=0.3] (0.5,1) rectangle (0.7,1.16);
			\draw (0,0.5) -- (1,0.5) ;
			\draw (0,0.5) -- (0,1) ;
			\draw (1,0.5) -- (1,1) ;
			\draw (0,1) -- (0.3,1) ;
			\draw (0.7,1) -- (1,1) ;
			\draw[blue,thick] (0.3,1) -- (0.3,0.84) ;
			\draw[blue,thick,dash dot] (0.5,1.16) -- (0.5,1) ;
			\draw[blue,thick] (0.5,1)--(0.5,0.84) ;
			\draw[blue,thick,dash dot] (0.7,1) -- (0.7,1.16) ;
			\draw[blue,thick] (0.4,0.84) node[below]{$\gamma_\mathrm n$} ;
			\draw[blue,thick] (0.6,1.16) node[above]{$\gamma_\mathrm p$} ;
			\draw[blue,thick,dash dot] (0.5,1.16) -- (0.7,1.16) ;
			\draw[blue,thick] (0.3,0.84) -- (0.5,0.84) ;
			\draw (0.15,0.65) node{$\Omega$} ;
			\draw[gray] (0.6,1.08)node{\small $F_\mathrm p$};
			\draw[green!50!black] (0.4,0.92)node{\small $F_\mathrm n$};
			\end{tikzpicture}
			\caption{Domain with a general complex feature.}
			\label{fig:ex1dcomplex}
		\end{center}
	\end{subfigure}
	~
	\begin{subfigure}[b]{0.23\textwidth}
		\begin{center}
			\begin{tikzpicture}[scale=2.8]
			\draw (0,0.5) -- (1,0.5) ;
			\draw (0,0.5) -- (0,1) ;
			\draw (1,0.5) -- (1,1) ;
			\draw (0,1) -- (0.3,1) ;
			\draw (0.7,1) -- (1,1) ;
			\draw[red,thick] (0.3,1) -- (0.7,1) ;
			\draw[red,thick] (0.3,0.98) -- (0.3, 1.02); 
			\draw[red,thick] (0.7,0.98) -- (0.7, 1.02); 
			\draw[red,thick] (0.5,1) node[below]{$\gamma_0$} ;
			\draw (0.15,0.65) node{$\Omega_0$} ;
			\end{tikzpicture}
			\caption{Simplified domain for domains (a)--(c), (e)--(f).}
			\label{fig:ex1bcomplex}
		\end{center}
	\end{subfigure}
	~
	\begin{subfigure}[b]{0.23\textwidth}
		\begin{center}
			\begin{tikzpicture}[scale=2.8]
			\fill[pattern=north west lines, pattern color=green!50!black, opacity=0.3] (0.3,0.84) rectangle (0.6,1);
			\fill[gray, fill, opacity=0.3] (0.5,1) rectangle (0.7,1.2);
			\draw (0,0.5) -- (1,0.5) ;
			\draw (0,0.5) -- (0,1) ;
			\draw (1,0.5) -- (1,1) ;
			\draw (0,1) -- (0.3,1) ;
			\draw (0.7,1) -- (1,1) ;
			\draw[blue,thick] (0.3,1) -- (0.3,0.84) ;
			\draw[blue,thick,dash dot] (0.5,1.2) -- (0.5,1); 
			\draw[blue,thick,dash dot] (0.5,1) -- (0.6,1); 
			\draw[blue,thick](0.6,1) -- (0.6,0.84) ;
			\draw[blue,thick,dash dot] (0.7,1) -- (0.7,1.2) ;
			\draw[blue,thick] (0.45,0.84) node[below]{$\gamma_\mathrm n$} ;
			\draw[blue,thick] (0.6,1.18) node[above]{$\gamma_\mathrm p$} ;
			\draw[blue,thick,dash dot] (0.5,1.2) -- (0.7,1.2) ;
			\draw[blue,thick] (0.3,0.84) -- (0.6,0.84) ;
			\draw (0.15,0.65) node{$\Omega$} ;
			\draw[gray] (0.6,1.1)node{\small $F_\mathrm p$};
			\draw[green!50!black] (0.45,0.92)node{\small $F_\mathrm n$};
			\end{tikzpicture}
			\caption{Domain with a general complex feature.}
			\label{fig:ex1fcomplex}
		\end{center}
	\end{subfigure}
	~
	\begin{subfigure}[b]{0.23\textwidth}
		\begin{center}
			\begin{tikzpicture}[scale=2.8]
			\fill[gray, fill, opacity=0.3] (0.4,0.87) rectangle (0.5,1.2);
			\fill[gray, fill, opacity=0.3] (0.5,1) rectangle (0.7,1.2);
			\fill[pattern=north west lines, pattern color=green!50!black, opacity=0.5](0.3,0.84) rectangle (0.6,1);
			\draw (0,0.5) -- (1,0.5) ;
			\draw (0,0.5) -- (0,1) ;
			\draw (1,0.5) -- (1,1) ;
			\draw (0,1) -- (0.3,1) ;
			\draw (0.7,1) -- (1,1) ;
			\draw[blue,thick] (0.3,1) -- (0.3,0.84) ;
			\draw[blue,thick,dash dot] (0.4,1.2) -- (0.4,0.87); 
			\draw[blue,thick,dash dot] (0.5,1) -- (0.5,0.87); 
			\draw[blue,thick,dash dot] (0.4,0.87) -- (0.5,0.87); 
			\draw[blue,thick,dash dot] (0.5,1) -- (0.6,1); 
			\draw[blue,thick](0.6,1) -- (0.6,0.84) ;
			\draw[blue,thick,dash dot] (0.7,1) -- (0.7,1.2) ;
			\draw[blue,thick] (0.45,0.84) node[below]{$\gamma_\mathrm n$} ;
			\draw[blue,thick] (0.55,1.18) node[above]{$\gamma_\mathrm p$} ;
			\draw[blue,thick,dash dot] (0.4,1.2) -- (0.7,1.2) ;
			\draw[blue,thick] (0.3,0.84) -- (0.6,0.84) ;
			\draw (0.15,0.65) node{$\Omega$} ;
			\draw[gray] (0.5,1.1)node{\small$F_\mathrm p$};
			\draw[green!50!black] (0.15,0.92)node{\small $F_\mathrm n$};
			\draw[green!50!black] (0.35,0.92) -- (0.23,0.92);
			\end{tikzpicture}
			\caption{Domain with a general complex feature.}
			\label{fig:ex1ecomplex}
		\end{center}
	\end{subfigure}
	~
	\begin{subfigure}[b]{0.23\textwidth}
		\begin{center}
			\begin{tikzpicture}[scale=2.8]
			\draw (0,0.5) -- (1,0.5) ;
			\draw (0,0.5) -- (0,1) ;
			\draw (1,0.5) -- (1,1) ;
			\draw (0,1) -- (0.3,1) ;
			\draw (0.7,1) -- (1,1) ;
			\draw[red,thick] (0.3,1) -- (0.6, 1); 
			\draw[red,thick] (0.6,0.98) -- (0.6, 1.02); 
			\draw[red,thick] (0.3,0.98) -- (0.3, 1.02); 
			\draw[red,thick] (0.7,0.98) -- (0.7, 1.02); 
			\draw[red,thick,dash dot] (0.6,1) -- (0.7,1) ;
			\draw[red,thick] (0.45,1) node[below]{$\gamma_{0,\mathrm n}$} ;
			\draw[red,thick] (0.65,1) node[above]{$\gamma_{0,\mathrm p}$} ;
			\draw (0.15,0.65) node{$\Omega_0$} ;
			\end{tikzpicture}
			\caption{Simplified domain for domains (e), (f).}
			\label{fig:ex1gcomplex}
		\end{center}
	\end{subfigure}
	~
	\begin{subfigure}[b]{0.23\textwidth}
		\begin{center}
			\begin{tikzpicture}[scale=2.8]
			\fill[gray, fill, opacity=0.3] (0.3,1.05) rectangle (0.7,1.2);
			\fill[gray, fill, opacity=0.3] (0.3,1) rectangle (0.4,1.05);
			\fill[gray, fill, opacity=0.3] (0.6,1) rectangle (0.7,1.05);
			\draw (0,0.5) -- (1,0.5) ;
			\draw (0,0.5) -- (0,1) ;
			\draw (1,0.5) -- (1,1) ;
			\draw (0,1) -- (0.3,1) ;
			\draw (0.7,1) -- (1,1) ;
			\draw (0.4,1) -- (0.6,1);
			\draw[blue,thick,dash dot] (0.3,1) -- (0.3, 1.2); 
			\draw[blue,thick,dash dot] (0.3,1.2) -- (0.7, 1.2); 
			\draw[blue,thick,dash dot] (0.7,1) -- (0.7, 1.2); 
			\draw[blue,thick,dash dot] (0.4,1) -- (0.4, 1.05); 
			\draw[blue,thick,dash dot] (0.4,1.05) -- (0.6, 1.05); 
			\draw[blue,thick,dash dot] (0.6,1) -- (0.6, 1.05); 
			\draw[blue,thick] (0.5,1.18) node[above]{$\gamma = \gamma_{\mathrm p}$} ;
			\draw (0.15,0.65) node{$\Omega_0$} ;
			\draw[gray] (0.5,1.1)node{\small $F=F_\mathrm p$};
			\draw[red,thick,dash dot] (0.6,1) -- (0.7,1) ;
			\draw[red,thick,dash dot] (0.3,1) -- (0.4,1) ;
			\draw[red] (0.49,0.9) -- (0.35,0.98);
			\draw[red] (0.51,0.9) -- (0.65,0.98);
			\draw[red] (0.545,0.91) node[below]{$\gamma_0 = \gamma_{0,\mathrm p}$} ;
			\end{tikzpicture}
			\caption{Domain  with a positive feature, \\$\Omega := \text{int}\left(\overline{\Omega_0} \cup \overline{F}\right)$.}
			\label{fig:ex1hcomplex}
		\end{center}
	\end{subfigure}
	\caption{Different examples of geometries with a negative, a positive, or a general complex feature.} \label{fig:exfeatcomplex}
\end{figure}

Suppose now that $F\subset \mathbb{R}^n$ is a complex feature. More precisely, this means that we suppose that $F$ is an open Lipschitz domain which is composed of a negative component $F_\mathrm n$ and a positive component $F_\mathrm p$ that can have a non-empty intersection (see Fig.~\ref{fig:exfeatcomplex}). More precisely, $F=\mathrm{int}\left(\overline{F_\mathrm n}\cup \overline{F_\mathrm p}\right)$, where $F_\mathrm n$ and $F_\mathrm p$ are open Lipschitz domains such that if we let 
\begin{equation*}
\Omega^\star := \Omega \setminus \overline{F_\mathrm p},
\end{equation*}
then 
$F_\mathrm p \subset \Omega$ and $\overline{F_\mathrm n} \cap \overline{\Omega^\star} \subset \partial \Omega^\star$.
In particular, note that if $F_\mathrm p = \emptyset$ and $F=F_\mathrm n$, then $F$ is negative, while if $F_\mathrm n = \emptyset$ and $F=F_\mathrm p$, then $F$ is positive, as defined in Sec.~\ref{sec:defeatpb}. 

In this setting, the defeatured geometry is defined by $\Omega_0 := \mathrm{int}\left(\overline{\Omega^\star}\cup \overline{F_\mathrm n}\right)\subset\mathbb{R}^n$,
and as before, we also assume that $\Omega_0$ is an open Lipschitz domain. Note that it is in general not true that $\Omega^\star = \Omega \cap \Omega_0$ (see Fig.~\ref{fig:ex1ecomplex}), while it is true if $F$ is completely negative or positive. 

The considered problem in the exact geometry $\Omega$ is still the Poisson equation defined in (\ref{eq:originalpb}), for which we assume that ${\Gamma_D} \cap \left(\partial F_\mathrm n \cup \partial F_\mathrm p\right) = \emptyset$. Moreover, let
\begin{align*}
\gamma_0 := \text{int}\left(\overline{\gamma_{0,\mathrm n}} \cup \overline{\gamma_{0,\mathrm p}}\right) \subset \partial \Omega_0 &\quad \text{with} \quad \gamma_{0,\mathrm n} := \partial F_\mathrm n \setminus \partial \Omega^\star, \quad \gamma_{0,\mathrm p}:= \partial F_\mathrm p \setminus \partial \Omega, \\
\gamma :=\text{int}\left(\overline{\gamma_{\mathrm n}} \cup \overline{\gamma_{\mathrm p}}\right) \subset \partial \Omega &\quad \text{with} \quad \gamma_\mathrm n:= \partial F_\mathrm n\setminus \overline{\gamma_{0,\mathrm n}}, \quad \gamma_\mathrm p:= \partial F_\mathrm p\setminus \overline{\gamma_{0,\mathrm p}}.
\end{align*} 
so that $\partial F_\mathrm n = \overline{\gamma_\mathrm n} \cup \overline{\gamma_{0,\mathrm n}}$ with $\gamma_\mathrm n\cap \gamma_{0,\mathrm n}=\emptyset$, and $\partial F_\mathrm p = \overline{\gamma_\mathrm p} \cup \overline{\gamma_{0,\mathrm p}}$ with $\gamma_\mathrm p\cap \gamma_{0,\mathrm p}=\emptyset$ (see Fig.~\ref{fig:exfeatcomplex}). 

Similarly to the negative feature case, consider any $L^2$-extension of the restriction $f\vert_{\Omega^\star}$ in the negative component $F_\mathrm n$ of $F$, that we still write $f\in L^2(\Omega_0)$ by abuse of notation. Then instead of (\ref{eq:originalpb}), we solve the defeatured (or simplified) problem (\ref{eq:simplpb}) whose weak formulation is given in (\ref{eq:weaksimplpb}). As previously, we are interested in controlling the energy norm of the defeaturing error, which we suitably define in what follows. 

Similarly to the positive feature case, the solution $u_0$ of the defeatured problem is not defined everywhere on $\Omega$ since $F_\mathrm p \setminus \overline{F_\mathrm n} \not\subset \Omega_0$ but $F_\mathrm p \setminus \overline{F_\mathrm n} \subset \Omega$. Therefore, following the same rationale for $F_\mathrm p$ as the one exposed in Sec.~\ref{sec:defeatpb}, let $\tilde{F}_\mathrm p\subset \mathbb{R}^n$ be a Lipschitz domain that contains $F_\mathrm p$ and such that $\gamma_{0,\mathrm p} \subset \left(\partial \tilde{F}_\mathrm p \cap \partial F_\mathrm p\right)$, that is, $\tilde F_\mathrm p$ is a suitable (simple) domain extension of $F_\mathrm p$ such as the bounding box of $F_\mathrm p$ for example. Let us also assume that $\tilde F_\mathrm p \setminus \overline{F_\mathrm p}$ is Lipschitz, and consider any $L^2$-extension of $f$ in $\tilde F_\mathrm p$, that we still write $f$ by abuse of notation. Let $\tilde{\mathbf{n}}$ be the unitary outward normal of $\tilde F_\mathrm p$, let $\tilde \gamma := \partial \tilde F_\mathrm p \setminus \partial F_\mathrm p$, and let $\gamma_\mathrm p$ be decomposed as $\gamma_\mathrm p = \text{int}(\overline{\gamma_\intersign}\cup\overline{\gamma_\setminussign})$, where $\gamma_\intersign$ and $\gamma_\setminussign$ are open, $\gamma_\intersign$ is the part of $\gamma_\mathrm p$ that is shared with $\partial \tilde F_\mathrm p$ while $\gamma_\setminussign$ is the remaining part of $\gamma_\mathrm p$, that is, the part that does not belong to $\partial \tilde F_\mathrm p$, see Fig.~\ref{fig:twofeatboundaries}. 

\begin{figure}
	\centering
	\begin{subfigure}[b]{0.48\textwidth}
		\begin{center}
			\begin{tikzpicture}[scale=5.5]
			\fill[gray!50!white, fill, opacity=0.3] (2.4,0.8) rectangle (2.7,1.00);
			\fill[gray, fill, opacity=0.3] (2.3,1) -- (2.6,1) -- (2.6,1) arc (0:90:0.3) -- cycle;
			\draw[gray] (2.55,0.9)node{$F_\mathrm n$};
			\draw[gray] (2.425,1.125)node{$F_\mathrm p$};
			\draw (2,0.7) -- (2,1) ;
			\draw (3,0.7) -- (3,1) ;
			\draw (2,1) -- (2.3,1) ;
			\draw (2.7,1) -- (3,1) ;
			\draw[blue,thick] (2.4,1) -- (2.4,0.8) ;
			\draw[blue,thick] (2.7,1) -- (2.7,0.8) ;
			\draw[blue,thick] (2.4,0.8) -- (2.7,0.8) ;
			\draw[cyan,thick] (2.3,1) -- (2.3,1.3) ;
			\draw[cyan,thick] (2.6,1) arc (0:90:0.3); 
			\draw[cyan,thick] (2.4,1) -- (2.6,1) ;
			\draw (2.25,0.8) node{$\Omega$} ;
			\draw[cyan,thick] (2.55,1.2) node[above]{$\gamma_\mathrm p$} ;
			\draw[blue,thick] (2.55,0.8) node[below]{$\gamma_\mathrm n$} ;
			\draw (2.41,0.99) -- (2.39,1.01);
			\end{tikzpicture}
			\caption{Domain $\Omega$ with feature $F$ that has non-empty positive and negative components $F_\mathrm p$ and $F_\mathrm n$.}
			\label{fig:exa}
		\end{center}
	\end{subfigure}
	~
	\begin{subfigure}[b]{0.48\textwidth}
		\begin{center}
			\begin{tikzpicture}[scale=5.5]
			\fill[pattern=north west lines, pattern color=green!50!black, opacity=0.3] (2.3,1) rectangle (2.6,1.3);
			\draw[green!30!black] (2.45,1.15)node{$\tilde F_\mathrm p$};
			\draw (2,0.7) -- (2,1) ;
			\draw (3,0.7) -- (3,1) ;
			\draw (2,1) -- (2.3,1) ;
			\draw (2.7,1) -- (3,1) ;
			\draw[green!40!black,thick] (2.3,1.3) -- (2.6,1.3) ;
			\draw[green!40!black,thick] (2.6,1.3) -- (2.6,1) ;
			\draw[red,thick] (2.3,1) -- (2.4,1) ;
			\draw (2.3,1) -- (2.3,1.3) ;
			\draw[cyan,thick] (2.6,1) arc (0:90:0.3); 
			\draw[orange,thick] (2.7,1) -- (2.4,1) ;
			\draw (2.5,0.8) node{$\Omega_0$} ;
			\draw[red] (2.35,1) node[below]{$\gamma_{0,\mathrm p}$} ;
			\draw[orange] (2.55,1) node[below]{$\gamma_{0,\mathrm n}$} ;
			\draw[green!40!black,thick] (2.5,1.3) node[above]{$\tilde \gamma$} ;
			\draw[cyan] (2.55,1.16) node[above]{$\gamma_\setminussign$} ;
			\draw (2.4,0.99) -- (2.4,1.01);
			\draw (2.7,0.99) -- (2.7,1.01);
			\end{tikzpicture}
			\caption{Simplified domain $\Omega_0$, extension $\tilde F_\mathrm p$ of the positive component of the feature, and different boundaries.}
			\label{fig:exb}
		\end{center}
	\end{subfigure}
	\caption{Example of a geometry with a feature whose positive and negative components share a part of the boundary.} \label{fig:twofeatboundaries}
\end{figure}
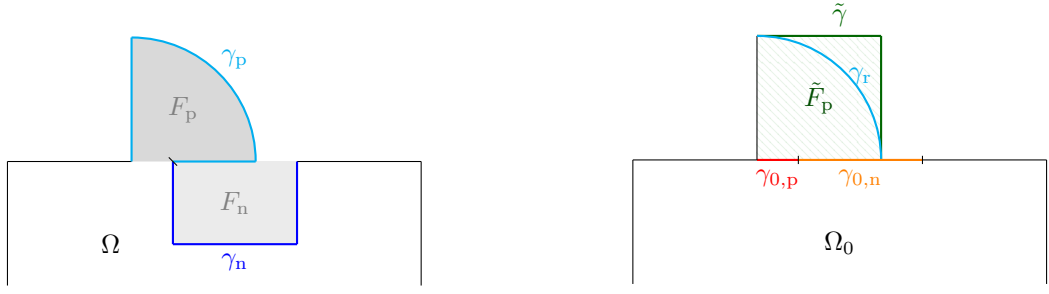

Therefore, and as for the positive feature case, we can consider the extension of the solution $u_0$ of (\ref{eq:simplpb}) on $\tilde F_\mathrm p$, called $\tilde{u}_0\in H^1_{u_0,\gamma_{0,\mathrm p}}\left(\tilde{F}_\mathrm p\right)$ and defined as the weak solution of (\ref{eq:featurepb}) where $F$, $\tilde F$ and $\gamma_0$ are replaced by $F_\mathrm p$, $\tilde F_\mathrm p$ and $\gamma_{0,\mathrm p}$, respectively. Now, we can define the extended defeatured solution $u_\mathrm d\in H^1_{h,\Gamma_D}\left(\Omega\right)$ as
\begin{equation}\label{eq:defudgen}
u_\mathrm d = u_0\vert_{\Omega_0 \setminus \overline{F_\mathrm n}} \text{ in } \Omega^\star = \Omega\setminus\overline{F_\mathrm p} \quad \text{ and } \quad u_\mathrm d = \tilde{u}_0\vert_{F_\mathrm p} \text{ in } F_\mathrm p. 
\end{equation}
Then the defeaturing error is defined by $\left|u-u_\mathrm d\right|_{1,\Omega}$. 

\begin{remark} Note that if $F_\mathrm n \cap F_\mathrm p \neq \emptyset$, it may happen that $u_0\neq \tilde u_0$ on $F_\mathrm n \cap F_\mathrm p$. But in this case,  on $F_\mathrm n \cap F_\mathrm p$, the definition of $u_\mathrm d$ in (\ref{eq:defudgen}) specifies that $u_\mathrm d = \tilde u_0$. 
\end{remark}

\noindent In this setting, we suppose that $\gamma_\mathrm n$, $\gamma_{0,\mathrm p}$ and $\gamma_\setminussign$ are isotropic according to Definition \ref{as:isotropy}, where the diameter and the convex hull of $\gamma_\mathrm n$, $\gamma_{0,\mathrm p}$ and $\gamma_\setminussign$ are considered in the manifolds $\partial \Omega$, $\partial \Omega_0$ and $\partial F$, respectively (see Sec.~\ref{sec:notation}). Note that, as before, the considered boundaries can be non-connected sub-manifolds. Finally, for further use, let $\Sigma := \{\gamma_{\mathrm n}, \gamma_{\setminussign}, \gamma_{0,\mathrm p} \}$ and let $d_\Sigma$ be defined piecewise as $d_\Sigma|_{\sigma} := d_\sigma$ for all $\sigma \in\Sigma$, with
\begin{equation} \label{eq:dsigma}
d_\sigma := \begin{cases}\vspace{1mm}
g - \displaystyle\frac{\partial u_\mathrm d}{\partial \mathbf n} & \text{if } \sigma = \gamma_\mathrm n \text{ or } \sigma = \gamma_\setminussign\\\vspace{1mm}
- \left(g_0 + \displaystyle\frac{\partial u_\mathrm d}{\partial \mathbf n_F}\right) & \text{if } \sigma = \gamma_{0,\mathrm p}.
\end{cases}
\end{equation}
{That is, $d_\sigma$ is the error term on the Neumann data for $\sigma = \gamma_\mathrm n$ or $\sigma = \gamma_\setminussign$, and $d_{\gamma_{0,\mathrm p}}$ is the jump in the normal derivative of $u_\mathrm d$ due to the Dirichlet extension of $u_\mathrm d$ in the positive component of the feature.}

\subsection{Complex feature \textit{a posteriori }defeaturing error estimator} \label{sec:complexest}
Recalling the definition of the defeaturing solution $u_\mathrm d$ in (\ref{eq:defudgen}) and of $\Sigma$ and $d_\sigma$ in (\ref{eq:dsigma}), we define the defeaturing error estimator as
\begin{align}
\mathcal E (u_\mathrm d) :=  &\left[ \sum_{\sigma \in \Sigma} \left( \left|\sigma\right|^{\frac{1}{n-1}} \left\|d_\sigma - \overline{d_\sigma}^{\sigma}\right\|_{0,\sigma}^2 + \,c_{\sigma}^2 \left|\sigma\right|^{\frac{n}{n-1}} \left| \overline{d_\sigma}^{\sigma} \right|^2\right) \right]^\frac{1}{2},\label{eq:genestimator}
\end{align}
where $c_{\sigma}$ is defined as in (\ref{eq:negconstant}). 

\begin{remark}
	If $F$ is a negative feature, then $\mathcal E(u_\mathrm d) = \mathcal E_\mathrm n(u_0)$ where $\mathcal E_\mathrm n(u_0)$ is defined in (\ref{eq:negestimator}), while if $F$ is a positive feature, then 
	$\mathcal E(u_\mathrm d) = \mathcal E_\mathrm p(\tilde u_0)$ where
	\begin{align*}
	\mathcal E_{\mathrm p} (\tilde u_0) :=  &\left( \left|\gamma_0\right|^{\frac{1}{n-1}} \left\|\left(g_0 + \frac{\partial \tilde u_0}{\partial \mathbf{n}_F}\right) - \overline{\left(g_0 + \frac{\partial \tilde u_0}{\partial \mathbf n_F}\right)}^{\gamma_0}\right\|_{0,\gamma_0}^2 \right. \\
	& \quad + \left|\gamma_\setminussign\right|^{\frac{1}{n-1}} \left\|\left(g - \frac{\partial \tilde u_0}{\partial \mathbf{n}_F}\right) - \overline{\left(g - \frac{\partial \tilde u_0}{\partial \mathbf n_F}\right)}^{\gamma_\setminussign}\right\|_{0,\gamma_\setminussign}^2 \nonumber \\
	& \quad \left. + \,c_{\gamma_0}^2 \left|\gamma_0\right|^{\frac{n}{n-1}} \left| \overline{\left(g_0 + \frac{\partial \tilde u_0}{\partial \mathbf n_F}\right)}^{\gamma_0} \right|^2 + c_{\gamma_\setminussign}^2 \left|\gamma_\setminussign\right|^{\frac{n}{n-1}}\left| \overline{\left(g - \frac{\partial \tilde u_0}{\partial \mathbf n_F}\right)}^{\gamma_\setminussign} \right|^2\,\right)^\frac{1}{2}.
	\end{align*}
\end{remark}

In this section, we first show that the quantity $\mathcal{E}(u_\mathrm d)$ is a reliable estimator for the defeaturing error, i.e., it is an upper bound for the defeaturing error (see Theorem \ref{thm:genupperbound}). Then, assuming that $\gamma_{\mathrm n}$, $\gamma_\setminussign$ and $\gamma_{0,\mathrm p}$ are also regular according to Definition \ref{as:pwsmoothshapereg}, and under mild assumptions for the two-dimensional case, we show that it is also efficient (up to oscillations), i.e., it is a lower bound for the defeaturing error up to oscillations (see Theorem \ref{thm:genlowerbound}). 

\begin{remark} \label{rmk:compatcondgen}
	Consider the simplified extended problem (\ref{eq:featurepb}) restricted to $F_\mathrm p$ and then to $\tilde F_\mathrm p\setminus F_\mathrm p$, with the natural Neumann boundary condition on $\gamma_{0,\mathrm p}$ and $\gamma_\setminussign$ respectively, in a similar way to (\ref{eq:simplpbrestricted}). By abuse of notation and as previously, we omit the explicit restriction of $\tilde u_0$ to $F_\mathrm p$ or to $\tilde F_\mathrm p\setminus F_\mathrm p$.  Then if we multiply the restricted problems by the constant function $1$ and integrate by parts, we obtain
	\begin{align*}
	&\int_{F_\mathrm p} f\, \mathrm dx + \int_{\gamma_\mathrm p} g \,\mathrm ds + \int_{\gamma_{0,\mathrm p}} \frac{\partial \tilde u_0}{\partial \mathbf n_F}\,\mathrm ds = 0, \\
	\text{ and }\quad &\int_{\tilde F_\mathrm p\setminus F_\mathrm p} f\, \mathrm dx + \int_{\tilde \gamma} \tilde g \,\mathrm ds - \int_{\gamma_\setminussign} \frac{\partial \tilde u_0}{\partial \mathbf n_F}\,\mathrm ds = 0.
	\end{align*}
	Consequently, 
	\begin{align*}
	\overline{d_{\gamma_{0,\mathrm p}}}^{\gamma_{0,\mathrm p}} &= \overline{\left(g_0+\frac{\partial \tilde u_0}{\partial \mathbf n_F}\right)}^{\gamma_{0,\mathrm p}} = \frac{1}{\left|\gamma_{0,\mathrm p}\right|}\left( \int_{\gamma_{0,\mathrm p}} g_0\,\mathrm ds - \int_{\gamma_\mathrm p} g\,\mathrm ds - \int_{F_\mathrm p} f\,\mathrm dx\right), \nonumber \\
	\overline{d_{\gamma_\setminussign}}^{\gamma_\setminussign} &= \overline{\left(g-\frac{\partial \tilde u_0}{\partial \mathbf n_F}\right)}^{\gamma_\setminussign} = \frac{1}{\left|\gamma_\setminussign\right|}\left( \int_{\gamma_\setminussign} g\,\mathrm ds - \int_{\tilde \gamma} \tilde g\,\mathrm ds - \int_{\tilde F_\mathrm p\setminus F_\mathrm p} f\,\mathrm dx\right). 
	\end{align*}
	Moreover, as in Remark \ref{rmk:compatcondneg}, it can be seen that 
	$$\overline{d_{\gamma_\mathrm n}}^{\gamma_\mathrm n} = \frac{1}{\left|\gamma_\mathrm n\right|}\left( \int_{\gamma_\mathrm n} g\,\mathrm ds - \int_{\gamma_{0,\mathrm n}} g_0\,\mathrm ds - \int_{F_\mathrm n} f\,\mathrm dx\right).$$
	Therefore, the terms involving the average values of $d_\sigma$ in the estimator $\mathcal{E}(u_\mathrm d)$ defined in (\ref{eq:genestimator}) only depend on the defeatured problem data. More precisely, they only depend on the choice of $g_0$ and $\tilde g$ that one chooses on $\gamma_0$ and $\tilde \gamma$ respectively, and on the choice of the extension of $f$ that one considers in the extended feature $\tilde F_\mathrm p$. 
	As a consequence, if those terms dominate, this means that the defeatured problem data should be better chosen. 
	Moreover, under the following reasonable flux conservation assumptions
	\begin{align}
	\int_{\gamma_{0,\mathrm p}} g_{0} \,\mathrm ds = \int_{\gamma_\mathrm p} g \,\mathrm{d}s + \int_{F_\mathrm p} f\,\mathrm dx, \qquad \int_{\tilde \gamma} \tilde g \,\mathrm ds = \int_{\gamma_{\setminussign}} g \,\mathrm{d}s - \int_{\tilde F_\mathrm p\setminus F_\mathrm p} f\,\mathrm dx,\nonumber \\
	\text{and } \qquad \int_{\gamma_{0,\mathrm n}} g_{0} \,\mathrm ds = \int_{\gamma_\mathrm n} g \,\mathrm{d}s - \int_{F_\mathrm n} f\,\mathrm dx, \label{eq:compatcondgen}
	\end{align}
	the defeaturing error estimator (\ref{eq:genestimator}) rewrites
	$\mathcal E(u_\mathrm d) := \left( \displaystyle\sum_{\sigma \in \Sigma} \left|\sigma\right|^{\frac{1}{n-1}} \left\| d_\sigma \right\|^2_{0,\sigma}\right)^\frac{1}{2}.$
	Conditions (\ref{eq:compatcondgen}) are easily met if the Neumann boundary condition $g$ and the source function $f$ are zero in the vicinity of the feature. 
\end{remark}

\begin{remark} \label{rmk:estgentilde}
	Analogously to the case of a negative feature in Remark \ref{rmk:estnegtilde}, note that
	\begin{align*}
	\mathcal{E}(u_\mathrm d) \lesssim \left( \displaystyle\sum_{\sigma\in\Sigma} c_\sigma^2 \left|\sigma\right|^{\frac{1}{n-1}} \left\| d_\sigma \right\|^2_{0,\sigma}\right)^\frac{1}{2} =: \tilde{\mathcal{E}}(u_\mathrm d). 
	\end{align*}
	One could be tempted to use the simpler indicator $\tilde{\mathcal{E}}(u_\mathrm d)$, but when $n=2$ and under the flux conservation conditions (\ref{eq:compatcondgen}), $\tilde{\mathcal{E}}\left(u_\mathrm d\right)$ is sub-optimal since in this case, $\tilde{\mathcal{E}}(u_\mathrm d) \lesssim \displaystyle\max_{\sigma\in\Sigma}\left( c_{\sigma} \right) {\mathcal{E}}(u_\mathrm d)$. Indeed, no lower bound can be proven for $\tilde{\mathcal{E}}(u_\mathrm d)$. 
\end{remark}

\subsubsection{Upper bound}
In this section, we prove that the error indicator defined in (\ref{eq:genestimator}) is reliable, that is, it is an upper bound for the defeaturing error. 

\begin{theorem} \label{thm:genupperbound}
	Let $u_\mathrm d$ be the defeaturing solution as defined in (\ref{eq:defudgen}). If $\gamma_{\mathrm n}$, $\gamma_\setminussign$ and $\gamma_{0,\mathrm p}$ are isotropic according to Definition \ref{as:isotropy}, then the defeaturing error in energy norm is bounded in terms of the estimator $\mathcal{E}(u_\mathrm d)$ introduced in (\ref{eq:genestimator}) as follows:
	\begin{equation*}
	\left|u-u_\mathrm d\right|_{1,\Omega} \lesssim \mathcal{E}(u_\mathrm d).
	\end{equation*}
\end{theorem}

\begin{proof}
	Using arguments similar to Theorem \ref{thm:upperbound}, let us first consider the original problem (\ref{eq:originalpb}) restricted to $\Omega^\star:=\Omega\setminus \overline{F_\mathrm p}$ with the natural Neumann boundary condition on $\gamma_{0,\mathrm p}$, that is the restriction $u\vert_{\Omega^\star} \in H^1_{h,\Gamma_D}(\Omega^\star)$ is the weak solution of
	\begin{align}\label{eq:originalpbrestricted}
	\begin{cases}
	-\Delta \left(u\vert_{\Omega^\star}\right) = f &\text{ in } \Omega^\star \\
	u\vert_{\Omega^\star} = h &\text{ on } \Gamma_D \\
	\displaystyle\frac{\partial \left(u\vert_{\Omega^\star}\right)}{\partial \mathbf{n}} = g &\text{ on } \Gamma_N\setminus \gamma_\mathrm p \vspace{0.1cm} \\
	\displaystyle\frac{\partial \left(u\vert_{\Omega^\star}\right)}{\partial \mathbf{n}_0} =	\displaystyle\frac{\partial u}{\partial \mathbf{n}_0}  &\text{ on } \gamma_{0,\mathrm p}.
	\end{cases}
	\end{align}
	By abuse of notation, we omit the explicit restriction of $u$ to $\Omega^\star$. Then for all $v_0\in H^1_{0,\Gamma_D}(\Omega^\star)$, 
	\begin{equation}\label{eq:weakoriginalpb0}
	\int_{\Omega^\star} \nabla u \cdot \nabla v_0 \,\mathrm dx = \int_{\Omega^\star} fv_0 \,\mathrm dx + \int_{\Gamma_N\setminus \gamma_\mathrm p} g v_0 \,\mathrm ds + \int_{\gamma_{0,\mathrm p}} \frac{\partial u}{\partial \mathbf{n}_0} v_0 \,\mathrm ds.
	\end{equation}
	Then, let us consider the simplified problem (\ref{eq:simplpb}) restricted to $\Omega^\star$ with the natural Neumann boundary condition on $\gamma_\mathrm n$, in the same way as in (\ref{eq:simplpbrestricted}). Thus, since by definition $u_\mathrm d\vert_{\Omega^\star} = u_0\vert_{\Omega^\star}$, and if we omit the explicit restriction of $u_\mathrm d$ to $\Omega^\star$, for all $v_0\in H^1_{0,\Gamma_D}(\Omega^\star)$, 
	\begin{equation}\label{eq:weaksimplpb0}
	\int_{\Omega^\star} \nabla u_\mathrm d \cdot \nabla v_0 \,\mathrm dx = \int_{\Omega^\star} fv_0 \,\mathrm dx + \int_{\Gamma_N\setminus \gamma} g v_0 \,\mathrm ds + \int_{\gamma_{\mathrm n}} \frac{\partial u_\mathrm d}{\partial \mathbf{n}} v_0 \,\mathrm ds + \int_{\gamma_{0,\mathrm p}} g_0 v_0 \,\mathrm ds.
	\end{equation}
	Let $e:=u-u_\mathrm d \in H^1_{0,\Gamma_D}(\Omega)$. So from (\ref{eq:weakoriginalpb0}) and (\ref{eq:weaksimplpb0}), for all $v_0\in H^1_{0,\Gamma_D}(\Omega^\star)$, we obtain
	\begin{equation}\label{eq:multieinter}
	\int_{\Omega^\star} \nabla e\cdot\nabla v_0 \,\mathrm dx = \int_{\gamma_\mathrm n} \left(g-\frac{\partial u_\mathrm d}{\partial \mathbf n}\right)v_0\,\mathrm ds + \int_{\gamma_{0,\mathrm p}} \left(\frac{\partial u}{\partial \mathbf n_0}- g_0\right) v_0\,\mathrm ds.
	\end{equation}
	Now, let us consider the simplified extended problem (\ref{eq:featurepb}) restricted to $F_\mathrm p$ with the natural Neumann boundary condition on $\gamma_\setminussign$, in a similar way to (\ref{eq:originalpbrestricted}). Note that $u_\mathrm d\vert_{F_\mathrm p} = \tilde u_0\vert_{F_\mathrm p}$, and by abuse of notation and as previously, we omit the explicit restriction of $u_\mathrm d$ to $F_\mathrm p$. That is, $u_\mathrm d\in H^1\big(F_\mathrm p\big)$ is one of the infinitely-many solutions (up to a constant) of
	\begin{equation}\label{eq:weakfeatpbrestricted}
	\int_{F_\mathrm p} \nabla u_\mathrm d \cdot \nabla v_\mathrm p \,\mathrm dx = \int_{F_\mathrm p} fv_\mathrm p \,\mathrm dx + \int_{\gamma_\intersign} g v_\mathrm p \,\mathrm ds + \int_{\gamma_{0,\mathrm p}\cup\gamma_\setminussign} \frac{\partial u_\mathrm d}{\partial \mathbf{n}_F} v_\mathrm p\,\mathrm ds, \qquad \forall v_\mathrm p\in H^1\big(F_\mathrm p\big). 
	\end{equation}
	And let us consider the original problem (\ref{eq:originalpb}) restricted to $F_\mathrm p$ with the natural Neumann boundary condition on $\gamma_{0,\mathrm p}$, again in a similar way to (\ref{eq:originalpbrestricted}). By abuse of notation and as previously, we omit the explicit restriction of $u$ to $F_\mathrm p$. 
	So $u\in H^1\big(F_\mathrm p\big)$ is one of the infinitely-many solutions (up to a constant) of
	\begin{equation}\label{eq:weakoriginalpb1}
	\int_{F_\mathrm p} \nabla u \cdot \nabla v_\mathrm p \,\mathrm dx = \int_{F_\mathrm p} fv_\mathrm p \,\mathrm dx + \int_{\gamma_\mathrm p} gv_\mathrm p \,\mathrm dx + \int_{\gamma_{0,\mathrm p}} \frac{\partial u}{\partial \mathbf n_F} v_\mathrm p \,\mathrm ds, \qquad \forall v_\mathrm p\in H^1\big(F_\mathrm p\big).
	\end{equation}
	Consequently, from (\ref{eq:weakfeatpbrestricted}) and (\ref{eq:weakoriginalpb1}), for all $v_\mathrm p\in H^1\big(F_\mathrm p\big)$,
	\begin{equation} \label{eq:multieFi}
	\int_{F_\mathrm p} \nabla e\cdot\nabla v_\mathrm p \,\mathrm dx = \int_{\gamma_{0,\mathrm p}} \frac{\partial \left(u-u_\mathrm d\right)}{\partial\mathbf n_F}v_\mathrm p\,\mathrm ds + \int_{\gamma_\setminussign}\left(g-\frac{\partial u_\mathrm d}{\partial \mathbf n}\right) v_\mathrm p\,\mathrm ds.
	\end{equation}
	Let $v\in H^1_{0,\Gamma_D}(\Omega)$, then $v\vert_{\Omega^\star} \in H_{0,\Gamma_D}^1(\Omega^\star)$ and $v\vert_{F_\mathrm p}\in H^1\big(F_\mathrm p\big)$. 
	Therefore, from equations (\ref{eq:multieinter}) and (\ref{eq:multieFi}), since $\mathbf n_0 = - \mathbf n_F$ on $\gamma_{0,\mathrm p}$ and recalling the definitions of $\Sigma$ and $d_\sigma$ in (\ref{eq:dsigma}), we obtain
	\begin{align}
	\int_{\Omega} \nabla e\cdot \nabla v \,\mathrm{d}x = \sum_{\sigma \in \Sigma} \int_\sigma d_\sigma v\,\mathrm ds. \label{eq:rewriteerrormulti}
	\end{align}
	Now, if we take $v=e\in H_{0,\Gamma_D}^1(\Omega)$ in (\ref{eq:rewriteerrormulti}), then
	\begin{equation} \label{eq:posmeinesavg}
	|e|_{1,\Omega}^2 = \sum_{\sigma\in\Sigma} \int_\sigma d_\sigma e \,\mathrm ds = \sum_{\sigma\in\Sigma} \left[ \int_\sigma \left( d_\sigma -\overline{d_\sigma}^\sigma\right) \left(e-\overline{e}^\sigma\right) \,\mathrm ds + \overline{d_\sigma}^\sigma \int_\sigma e \,\mathrm ds \right].
	\end{equation}
	For each $\sigma\in\Sigma$, the first terms of (\ref{eq:posmeinesavg}) can be estimated as in (\ref{eq:12}), using Appendix \ref{lemma:poincarebd}, trace inequalities and the discrete Cauchy-Schwarz inequality. Thus we obtain
	\begin{align}
	\sum_{\sigma \in \Sigma} \int_{\sigma} \left(d_\sigma - \overline{d_\sigma}^{\sigma}\right) \left(e-\overline{e}^{\sigma}\right) \,\mathrm ds 
	\lesssim &\sum_{\sigma\in\Sigma} \left|\sigma\right|^{\frac{1}{2(n-1)}} \left\|d_\sigma - \overline{d_\sigma}^{\sigma}\right\|_{0,\sigma} |e|_{\frac{1}{2},\sigma} \nonumber\\
	\lesssim &\left|\gamma_\mathrm n\right|^{\frac{1}{2(n-1)}} \left\|d_{\gamma_\mathrm n} - \overline{d_{\gamma_\mathrm n}}^{\gamma_{\mathrm n}}\right\|_{0,\gamma_{\mathrm n}} \|e\|_{1,\Omega^\star} \nonumber \\
	& + \left(\sum_{\sigma\in\{ \gamma_{0,\mathrm p}, \gamma_\setminussign \}}\left|\sigma\right|^{\frac{1}{2(n-1)}} \left\|d_\sigma - \overline{d_\sigma}^{\sigma}\right\|_{0,\sigma}\right) \|e\|_{1,F_\mathrm p} \nonumber\\
	\lesssim & \left( \sum_{\sigma\in\Sigma} \left|\sigma\right|^{\frac{1}{n-1}} \left\|d_\sigma - \overline{d_\sigma}^{\sigma}\right\|^2_{0,\sigma} \right)^\frac{1}{2} |e|_{1,\Omega}. \label{eq:discrcs}
	\end{align}
	Moreover, for each $\sigma\in\Sigma$, the last terms of (\ref{eq:posmeinesavg}) can be estimated using Appendix \ref{lemma:savare}, trace inequalities and the discrete Cauchy-Schwarz inequality to obtain
	\begin{align} 
	\sum_{\sigma\in\Sigma} \overline{d_\sigma}^{\sigma} \int_{\sigma} e \,\mathrm ds \lesssim &\sum_{\sigma\in\Sigma} \left| \overline{d_\sigma}^{\sigma} \right| \left|\sigma\right|^\frac{1}{2} \|e\|_{0, \sigma} \nonumber \\
	\lesssim &\left( \sum_{\sigma\in \{\gamma_\mathrm n, \gamma_{0,\mathrm p}\}} \left| \overline{d_\sigma}^{\sigma} \right| c_{\sigma} \left|\sigma\right|^{\frac{1}{2(n-1)}+\frac{1}{2}} \right) \|e\|_{\frac{1}{2},\partial \Omega^\star} \nonumber \\
	& + \left| \overline{d_{\gamma_\setminussign}}^{\gamma_\setminussign} \right| c_{\gamma_\setminussign} \left|\gamma_\setminussign\right|^{\frac{1}{2(n-1)}+\frac{1}{2}} \|e\|_{\frac{1}{2},\partial \Omega} \nonumber \\
	\lesssim &\left( \sum_{\sigma \in \Sigma} c_{\sigma}^2 \left|\gamma_\sigma\right|^{\frac{n}{n-1}} \left| \overline{d_\sigma}^{\sigma} \right|^2\right)^\frac{1}{2} |e|_{1,\Omega}. \label{eq:poseq19}
	\end{align}
	Therefore, combining (\ref{eq:posmeinesavg}), (\ref{eq:discrcs}) and (\ref{eq:poseq19}), and simplifying on both sides, we obtain the desired result.
\end{proof}

\subsubsection{Lower bound}
In this section, we prove that the error indicator defined in (\ref{eq:genestimator}) is efficient, that is, it is a lower bound for the defeaturing error, up to oscillations. In the case $n=2$, the flux conservation assumptions (\ref{eq:compatcondgen}) are also required. 

\begin{theorem} \label{thm:genlowerbound}
	Consider the same notation and assumptions as in Theorem \ref{thm:genupperbound}, and assume that all $\sigma \in \Sigma$ are also regular according to Definition \ref{as:pwsmoothshapereg} with $\left|\gamma_\mathrm n\right| \simeq \left|\gamma_\setminussign\right| \simeq \left| \gamma_{0,\mathrm p} \right|$. 
	Suppose that either $n=3$, or $n=2$ and the flux conservation conditions (\ref{eq:compatcondgen}) are satisfied. Then the defeaturing error, in energy norm, bounds up to oscillations the estimator $\mathcal{E}(u_\mathrm d)$ introduced in (\ref{eq:genestimator}), that is
	\begin{equation*}
	\mathcal{E}(u_\mathrm d) \lesssim 
	\left|u-u_\mathrm d\right|_{1,\Omega} + \mathrm{osc}(u_\mathrm d),
	\end{equation*}
	where
	\begin{align}
	\mathrm{osc}(u_\mathrm d) &:=  \left|\Gamma\right|^\frac{1}{2(n-1)}  \left( \sum_{\sigma\in\Sigma} \left\| d_\sigma - \Pi_{m}(d_\sigma)\right\|^2_{0,\sigma}\right)^\frac{1}{2} \label{eq:oscgen}
	\end{align}
	for any $m\in \mathbb{N}$, with $\Gamma := \gamma_{\mathrm n} \cup \gamma_{\setminussign} \cup \gamma_{0,\mathrm p}$, $\Sigma$ and $d_\sigma$ defined as in (\ref{eq:dsigma}) and $\Pi_m$ such that $\Pi_{m}\vert_{\sigma} \equiv \Pi_{{m},\sigma}$ for all $\sigma\in\Sigma$, $\Pi_{{m},\sigma}$ being extensions of the Cl\'ement operator as defined in Sec.~\ref{sec:notation}. 
\end{theorem}

\begin{proof}
	As before, let $e:=u-u_\mathrm d \in H^1_{0,\Gamma_D}(\Omega)$. Then from equation (\ref{eq:rewriteerrormulti}), for all $v\in H^1_{0,\Gamma_D}(\Omega)$, 
	\begin{equation} \label{eq:errOmega}
	\sum_{\sigma \in \Sigma} \int_\sigma d_\sigma v \,\mathrm ds = \int_{\Omega} \nabla e \cdot \nabla v \,\mathrm dx \leq |e|_{1,\Omega} |v|_{1,\Omega}.
	\end{equation}
	Now, let $H := \left\{ v\in H_{00}^\frac{1}{2}\left(\Gamma\right) : v|_\sigma \in H_{00}^\frac{1}{2}(\sigma), \text{ for all } \sigma\in\Sigma \right\}$, equipped with the norm $$\|\cdot\|_{H} := \left( \displaystyle\sum_{\sigma\in\Sigma} \|\cdot\|_{H_{00}^{1/2}(\sigma)}^2 \right)^\frac{1}{2},$$
	and let $H^*$ be its dual space equipped with the dual norm $\|\cdot\|_{H^*}$. Recall that $\Omega^\star := \Omega\setminus \overline{F_\mathrm p}$, so that $\Omega = \text{int}\left(\overline{\Omega^\star} \cup \overline{F_\mathrm p}\right)$. So for all $w\in H$, let us define piecewise the function $u_w\in H^1_{0,\partial \Omega\setminus\left(\gamma_\mathrm n \cup \gamma_\setminussign\right)}(\Omega)$ as the unique solution of
	\begin{align*}
	\begin{cases}
	-\Delta \left( u_w|_{F_\mathrm p} \right) = 0 & \text{in } F_\mathrm p \\
	u_w|_{F_\mathrm p} = \left(w\vert_{\gamma_\setminussign\cup\gamma_{0,\mathrm p}}\right)^\star & \text{on } \partial F_\mathrm p, 
	\end{cases}
	\qquad
	\begin{cases}
	-\Delta \left( u_w|_{\Omega^\star} \right) = 0 & \text{in } \Omega^\star \\
	u_w|_{\Omega^\star} = \left(w|_{\gamma_\mathrm n\cup \gamma_{0,\mathrm p}}\right)^\star & \text{on } \partial\Omega^\star, 
	\end{cases}
	\end{align*} 
	where $\left(w|_{\gamma_\setminussign\cup\gamma_{0,\mathrm p}}\right)^\star$ and $\left(w|_{\gamma_\mathrm n\cup \gamma_{0,\mathrm p}}\right)^\star$ are the extensions by $0$ of $w|_{\gamma_\setminussign\cup\gamma_{0,\mathrm p}}$ on $\partial F_\mathrm p$ and of $w|_{\gamma_\mathrm n\cup \gamma_{0,\mathrm p}}$ on $\partial\Omega^\star$, respectively.
	Then by continuity of the solution on the data and from Appendix \ref{lemma:sumH0012}, 
	\begin{align}
	|u_w|_{1,\Omega} &= \left( \left|u_w\right|^2_{1,F_\mathrm p} + \left|u_w\right|^2_{1,\Omega^\star} \right)^\frac{1}{2} \nonumber \\
	&\lesssim \left( \left\| \left(w\vert_{\gamma_\setminussign\cup\gamma_{0,\mathrm p}}\right)^\star \right\|^2_{\frac{1}{2}, \partial F_\mathrm p} + \left\| \left(w|_{\gamma_\mathrm n\cup \gamma_{0,\mathrm p}}\right)^\star \right\|^2_{\frac{1}{2},\partial \Omega^\star}\right)^\frac{1}{2} \nonumber \\
	& = \left( \|w\|_{H_{00}^{1/2}\left(\gamma_\setminussign\cup\gamma_{0,\mathrm p}\right)}^2 + \|w\|_{H_{00}^{1/2}\left(\gamma_\mathrm n\cup\gamma_{0,\mathrm p}\right)}^2 \right)^\frac{1}{2} \lesssim \|w\|_{H}. \label{eq:H00}
	\end{align}
	So, recalling that by definition, $d_\Sigma|_\sigma = d_\sigma$ on each $\sigma\in\Sigma$, thanks to (\ref{eq:errOmega}) and (\ref{eq:H00}) and since $H^1_{0,\partial \Omega\setminus\left(\gamma_\mathrm n \cup \gamma_\setminussign\right)}(\Omega)\subset H^1_{0,\Gamma_D}(\Omega)$, then
	\begin{align}
	\left\| d_\Sigma\right\|_{H^*} = \sup_{\substack{w\in H\\ w\neq 0}} \frac{\displaystyle \int_{\Gamma} d_\Sigma w \,\mathrm ds}{\|w\|_{H}}
	&\lesssim \sup_{\substack{w\in H\\ w\neq 0}} \frac{\displaystyle \sum_{\sigma\in\Sigma} \int_{\sigma} d_\sigma u_w\,\mathrm ds}{\left|u_w\right|_{1,\Omega}} \nonumber \\
	&\leq \sup_{\substack{v\in H^1_{0,\Gamma_D}\left(\Omega\right)\\ v\neq 0}} \frac{\displaystyle \sum_{\sigma\in\Sigma} \int_{\sigma} d_\sigma v\,\mathrm ds}{\left|v\right|_{1,\Omega}} \leq |e|_{1,\Omega}. \label{eq:decomperrdualOmega}
	\end{align}
	Moreover, using Remark \ref{rmk:estgentilde} if $n=3$, or Remark \ref{rmk:compatcondgen} if $n=2$ and if the flux conservation conditions (\ref{eq:compatcondgen}) are satisfied, then
	$$\mathcal{E}(u_\mathrm d) \lesssim \left( \sum_{\sigma\in\Sigma} |\sigma|^\frac{1}{n-1} \left\|d_\sigma\right\|_{0,\sigma}^2 \right)^\frac{1}{2}.$$
	Therefore, using the triangle inequality, and the fact that $\left|\gamma_\mathrm n\right| \simeq \left|\gamma_\setminussign\right| \simeq \left| \gamma_{0,\mathrm p} \right| \simeq |\Gamma|$, then
	\begin{align*}
	\mathcal E(u_\mathrm d)^2 &\leq \sum_{\sigma \in \Sigma} |\sigma|^\frac{1}{n-1} \left\|\Pi_{m}(d_\sigma)\right\|^2_{0,\sigma} +  \sum_{\sigma \in \Sigma} \left\|d_\sigma - \Pi_{m}(d_\sigma)\right\|^2_{0,\sigma} \\
	& \lesssim |\Gamma|^\frac{1}{n-1} \left\| \Pi_{m}(d_\Sigma)\right\|^2_{0,{\Gamma}} + |\Gamma|^\frac{1}{n-1} \left\|d_\Sigma - \Pi_{m}(d_\Sigma)\right\|^2_{0,{\Gamma}}.
	\end{align*}
	Now, we use the definition of the broken norm in $H^*$ to apply the inverse inequality of Appendix \ref{lemma:inverseineqH}. Recalling the definition (\ref{eq:oscgen}) of the oscillations, and using again the triangle inequality, we thus obtain
	\begin{align} 
	\mathcal E(u_\mathrm d)^2 &\lesssim \left\|\Pi_m \left( d_\Sigma\right)\right\|_{H^*}^{{2}} + \text{osc}\left(u_\mathrm d\right)^2 \nonumber \\
	&\lesssim \big[ \left\|d_\Sigma\right\|_{H^*} + \left\|\Pi_m \left( d_\Sigma\right) - d_\Sigma\right\|_{H^*} + \text{osc}\left(u_\mathrm d\right) \big]^2.\label{eq:finalestimation}
	\end{align}
	Furthermore, applying Appendix \ref{lemma:sumH0012} and then Appendix \ref{lemma:inverseineq2}, we have
	\begin{align}
	\left\|\Pi_m \left( d_\Sigma\right) - d_\Sigma\right\|_{H^*} &\lesssim \left\|\Pi_m \left( d_\Sigma\right) - d_\Sigma\right\|_{H_{00}^{-1/2}(\Gamma)} \nonumber \\
	&\lesssim |\Gamma|^\frac{1}{2(n-1)} \left\|\Pi_m \left( d_\Sigma\right) - d_\Sigma\right\|_{0,\Gamma} = \text{osc}\left(u_\mathrm d\right).  \label{eq:oscillationsgenderive}
	\end{align}
	To conclude, we plug in (\ref{eq:decomperrdualOmega}) and (\ref{eq:oscillationsgenderive}) into equation (\ref{eq:finalestimation}), and thus
	\begin{align*}
	\mathcal E(u_\mathrm d) \lesssim |e|_{1,\Omega} +  \mathrm{osc}(u_\mathrm d). 
	\end{align*}
\end{proof}

\begin{remark}
	As in Remark \ref{rmk:oscsmall}, when the data is regular, it is always possible to choose
	$m$ large enough so that the asymptotic behavior of the oscillations is $\mathcal{O}\left(|\Gamma|^{m+\frac{1}{2(n-1)}}\right)$. Therefore, we can make sure that the oscillations get small with respect to the defeaturing error, when the feature gets small. 
\end{remark}

\section{Numerical  considerations and experiments} \label{sec:numexp}
From the definition of the \textit{a posteriori} defeaturing error estimator (\ref{eq:genestimator}) in the general case, to estimate the error introduced by defeaturing the problem geometry, we only need to perform the following steps.
\begin{enumerate}
	\item[(i)] Choose the Neumann data $g_0$ and solve the defeatured problem (\ref{eq:simplpb}).
	\item[(ii)] For the positive component $ F_\mathrm p$ of the feature $F$, choose the Neumann data $\tilde g$ and solve the local extension problem (\ref{eq:featurepbonF}). However, features may be geometrically complex, and the solution of the extension problem an unwanted burden. Therefore, instead of (\ref{eq:featurepbonF}), one can solve the extension problem (\ref{eq:featurepb}) in a chosen (simple) domain $\tilde F_\mathrm p$ that contains $F_\mathrm p$ and such that $\gamma_{0,\mathrm p} \subset \left( \partial \tilde F_\mathrm p \cap \partial F_\mathrm p\right)$.
	\item[(iii)] Compute the boundary averages and integrals 
	$\overline{d_\sigma}^\sigma$ and $\left\|d_\sigma - \overline{d_\sigma}^\sigma \right\|_{0,\sigma}$
	for each $\sigma\in\Sigma$, as defined in (\ref{eq:dsigma}). That is, we suitably evaluate the error made on the normal derivative of the solution on specific parts of the boundaries of the features. 
\end{enumerate} 

In the remaining part of the paper, we present a few numerical examples to illustrate the validity of our defeaturing error estimator. All the numerical experiments presented in the following section have been implemented in GeoPDEs \cite{geopde}, an open-source and free Octave/Matlab package for the resolution of partial differential equations specifically designed for isogeometric analysis \cite{igabook}. For the geometric description of the features and the local meshing process required, multipatch and trimming techniques have been used \cite{antolinvreps,weimarrusigantolin}. Moreover, a rather fine mesh is used in order to neglect the error due to the numerical approximation. 

\subsection{Impact of some  properties of the feature on the defeaturing error} \label{sec:impactfeat}
While validating the theory developed in Sec.~\ref{sec:negativeest} and Sec.~\ref{sec:complexfeat}, we study the impact of the shape and the size of a feature on the defeaturing error and estimator, and of the choice of the defeatured Neumann data. Moreover, as the estimator depends upon the size of the features and the size of the solution gradients ``around" the feature, we will be able to show an example where small features count more than big ones. 

\subsubsection{Feature shape} \label{sec:featshape}
In this example, we compare the behavior of the error and the estimator on the same Poisson problem in three different geometries: one with a star-shaped feature, another one with a circular feature, and the last one with a squared feature. Let $$\Omega_0 := \left\{ \big( r\cos(\theta), r\sin(\theta) \big) \in \mathbb R^2 : 0\leq r<1, 0\leq \theta \leq 2\pi \right\},$$ let $\Omega_\star:= \Omega_0\setminus \overline{F_\star}$, $\Omega_c:= \Omega_0\setminus \overline{F_c}$ and $\Omega_s:=\Omega_0\setminus \overline{F_s}$, with 
\begin{itemize}
	\item $F_\star$ the $10$-branch regular star of inner radius $r_\star>0$, outer radius $2r_\star$, centered in $(0,0)$,
	\item $F_c$ the circle of radius $r_c>0$, centered in $(0,0)$,
	\item $F_s$ the square of side length $2r_s>0$, centered in $(0,0)$,
\end{itemize}
as in Fig.~\ref{fig:shapeholes}. 

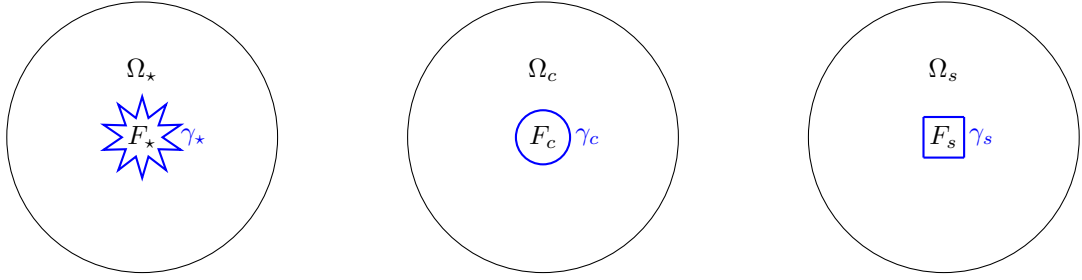
\begin{figure}
	\centering
	\begin{subfigure}{0.3\textwidth}
		\begin{center}
			\begin{tikzpicture}[scale=1.8]
			\draw (0.5,0.5) circle(1) ;
			\draw (0.5, 0.5) node{$F_\star$} ;
			\draw (0.5, 1) node{$\Omega_\star$} ;
			\draw[blue] (0.88, 0.5) node{$\gamma_\star$} ;
			\tstar{0.15}{0.3}{10}{0}{blue,shift={(0.5,0.5)},thick};
			\end{tikzpicture}
			\caption{Domain with a star feature.}
		\end{center}
	\end{subfigure}
	~
	\begin{subfigure}{0.3\textwidth}
		\begin{center}
			\begin{tikzpicture}[scale=1.8]
			\draw (0.5,0.5) circle(1) ;
			\draw[blue,thick] (0.5,0.5) circle(0.2) ;
			\draw (0.5, 0.5) node{$F_c$} ;
			\draw (0.5, 1) node{$\Omega_c$} ;
			\draw[blue] (0.83, 0.5) node{$\gamma_c$} ;
			\end{tikzpicture}
			\caption{Domain with a circle feature.} \label{fig:fouriergeom}
		\end{center}
	\end{subfigure}
	~
	\begin{subfigure}{0.3\textwidth}
		\begin{center}
			\begin{tikzpicture}[scale=1.8]
			\draw (0.5,0.5) circle(1) ;
			\draw[blue,thick] (0.35, 0.35)--(0.35,0.65) ;
			\draw[blue,thick] (0.35, 0.65)--(0.65,0.65) ;
			\draw[blue,thick] (0.65, 0.65)--(0.65,0.35) ;
			\draw[blue,thick] (0.65,0.35)--(0.35, 0.35);
			\draw (0.5, 0.5) node{$F_s$} ;
			\draw (0.5, 1) node{$\Omega_s$} ;
			\draw[blue] (0.78, 0.5) node{$\gamma_s$} ;
			\end{tikzpicture}
			\caption{Domain with a square feature.}
		\end{center}
	\end{subfigure}
	\caption{Comparison between feature shapes.}  \label{fig:shapeholes}
\end{figure}

We choose $r_\star, r_c,r_s>0$ such that $F_\star$, $F_c$ and $F_s$ have, first, the same area, and then, the same perimeter. 
We consider Poisson problem (\ref{eq:originalpb}) solved in $\Omega_\star$, $\Omega_c$ and in $\Omega_s$, and its defeatured version (\ref{eq:simplpb}). We take $f\equiv 1$ in $\Omega_0$, $h\equiv 0$ on $\Gamma_D := \partial \Omega_0$, and $g\equiv 0$ on $\partial F_\star$, on $\partial F_c$ and on $\partial F_s$. 

\begin{table}
	\centering
	{\bgroup
		\def\arraystretch{1.2}
		\begin{tabular}{@{}ccccccc@{}} 
			\hline
			\\[-1em]
			Domain $\Omega$ & ${\let\scriptstyle\textstyle\substack{\text{Perimeter} \\\text{of } F}}$ & Area of $F$ & $\mathcal{E}(u_0)$ & $|u-u_0|_{1,\Omega}$ & {$\displaystyle\frac{|u-u_0|_{1,\Omega}}{|u_0|_{0,\Omega}}$} & Eff. index \\ [-1em]
			\\
			\hline
			$\Omega_\star$, $r_\star=1.83\cdot 10^{-2}$ & $0.400$ & $2.07\cdot 10^{-3}$ & $1.98\cdot 10^{-3}$ & $1.56\cdot 10^{-3}$ & {$2.49\cdot 10^{-3}$} & $1.27$ \\ 
			$\Omega_c$, $r_c=6.37\cdot 10^{-2}$ & $0.400$ & $1.27\cdot 10^{-2}$ & $1.21\cdot 10^{-2}$ & $8.42\cdot 10^{-3}$ & {$1.34\cdot 10^{-2}$} & $1.45$ \\ 
			$\Omega_s$, $r_s=5.00\cdot 10^{-2}$ & $0.400$ & $1.00\cdot 10^{-2}$ & $9.57\cdot 10^{-3}$ & $6.74\cdot 10^{-3}$ & {$1.07\cdot 10^{-2}$} & $1.42$ \\ 
			$\Omega_c$, $r_c=5.64\cdot 10^{-2}$ & $0.355$ & $1.00\cdot 10^{-2}$ & $1.01\cdot 10^{-2}$ & $6.76\cdot 10^{-3}$ & {$1.08\cdot 10^{-2}$} & $1.51$ \\ 
			$\Omega_\star$, $r_\star=4.02\cdot 10^{-2}$ & $0.880$ & $1.00\cdot 10^{-2}$ & $7.53\cdot 10^{-3}$ & $6.65\cdot 10^{-3}$ & {$1.38\cdot 10^{-2}$} & $1.13$ \\ 
			\hline
	\end{tabular}
	\egroup}
	\caption{Results of the comparison between feature shapes. \label{table:compshape}}
\end{table}

The results are summarized in Table \ref{table:compshape}. 
We can see that in all the cases, the larger the area of the feature, the larger the defeaturing error and estimator. Moreover, the effectivity index only changes slightly when considering the same feature but with different measures: this shows that it is indeed independent from the measure of the considered feature and its boundary. The small change in the effectivity index is only due to numerical approximation, the solutions not being exact but being obtained on a very fine mesh. 
Furthermore, the shape of the feature does not impact much the defeaturing estimator: 
we do not observe any major difference between the smooth feature (the circle), the convex non-smooth Lipschitz feature (the square), and the non-convex non-smooth Lipschitz feature (the star). Our theory indeed treats those different types of geometries in the same way. {Finally, even if the estimator is referred to the absolute error, both the relative and the absolute errors are given to be able to quantify the magnitude of the defeaturing effect.}

\subsubsection{Feature size} \label{sec:featsize}
Removing a small feature where the solution of the PDE has a high gradient can significantly increase the defeaturing error, while the error might almost not be affected when removing a large feature where the solution of the PDE is nearly constant. The following example shows that our estimator is also able to capture this. Let $\Omega_0 := (0,1)^2$ and $\Omega := \Omega_0 \setminus \left(\overline{F^1} \cup \overline{F^2}\right)$, where $F^1$ and $F^2$ are circles of two different sizes given by
\begin{align*}
F^1 := \left\{ (1.1\cdot 10^{-3}, 1.1\cdot 10^{-3}) + \big( r\cos(\theta), r\sin(\theta) \big) \in \mathbb R^2 : \right.\quad& \\
\left.0\leq r<10^{-3}, 0\leq \theta \leq 2\pi \right\},& \\
F^2 := \left\{ (8.9\cdot 10^{-1}, 8.9\cdot 10^{-1}) + \big( r\cos(\theta), r\sin(\theta) \big) \in \mathbb R^2 : \right.\quad& \\
\left.0\leq r<10^{-1}, 0\leq \theta \leq 2\pi \right\},&
\end{align*}
similarly as in Fig.~\ref{fig:shapesize}. We consider Poisson problem (\ref{eq:originalpb}) solved in $\Omega$, and its defeatured version (\ref{eq:simplpb}) in $\Omega_0$. Let $f(x,y) := -128e^{-8(x+y)}$ in $\Omega_0$, $h (x,y) := e^{-8(x+y)}$ on 
$$\Gamma_D := \big\{ (x,0), (0,y)\in \mathbb R^2: 0\leq x,y < 1 \big\},$$
the bottom and left sides, $g(x,y):= -8e^{-8(x+y)}$ on $\partial \Omega_0 \setminus \overline{\Gamma_D}$, and finally $g \equiv 0$ on $\partial F^1 \cup \partial F^2$. Since the geometry contains two features, we call $\mathcal E^1$ and $\mathcal E^2$ the defeaturing estimators defined in (\ref{eq:negestimator}) and computed, respectively, on the boundary of $F^1$ and on the boundary of $F^2$, and we consider the sum of $\mathcal E^1$ and $\mathcal E^2$ as the total defeaturing estimator $\mathcal E$. 

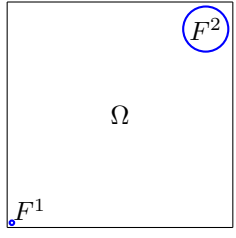
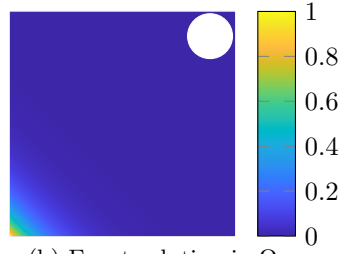
\begin{figure}
	\centering
	\begin{subfigure}{0.45\textwidth}
		\begin{center}
			\begin{tikzpicture}[scale=3]
			\draw (0,0) -- (1,0) ;
			\draw (0,0) -- (0,1) ;
			\draw (1,0) -- (1,1) ;
			\draw (0,1) -- (1,1) ;
			\draw[blue,thick] (0.021, 0.021) circle (0.01);
			\draw (0.1,0.08) node{$F^1$};
			\draw[blue,thick] (0.88,0.88) circle (0.1);
			\draw (0.88,0.88) node{$F^2$};
			\draw (0.5,0.5) node{$\Omega$};
			\end{tikzpicture}
			\caption{Exact domain $\Omega$ with two features (not at scale).}\label{fig:shapesize}
		\end{center}
	\end{subfigure}
	~
	\begin{subfigure}{0.45\textwidth}
		\begin{center}
			\vspace{-12cm}
			\include{data/featdiffsize}
			\vspace{-0.8cm}
			\caption{Exact solution in $\Omega$.}\label{fig:exactsolshapesize}
		\end{center}
	\end{subfigure}
	\caption{Geometry with two features of different size and exact solution.}
\end{figure}

\begin{table}
	\centering
	{
		\def\arraystretch{1.2}
		\begin{tabular}{@{}cccccc@{}} 
			\hline \\ [-1em]
			$\mathcal{E}^1(u_0)$ & $\mathcal{E}^2(u_0)$ & $\mathcal{E}(u_0)$ & $\big|u-u_0\big|_{1,\Omega}$ & {$\displaystyle\frac{|u-u_0|_{1,\Omega}}{|u_0|_{0,\Omega}}$} & Effectivity index \\ [-1em]
			\\
			\hline
			$5.03\cdot 10^{-2}$ & $7.86\cdot 10^{-6}$ & $5.03\cdot 10^{-2}$ & $1.45\cdot 10^{-2}$ & {$2.05\cdot 10^{-2}$} & $3.47$ \\
			\hline
	\end{tabular}}
	\caption{Results of the comparison between feature sizes. \label{table:sizefeat}}
\end{table}

With this choice, the solution to Poisson problem has a very high gradient near feature $F^1$, and it is almost constantly zero near feature $F^2$, as we can observe in Fig.~\ref{fig:exactsolshapesize}. Therefore, one can expect the presence of $F^1$ to be more important than $F^2$ with respect to the solution accuracy, even if $F^1$ is notably smaller than $F^2$. The results are presented in Table \ref{table:sizefeat}, where we can see that this is indeed the case: the estimator on $F^2$ is four orders of magnitude smaller than the estimator on $F^1$, even if the radius of $F^1$ is two orders of magnitude smaller than the one of $F^2$. This confirms the fact that our estimator as written in (\ref{eq:genestimator}) correctly trades off the measure of the features and their position in the geometrical domain, in order to correctly assess the impact of defeaturing on the solution. 

{Finally, and as in the previous numerical experiment of Sec.~\ref{sec:featsize}, both the relative error and the absolute error are given to be able to quantify the magnitude of the defeaturing effect. In the following, we will be interested in the convergence of the error and estimator with respect to the size of the feature. Since the relative error is a scaling of the absolute error, the convergence will be the same whether one considers the relative or the absolute error. Moreover, since the magnitude of the error depends on the problem at hand (geometries, size of the feature, and PDE data), and since the derived estimator is referred to the absolute error, we will only look at the absolute defeaturing error in the next experiments.}

\subsection{Error convergence with respect to the feature size} \label{sec:errcv}
We now analyze the convergence of our estimator with respect to the size of the feature and we compare it with the convergence of the defeaturing error. Moreover, we show an example in which the choice of the defeatured problem data influences drastically the convergence of both the estimator and the defeaturing error. 

\subsubsection{Two-dimensional geometries} \label{sec:cv2d}
We begin with two-dimensional examples of geometries with a negative feature. For $k=0,1,\ldots,6$, let $\varepsilon = \displaystyle\frac{10^{-2}}{2^k}$, and let $\Omega_\varepsilon^i := \Omega_0 \setminus \overline{F_\varepsilon^i}$ for $i=1,2$ with $\Omega_0 := (0,1)^2$ and 
\begin{align*}
F_\varepsilon^1 &:= \left\{ \big( 0.5+r\cos(\theta), 1+r\sin(\theta) \big) \in \mathbb R^2 : 0\leq r<\varepsilon, -\pi< \theta < 0 \right\}, \\
F_\varepsilon^2 &:= (1-\varepsilon,1)^2,
\end{align*}
as in Fig.~\ref{fig:geomcv2Dneg} and Fig.~\ref{fig:geomcv2Dnegangle}. For $i=1,2$, we consider Poisson problem (\ref{eq:originalpb}) solved in $\Omega_\varepsilon^i$, and its defeatured version (\ref{eq:simplpb}) in $\Omega_0$. We take $f(x,y) := 10\cos(3\pi x)\sin(5\pi y)$ in $\Omega_0$, $h\equiv 0$ on 
$$\Gamma_D := \big\{ (x,0)\in \mathbb R^2: 0< x < 1 \big\},$$ 
$g\equiv 0$ on $\Gamma_N := \partial \Omega_\varepsilon^i \setminus \overline{\Gamma_D}$, and $g_0 \equiv 0$ on $\partial \Omega_0 \setminus \partial \Omega_\varepsilon^i$. We respectively call $u^{(i)}$ and $u_0^{(i)}$ the exact and defeatured solutions. 

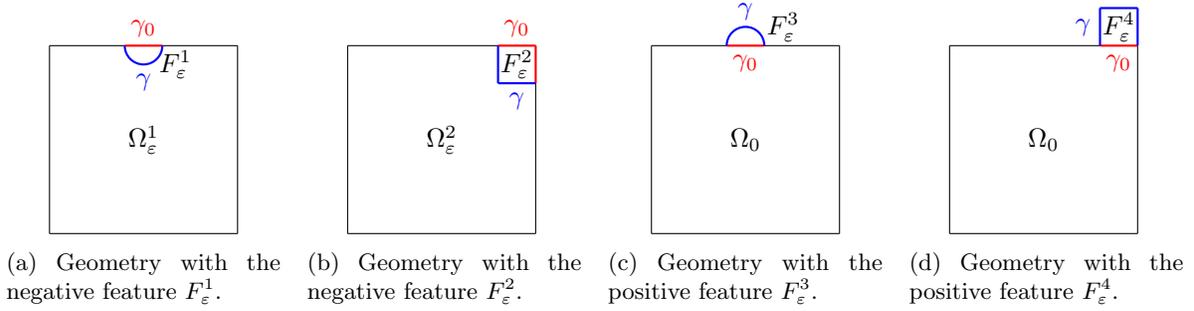
\begin{figure}
	\centering
	\begin{subfigure}{0.22\textwidth}
		\begin{center}
			\begin{tikzpicture}[scale=2.5]
			\draw (2,0) -- (3,0) ;
			\draw (2,0) -- (2,1) ;
			\draw (3,0) -- (3,1) ;
			\draw (2,1) -- (2.4,1) ;
			\draw (2.6,1) -- (3,1) ;
			\draw[red,thick] (2.4,1) -- (2.6,1) ;
			\draw[red] (2.5,1) node[above]{$\gamma_0$} ;
			\draw[blue,thick] (2.4, 1) arc (-180:0:.1cm);
			\draw (2.5,0.5) node{$\Omega_\varepsilon^1$} ;
			\draw (2.67,0.9) node{$F_\varepsilon^1$} ;
			\draw[blue] (2.5,0.9) node[below]{$\gamma$} ;
			\draw (2.5,1.3) node{} ;
			\end{tikzpicture}
			\caption{Geometry with the negative feature $F_\varepsilon^1$.}\label{fig:geomcv2Dneg}
		\end{center}
	\end{subfigure}
	~
	\begin{subfigure}{0.22\textwidth}
		\begin{center}
			\begin{tikzpicture}[scale=2.5]
			\draw (2,0) -- (3,0) ;
			\draw (2,0) -- (2,1) ;
			\draw (3,0) -- (3,0.8) ;
			\draw (2,1) -- (2.8,1) ;
			\draw[red,thick] (2.8,1) -- (3,1) ;
			\draw[red,thick] (3,0.8) -- (3,1) ;
			\draw[red] (2.9,1) node[above]{$\gamma_0$} ;
			\draw[blue,thick] (2.8,0.8) -- (2.8,1);
			\draw[blue,thick] (2.8,0.8) -- (3,0.8);
			\draw (2.5,0.5) node{$\Omega_\varepsilon^2$} ;
			\draw (2.9,0.9) node{$F_\varepsilon^2$} ;
			\draw[blue] (2.9,0.8) node[below]{$\gamma$} ;
			\draw (2.5,1.3) node{} ;
			\end{tikzpicture}
			\caption{Geometry with the negative feature $F_\varepsilon^2$.}\label{fig:geomcv2Dnegangle}
		\end{center}
	\end{subfigure}
	~
	\begin{subfigure}{0.22\textwidth}
		\begin{center}
			\begin{tikzpicture}[scale=2.5]
			\draw (2,0) -- (3,0) ;
			\draw (2,0) -- (2,1) ;
			\draw (3,0) -- (3,1) ;
			\draw (2,1) -- (2.4,1) ;
			\draw (2.6,1) -- (3,1) ;
			\draw[red,thick] (2.4,1) -- (2.6,1) ;
			\draw[red] (2.5,1) node[below]{$\gamma_0$} ;
			\draw[blue,thick] (2.6, 1) arc (0:180:.1cm);
			\draw (2.5,0.5) node{$\Omega_0$} ;
			\draw (2.7,1.1) node{$F_\varepsilon^3$} ;
			\draw[blue] (2.5,1.18) node{$\gamma$} ;
			\draw (2.5,1.3) node{} ;
			\end{tikzpicture}
			\caption{Geometry with the positive feature $F_\varepsilon^3$.}\label{fig:geomcv2Dpos}
		\end{center}
	\end{subfigure}
	~
	\begin{subfigure}{0.22\textwidth}
		\begin{center}
			\begin{tikzpicture}[scale=2.5]
			\draw (2,0) -- (3,0) ;
			\draw (2,0) -- (2,1) ;
			\draw (3,0) -- (3,1) ;
			\draw (2,1) -- (2.8,1) ;
			\draw[red,thick] (2.8,1) -- (3,1) ;
			\draw[red] (2.9,1) node[below]{$\gamma_0$} ;
			\draw[blue,thick] (2.8,1.2) -- (2.8,1);
			\draw[blue,thick] (2.8,1.2) -- (3,1.2);
			\draw[blue,thick] (3,1.2) -- (3,1);
			\draw (2.5,0.5) node{$\Omega_0$} ;
			\draw (2.9,1.1) node{$F_\varepsilon^4$} ;
			\draw[blue] (2.8,1.1) node[left]{$\gamma$} ;
			\draw (2.5,1.18) node{} ;
			\draw (2.5,1.3) node{} ;
			\end{tikzpicture}
			\caption{Geometry with the positive feature $F_\varepsilon^4$.}\label{fig:geomcv2Dposangle}
		\end{center}
	\end{subfigure}
	\caption{$2$D geometries $\Omega_\varepsilon^i$, $i=1,2,3,4$.} \label{fig:geomcv2D}
\end{figure}

The results are presented in Fig.~\ref{fig:cv2Dneg}. Both the error and the estimator converge with respect to the size of the feature as $\varepsilon\propto |\gamma|$ in the first geometry $\Omega_\varepsilon^1$, and as $\varepsilon^2 \propto |\gamma|^2$ in the second geometry $\Omega_\varepsilon^2$. Clearly, the difference in asymptotic behavior of the error depends on symmetries and on the Neumann boundary conditions. Indeed, $\Omega_\varepsilon^2$ has features with sides parallel to $\partial \Omega_0$. Moreover, the effectivity index is indeed independent from the size of the feature since it remains nearly equal to $1.81$ and $1.78$, respectively, and for all values of $\varepsilon$. That is, as predicted by the theory since the estimator is both reliable (Theorem \ref{thm:upperbound}) and efficient up to oscillations (Theorem \ref{thm:lowerbound}), here in dimension two, the dependence of the estimator with respect to the size of the feature is explicit. \\

\begin{figure}
	\centering
	\begin{subfigure}{0.48\textwidth}
		\begin{center}
			\begin{tikzpicture}[scale=0.7]
			\begin{axis}[xmode=log, ymode=log, legend style={at={(0.5,-0.15)}, legend columns=3, anchor=north, draw=none}, xlabel=$\varepsilon$, ylabel=Error]
			\addplot[mark=x, red, thick] table [x=eps, y=errh1s, col sep=comma] {data/cv2dneg.csv};
			\addplot[mark=x, blue, densely dashed, thick, mark options=solid] table [x=eps, y=est, col sep=comma] {data/cv2dneg.csv};
			\addplot[mark=none, loosely dotted, thick] table [x=eps, y=epsdiv, col sep=comma] {data/cv2dpos.csv};
			\addplot[mark=o, red, thick] table [x=eps, y=errh1s, col sep=comma] {data/cv2dnegangle.csv};
			\addplot[mark=o, blue, densely dashed, thick, mark options=solid] table [x=eps, y=est, col sep=comma] {data/cv2dnegangle.csv};
			\addplot[mark=none, thick, densely dotted] table [x=eps, y=eps2, col sep=comma] {data/cv2dnegangle.csv};
			\legend{$\left|u^{(1)}-u_0^{(1)}\right|_{1,\Omega_\varepsilon^1}$,$\mathcal E\Big(u_0^{(1)}\Big)$,$\mathcal O(\varepsilon)$,$\left|u^{(2)}-u_0^{(2)}\right|_{1,\Omega_\varepsilon^2}$,$\mathcal E\Big(u_0^{(2)}\Big)$,$\mathcal O\left(\varepsilon^2\right)$};
			\end{axis}
			\end{tikzpicture}
			\caption{Geometries $\Omega_\varepsilon^1$ and $\Omega_\varepsilon^2$ with a negative feature.} \label{fig:cv2Dneg}
		\end{center}
		\vspace{0.2cm}
	\end{subfigure}
	~
	\begin{subfigure}{0.48\textwidth}
		\begin{center}
			\begin{tikzpicture}[scale=0.7]
			\begin{axis}[xmode=log, ymode=log, legend style={at={(0.5,-0.15)}, legend columns=3, anchor=north, draw=none}, xlabel=$\varepsilon$, ylabel=Error]
			\addplot[mark=x, red, thick] table [x=eps, y=errh1s, col sep=comma] {data/cv2dpos.csv};
			\addplot[mark=x, blue, densely dashed, thick, mark options=solid] table [x=eps, y=est, col sep=comma] {data/cv2dpos.csv};
			\addplot[mark=none, loosely dotted, thick] table [x=eps, y=epsdiv, col sep=comma] {data/cv2dpos.csv};
			\addplot[mark=o, red, thick] table [x=eps, y=err, col sep=comma] {data/cv2dposangle.csv};
			\addplot[mark=o, blue, densely dashed, thick, mark options=solid] table [x=eps, y=est, col sep=comma] {data/cv2dposangle.csv};
			\addplot[mark=none, thick, densely dotted] table [x=eps, y=eps2, col sep=comma] {data/cv2dnegangle.csv};
			\legend{$\left|u^{(3)}-u_\mathrm d^{(3)}\right|_{1,\Omega_\varepsilon^3}$, $\mathcal E\Big(u_\mathrm d^{(3)}\Big)$,$\mathcal O(\varepsilon)$, $\left|u^{(4)}-u_\mathrm d^{(4)}\right|_{1,\Omega_\varepsilon^4}$, $\mathcal E\Big(u_\mathrm d^{(4)}\Big)$, $\mathcal O\left(\varepsilon^2\right)$};
			\end{axis}
			\end{tikzpicture}
			\caption{Geometries $\Omega_\varepsilon^3$ and $\Omega_\varepsilon^4$ with a positive feature.}\label{fig:cv2Dpos}
		\end{center}
		\vspace{0.2cm}
	\end{subfigure}
	~
	\begin{subfigure}{0.48\textwidth}
		\begin{center}
			\begin{tikzpicture}[scale=0.7]
			\begin{axis}[xmode=log, ymode=log, legend style={at={(0.5,-0.15)}, legend columns=3, anchor=north, draw=none}, xlabel=$\varepsilon$, ylabel=Error]
			\addplot[mark=o, red, thick] table [x=eps, y=errh1s, col sep=comma] {data/cv2dgena_data2.csv};
			\addplot[mark=o, blue, densely dashed, thick, mark options=solid] table [x=eps, y=est, col sep=comma] {data/cv2dgena_data2.csv};
			\addplot[mark=none, loosely dotted, thick] table [x=eps, y=epsdiv, col sep=comma] {data/cv2dgenb_data2.csv};
			\legend{$\left|u^{(5)}-u_\mathrm d^{(5)}\right|_{1,\Omega_\varepsilon^5}$, $\mathcal E\Big(u_\mathrm d^{(5)}\Big)$,$\mathcal O(\varepsilon)$};
			\end{axis}
			\end{tikzpicture}
			\caption{Geometry $\Omega_\varepsilon^5$ with a general complex feature.}\label{fig:cv2Dgen}
		\end{center}
	\end{subfigure}
	~
	\begin{subfigure}{0.48\textwidth}
		\begin{center}
			\begin{tikzpicture}[scale=0.7]
			\begin{axis}[xmode=log, ymode=log, legend style={at={(0.5,-0.15)}, legend columns=3, anchor=north, draw=none}, xlabel=$\varepsilon$, ylabel=Error]
			\addplot[mark=x, red, thick] table [x=eps, y=errh1s, col sep=comma] {data/cv2dgenb_data2.csv};
			\addplot[mark=x, blue, densely dashed, thick, mark options=solid] table [x=eps, y=est, col sep=comma] {data/cv2dgenb_data2.csv};
			\addplot[mark=none, densely dotted, thick] table [x=eps, y=epsdiv, col sep=comma] {data/cv2dgenb_data2.csv};
			\legend{$\left|u^{(6)}-u_\mathrm d^{(6)}\right|_{1,\Omega_\varepsilon^6}$, $\mathcal E\Big(u_\mathrm d^{(6)}\Big)$, $\mathcal O(\varepsilon)$};
			\end{axis}
			\end{tikzpicture}
			\caption{Geometry $\Omega_\varepsilon^6$ with a general complex feature.}\label{fig:cv2Dgenb}
		\end{center}
	\end{subfigure}
	\caption{Convergence of the error and of the estimator in $2$D domains with one feature.} \label{fig:cv2D}
\end{figure}

Let us now consider two-dimensional examples of geometries with a positive feature. Let $\Omega_0$, $\Gamma_D$, $f$, $h$ and $g$ be as before, and let $\Omega_\varepsilon^j := \textrm{int}\left(\overline{\Omega_0} \cup \overline{F_\varepsilon^j}\right)$ for $j=3,4$ with 
\begin{align*}
F_\varepsilon^3 &:= \left\{ \big( 0.5+r\cos(\theta), 1+r\sin(\theta) \big) \in \mathbb R^2 : 0\leq r<\varepsilon, 0< \theta < \pi \right\},\\
F_\varepsilon^4 &:=(1-\varepsilon,1)\times (1,1+\varepsilon),
\end{align*}
as in Fig.~\ref{fig:geomcv2Dpos} and Fig.~\ref{fig:geomcv2Dposangle}. Let $\Gamma_N := \partial \Omega_\varepsilon^j \setminus \overline{\Gamma_D}$. 
For each $j=3,4$, we consider the same Poisson problem (\ref{eq:originalpb}) as before, but solved in $\Omega_\varepsilon^j$. We also solve its defeatured version (\ref{eq:simplpb}) in $\Omega_0$ with $g_0 \equiv 0$ on $\partial \Omega_0 \setminus \partial \Omega_\varepsilon^j$. Then we extend the defeatured solution to $F_\varepsilon^j$ by (\ref{eq:featurepb}) with $\tilde F := F_\varepsilon^j$. We respectively call $u^{(j)}$ and $u_0^{(j)}$ the exact and defeatured solutions, and $u_\mathrm d^{(j)}$ the defeatured solution extended to $F_\varepsilon^j$. 

The results are presented in Fig.~\ref{fig:cv2Dpos}. As for the negative feature case, the error in $\Omega_0$ and the estimator converge with respect to the size of the feature as $\varepsilon\propto \left|\gamma_0\right|$ in the first geometry $\Omega_\varepsilon^3$, and as $\varepsilon^2 \propto |\gamma|^2$ in the second geometry $\Omega_\varepsilon^4$. Again, the difference in asymptotic behavior of the error depends on symmetries and on the Neumann boundary conditions. Indeed, $\Omega_\varepsilon^4$ has features with sides parallel to $\partial \Omega_0$. Moreover, the effectivity index is indeed almost independent from the size of the feature since it remains nearly equal to $2.93$ and $3.22$, respectively, for all values of $\varepsilon$. That is, as predicted by the theory since the estimator is both reliable (Theorem \ref{thm:genupperbound}) and efficient up to oscillations (Theorem \ref{thm:genlowerbound}), here in dimension two, the dependence of the estimator with respect to the size of the feature is explicit. We also remark that the effectivity indices for the positive features are little bit larger than the ones for the negative features.\\

\begin{figure}
	\begin{center}
		\begin{subfigure}{0.23\textwidth}
			\begin{center}
				\begin{tikzpicture}[scale=2.8]
				\fill[orange!50!black, fill, opacity=0.3] (0.4,1) rectangle (0.5,1.1);
				\fill[pattern=north west lines, pattern color=green!50!black, opacity=0.3] (0.5,1) rectangle (0.6,0.9);
				\draw[gray] (0,0) -- (1,0) ;
				\draw (0,0) -- (0,1) ;
				\draw (1,0) -- (1,1) ;
				\draw (0,1) -- (0.4,1) ;
				\draw (0.6,1) -- (1,1) ;
				\draw (0.4,1.1)--(0.5,1.1);
				\draw (0.4,1.1)--(0.4,1);
				\draw (0.5,1.1)--(0.5,1);
				\draw (0.5,0.9)--(0.6,0.9);
				\draw (0.5,0.9)--(0.5,1);
				\draw (0.6,0.9)--(0.6,1);
				\draw (0.5,0.5) node{$\Omega_\varepsilon^5$} ;
				\draw[orange!50!black] (0.45,1.08) node[above]{$F_{\mathrm p,\varepsilon}^5$} ;
				\draw[green!50!black] (0.55,0.9) node[below]{$F_{\mathrm n,\varepsilon}^5$} ;
				\end{tikzpicture}
				\caption{Exact domain $\Omega_\varepsilon^5$.}
			\end{center}
		\end{subfigure}
		~
		\begin{subfigure}{0.23\textwidth}
			\begin{center}
				\begin{tikzpicture}[scale=4.6]
				\draw (0.25,1.3) -- (0.75,1.3) ;
				\draw[red,thick] (0.4,1.3) -- (0.5,1.3) ;
				\draw[red,thick] (0.4,1.28) -- (0.4,1.32) ;
				\draw[red,thick] (0.45,1.32) node[above]{$\gamma_{0,\mathrm p}$} ;
				\draw[red,thick] (0.5,1.3) -- (0.6,1.3) ;
				\draw[red,thick] (0.5,1.28) -- (0.5,1.32) ;
				\draw[red,thick] (0.6,1.28) -- (0.6,1.32) ;
				\draw[red,thick] (0.55,1.28) node[below]{$\gamma_{0, \mathrm n}$} ;
				\draw (0.25,1) -- (0.4,1) ;
				\draw[blue] (0.48,1) -- (0.52, 1) ;
				\draw (0.6,1) -- (0.75,1) ;
				\draw[blue,thick] (0.4,1.1)--(0.5,1.1);
				\draw[blue,thick] (0.4,1.1)--(0.4,1);
				\draw[blue,thick] (0.5,1.1)--(0.5,1);
				\draw[blue,thick] (0.5,0.9)--(0.6,0.9);
				\draw[blue,thick] (0.5,0.9)--(0.5,1);
				\draw[blue,thick] (0.6,0.9)--(0.6,1);
				\draw[blue] (0.45,1.1) node[above]{$\gamma_\mathrm p$} ;
				\draw[blue] (0.55,0.9) node[below]{$\gamma_\mathrm n$} ;
				\end{tikzpicture}
				\caption{Zoom on part of the upper boundary of $\Omega_0$ (up) and $\Omega_\varepsilon^5$ (down).}
			\end{center}
		\end{subfigure}
		~
		\begin{subfigure}{0.23\textwidth}
			\begin{center}
				\begin{tikzpicture}[scale=2.8]
				\fill[orange!50!black, fill, opacity=0.3] (0.425,1) rectangle (0.525,1.1);
				\fill[pattern=north west lines, pattern color=green!50!black, opacity=0.3] (0.475,1) rectangle (0.575,0.9);
				\draw[gray] (0,0) -- (1,0) ;
				\draw (0,0) -- (0,1) ;
				\draw (1,0) -- (1,1) ;
				\draw (0,1) -- (0.425,1) ;
				\draw (0.525,1) -- (0.475, 1) ;
				\draw (0.575,1) -- (1,1) ;
				\draw (0.425,1.1)--(0.525,1.1);
				\draw (0.425,1.1)--(0.425,1);
				\draw (0.525,1.1)--(0.525,1);
				\draw (0.475,0.9)--(0.575,0.9);
				\draw (0.475,0.9)--(0.475,1);
				\draw (0.575,0.9)--(0.575,1);
				\draw (0.5,0.5) node{$\Omega_\varepsilon^6$} ;
				\draw[orange!50!black] (0.475,1.08) node[above]{$F_{\mathrm p,\varepsilon}^6$} ;
				\draw[green!50!black] (0.525,0.9) node[below]{$F_{\mathrm n,\varepsilon}^6$} ;
				\end{tikzpicture}
				\caption{Exact domain $\Omega_\varepsilon^6$.}
			\end{center}
		\end{subfigure}
		~
		\begin{subfigure}{0.23\textwidth}
			\begin{center}
				\begin{tikzpicture}[scale=4.6]
				\draw (0.25,1.3) -- (0.75,1.3) ;
				\draw[red,thick] (0.425,1.3) -- (0.475,1.3) ;
				\draw[red,thick] (0.475,1.3) -- (0.575,1.3) ;
				\draw[red] (0.44,1.32) node[above]{$\gamma_{0,\mathrm p}$} ;
				\draw[red] (0.525,1.28) node[below]{$\gamma_{0,\mathrm n}$} ;
				\draw[red,thick] (0.475,1.28) -- (0.475,1.32);
				\draw[blue] (0.465,1.01) -- (0.485, 0.985) ;
				\draw (0.25,1) -- (0.425,1) ;
				\draw[blue,thick] (0.475,1) -- (0.525, 1) ;
				\draw (0.575,1) -- (0.75,1) ;
				\draw[blue,thick] (0.425,1.1)--(0.525,1.1);
				\draw[blue,thick] (0.425,1.1)--(0.425,1);
				\draw[blue,thick] (0.525,1.1)--(0.525,1);
				\draw[blue,thick] (0.475,0.9)--(0.575,0.9);
				\draw[blue,thick] (0.475,0.9)--(0.475,1);
				\draw[blue,thick] (0.575,0.9)--(0.575,1);
				\draw[red,thick] (0.575,1.28) -- (0.575,1.32) ;
				\draw[red,thick] (0.425,1.28) -- (0.425,1.32) ;
				\draw[blue] (0.45,1.1) node[above]{$\gamma_\mathrm p$} ;
				\draw[blue] (0.525,0.9) node[below]{$\gamma_\mathrm n$} ;
				\end{tikzpicture}
				\caption{Zoom on part of the upper boundary of $\Omega_0$ (up) and $\Omega_\varepsilon^6$ (down).}
			\end{center}
		\end{subfigure}
		\caption{Exact domains $\Omega_\varepsilon^5$ and $\Omega_\varepsilon^6$.} \label{fig:twoclosefeat}
	\end{center}
\end{figure}
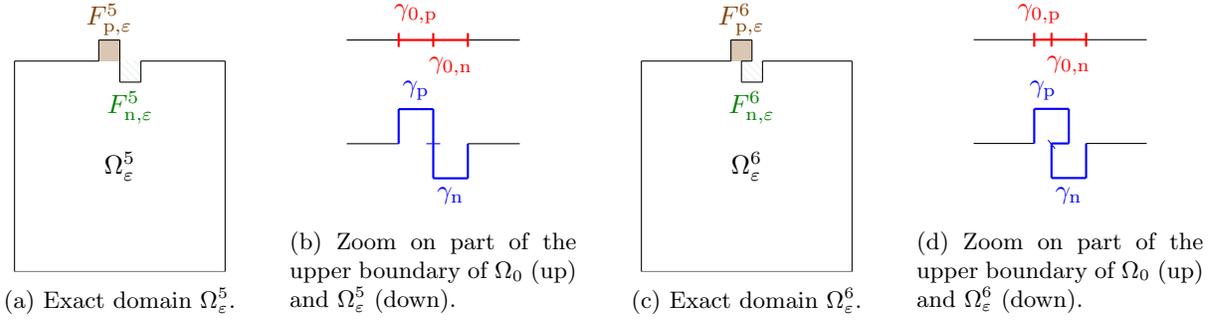

Let us finally consider two-dimensional examples of geometries with a general complex feature. Let $\Omega_0$, $\Gamma_D$, $f$, $h$ and $g$ be again as before, and for $\ell=5,6$, let $\Omega_\varepsilon^\ell := \mathrm{int}\left(\overline{\Omega_0}\cup \overline{F_{\mathrm n,\varepsilon}^\ell} \setminus \overline{F_{\mathrm p,\varepsilon}^\ell}\right)$ where, as illustrated in Fig.~\ref{fig:twoclosefeat},
\begin{align*}
F_{\mathrm p,\varepsilon}^5 &:= \left\{ \left(0.5-\varepsilon, 1\right) + (r,t) : 0<r,t<\varepsilon \right\},\\
F_{\mathrm n,\varepsilon}^5 &:= \left\{ \left(0.5, 1\right) + (r, -t) : 0<r,t<\varepsilon \right\}, \\
F_{\mathrm p,\varepsilon}^6 &:= \left\{ \left(0.5-\frac{3\varepsilon}{4}, 1\right) + (r,t) : 0<r,t<\varepsilon \right\},\\
F_{\mathrm n,\varepsilon}^6 &:= \left\{ \left(0.5-\frac{\varepsilon}{4}, 1\right) + (r, -t) : 0<r,t<\varepsilon \right\}.
\end{align*}
For each $\ell=5,6$, let $\Gamma_N := \partial \Omega_\varepsilon^\ell \setminus \overline{\Gamma_D}$ and we consider the same Poisson problem (\ref{eq:originalpb}) as before, but solved in $\Omega_\varepsilon^\ell$. We also solve its defeatured version (\ref{eq:simplpb}) in $\Omega_0$ with $g_0 \equiv 0$ on $\gamma_0$ (note from Fig.~\ref{fig:twoclosefeat} that $\gamma_0$ is different whether $\ell=5$ or $\ell=6$). Then we extend the defeatured solution to $F_\varepsilon^\ell$ by (\ref{eq:featurepb}) with $\tilde F := F_{\mathrm n,\varepsilon}^\ell$. As before, we respectively call $u^{(\ell)}$ and $u_0^{(\ell)}$ the exact and defeatured solutions, and $u_\mathrm d^{(\ell)}$ the defeatured solution extended to $F_{\mathrm n,\varepsilon}^\ell$. 

The results are presented in Fig.~\ref{fig:cv2Dgen}. As for the negative and positive feature cases, the error in $\Omega_0$ and the estimator converge with respect to the size of the feature as $\varepsilon\propto \left|\gamma_\mathrm n\right| \simeq \left|\gamma_{0,\mathrm p}\right|$ in both geometries $\Omega_\varepsilon^5$ and $\Omega_\varepsilon^6$. Moreover, the effectivity index is indeed almost independent from the size of the feature since it remains nearly equal to $1.71$ and $1.84$, respectively, for all values of $\varepsilon$. That is, as predicted by the theory since the estimator is both reliable (Theorem \ref{thm:genupperbound}) and efficient up to oscillations (Theorem \ref{thm:genlowerbound}), here in dimension two, the dependence of the estimator with respect to the size of the feature is explicit. We finally remark that the effectivity indices for the positive features are little bit larger than the ones for the negative features. 

\subsubsection{Three-dimensional geometries} \label{sec:cv3d}
Let us first consider three-dimensional examples of geometries with a negative feature. Let $\varepsilon = \displaystyle\frac{10^{-2}}{2^k}$ for $k=0,1,\ldots,6$, and $\Omega_\varepsilon^i := \Omega_0 \setminus \overline{F_\varepsilon^i}$ for $i=1,2$ with $\Omega_0 := (0,1)^3$ and 
\begin{align*}
F_\varepsilon^1 &:= \left\{ (x,y,z)\in\mathbb R^3: 0.5-\frac{\varepsilon}{2}<x<0.5+\frac{\varepsilon}{2}, 1-\varepsilon<y<1, 0<z<\varepsilon \right\}, \\
F_\varepsilon^2 &:= F_\varepsilon^1+\left(0.5-\frac{\varepsilon}{2},0,0\right),
\end{align*}
as in Fig.~\ref{fig:geomcv3Dneg} and Fig.~\ref{fig:geomcv3Dnegangle}. 
For each $i=1,2$, we consider Poisson problem (\ref{eq:originalpb}) solved in $\Omega_\varepsilon^i$, and its defeatured version (\ref{eq:simplpb}) in $\Omega_0$. We take 
\begin{align*}
&f(x,y) := 10\cos(3\pi x)\sin(5\pi y)\sin(7\pi z) \text{ in } \Omega,\\
&h\equiv 0 \text{ on }\Gamma_D := \big\{ (x,0,z)\in \mathbb R^3: 0< x,z < 1 \big\},
\end{align*}
$g\equiv 0$ on $\Gamma_N := \partial \Omega_\varepsilon^i \setminus \overline{\Gamma_D}$, and $g_0 \equiv 0$ on $\partial \Omega_0 \setminus \partial \Omega_\varepsilon^i$. 

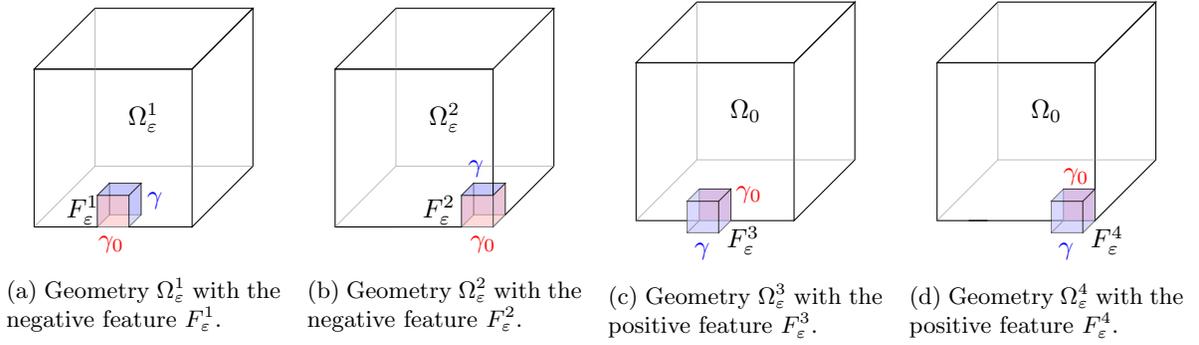
\begin{figure}
	\centering
	\begin{subfigure}{0.22\textwidth}
		\begin{center}
			\begin{tikzpicture}[scale=2.1]
			\pgfmathsetmacro{\cubex}{1}
			\pgfmathsetmacro{\cubey}{1}
			\pgfmathsetmacro{\cubez}{1}
			\draw (0,-\cubey,0) -- (-0.4,-\cubey,0);
			\draw (0,0,0) -- ++(-\cubex,0,0) -- ++(0,-\cubey,0) -- ++(0.32,0,0);
			\draw (0,0,0) -- ++(0,0,-\cubez) -- ++(0,-\cubey,0) -- ++(0,0,\cubez) -- cycle; 
			\draw (0,0,0) -- ++(-\cubex,0,0) -- ++(0,0,-\cubez) -- ++(\cubex,0,0) -- cycle;
			\pgfmathsetmacro{\cubex}{0.2}
			\pgfmathsetmacro{\cubey}{0.2}
			\pgfmathsetmacro{\cubez}{0.2}
			\draw (-0.6,-1,0) -- (-0.68,-1,0);
			\draw[fill=blue!30,opacity=0.5] (-0.4,-0.8,-0.2) -- ++(-\cubex,0,0) -- ++(0,-\cubey,0) -- ++(\cubex,0,0) -- cycle;
			\draw[fill=blue!30,opacity=0.5] (-0.4,-0.8,0) -- ++(0,0,-\cubez) -- ++(0,-\cubey,0) -- ++(0,0,\cubez) -- cycle; 
			\draw[fill=blue!30,opacity=0.5] (-0.6,-0.8,0) -- ++(0,0,-\cubez) -- ++(0,-\cubey,0) -- ++(0,0,\cubez) -- cycle; 
			\draw[fill=blue!30,opacity=0.5] (-0.4,-0.8,0) -- ++(-\cubex,0,0) -- ++(0,0,-\cubez) -- ++(\cubex,0,0) -- cycle;
			\draw[fill=red!30,opacity=0.5] (-0.4,-0.8,0) -- ++(-\cubex,0,0) -- ++(0,-\cubey,0) -- ++(\cubex,0,0) -- cycle;
			\draw[gray,opacity=0.5] (-1,-1,-1) -- (-1,-1,0);
			\draw[gray,opacity=0.5] (-1,-1,-1) -- (-1,0,-1);
			\draw[gray,opacity=0.5] (-1,-1,-1) -- (0,-1,-1);
			\draw (-0.5,-0.5,-0.5) node{$\Omega_\varepsilon^1$};
			\draw (-0.5,-0.7,0.5) node{$F_\varepsilon^1$};
			\draw[blue] (0.15,-0.45,1) node{$\gamma$};
			\draw[red] (-0.05,-0.65,1.2) node{$\gamma_0$};
			\end{tikzpicture}
			\caption{Geometry $\Omega_\varepsilon^1$ with the negative feature $F_\varepsilon^1$.}\label{fig:geomcv3Dneg}
		\end{center}
	\end{subfigure}
	~
	\begin{subfigure}{0.22\textwidth}
		\begin{center}
			\begin{tikzpicture}[scale=2.1]
			\pgfmathsetmacro{\cubex}{1}
			\pgfmathsetmacro{\cubey}{1}
			\pgfmathsetmacro{\cubez}{1}
			\draw (-1,-\cubey,0) -- (-0.2,-\cubey,0);
			\draw (0,-1,-1) -- (0,0,-1);
			\draw (0, -1, -1) -- (0, -1, -0.2); 
			\draw (0,0,0) -- ++(-\cubex,0,0) -- ++(0,-\cubey,0) -- ++(0.32,0,0);
			\draw (0,0,0) -- ++(-\cubex,0,0) -- ++(0,0,-\cubez) -- ++(\cubex,0,0) -- cycle;
			\pgfmathsetmacro{\cubex}{0.2}
			\pgfmathsetmacro{\cubey}{0.2}
			\pgfmathsetmacro{\cubez}{0.2}
			\draw (-0.6,-1,0) -- (-0.68,-1,0);
			\draw[fill=blue!30,opacity=0.5] (0,-0.8,-0.2) -- ++(-\cubex,0,0) -- ++(0,-\cubey,0) -- ++(\cubex,0,0) -- cycle;
			\draw[fill=red!30,opacity=0.5] (0,-0.8,0) -- ++(0,0,-\cubez) -- ++(0,-\cubey,0) -- ++(0,0,\cubez) -- cycle; 
			\draw[fill=blue!30,opacity=0.5] (-0.2,-0.8,0) -- ++(0,0,-\cubez) -- ++(0,-\cubey,0) -- ++(0,0,\cubez) -- cycle; 
			\draw[fill=blue!30,opacity=0.5] (0,-0.8,0) -- ++(-\cubex,0,0) -- ++(0,0,-\cubez) -- ++(\cubex,0,0) -- cycle;
			\draw[fill=red!30,opacity=0.5] (0,-0.8,0) -- ++(-\cubex,0,0) -- ++(0,-\cubey,0) -- ++(\cubex,0,0) -- cycle;
			\draw (0,-0.8,0) -- (0,0,0);
			\draw[gray,opacity=0.5] (-1,-1,-1) -- (-1,-1,0);
			\draw[gray,opacity=0.5] (-1,-1,-1) -- (-1,0,-1);
			\draw[gray,opacity=0.5] (-1,-1,-1) -- (0,-1,-1);
			\draw[gray,opacity=0.5] (0, -1, 0) -- (0, -1, -0.2); 
			\draw[gray,opacity=0.5] (0,-0.8,0) -- (0,-1,0);
			\draw (-0.5,-0.5,-0.5) node{$\Omega_\varepsilon^2$};
			\draw (-0.15,-0.7,0.5) node{$F_\varepsilon^2$};
			\draw[blue] (0.2,-0.32,0.8) node{$\gamma$};
			\draw[red] (0.4,-0.65,1.2) node{$\gamma_0$};
			\end{tikzpicture}
			\caption{Geometry $\Omega_\varepsilon^2$ with the negative feature $F_\varepsilon^2$.}\label{fig:geomcv3Dnegangle}
		\end{center}
	\end{subfigure}
	~
	\begin{subfigure}{0.22\textwidth}
		\begin{center}
			\begin{tikzpicture}[scale=2.1]
			\pgfmathsetmacro{\cubex}{1}
			\pgfmathsetmacro{\cubey}{1}
			\pgfmathsetmacro{\cubez}{1}
			\draw (0,-\cubey,0) -- (-0.4,-\cubey,0);
			\draw (0,0,0) -- ++(-\cubex,0,0) -- ++(0,-\cubey,0) -- ++(0.32,0,0);
			\draw (0,0,0) -- ++(0,0,-\cubez) -- ++(0,-\cubey,0) -- ++(0,0,\cubez) -- cycle; 
			\draw (0,0,0) -- ++(-\cubex,0,0) -- ++(0,0,-\cubez) -- ++(\cubex,0,0) -- cycle;
			\pgfmathsetmacro{\cubex}{0.2}
			\pgfmathsetmacro{\cubey}{0.2}
			\pgfmathsetmacro{\cubez}{0.2}
			\draw[gray,fill=red!20] (-0.4,-0.8,0) -- ++(-\cubex,0,0) -- ++(0,-\cubey,0) -- ++(\cubex,0,0) -- cycle;
			\draw[gray] (-0.6,-1,0) -- (-0.6,-1,0.2);
			\draw[gray] (-0.6,-1,0) -- (-0.68,-1,0);
			\draw[fill=blue!30,opacity=0.5] (-0.4,-0.8,0.2) -- ++(-\cubex,0,0) -- ++(0,-\cubey,0) -- ++(\cubex,0,0) -- cycle;
			\draw[fill=blue!30,opacity=0.5] (-0.4,-0.8,0.2) -- ++(0,0,-\cubez) -- ++(0,-\cubey,0) -- ++(0,0,\cubez) -- cycle;
			\draw[fill=blue!30,opacity=0.5] (-0.4,-0.8,0.2) -- ++(-\cubex,0,0) -- ++(0,0,-\cubez) -- ++(\cubex,0,0) -- cycle;
			\draw[gray,opacity=0.5] (-1,-1,-1) -- (-1,-1,0);
			\draw[gray,opacity=0.5] (-1,-1,-1) -- (-1,0,-1);
			\draw[gray,opacity=0.5] (-1,-1,-1) -- (0,-1,-1);
			\draw (-0.5,-0.5,-0.5) node{$\Omega_0$};
			\draw (0.1,-0.7,1.1) node{$F_\varepsilon^3$};
			\draw[red] (0.1,-0.45,1) node{$\gamma_0$};
			\draw[blue] (-0.12,-0.73,1.2) node{$\gamma$};
			\end{tikzpicture}
			\caption{Geometry $\Omega_\varepsilon^3$ with the positive feature $F_\varepsilon^3$.}\label{fig:geomcv3Dpos}
		\end{center}
	\end{subfigure}
	~
	\begin{subfigure}{0.22\textwidth}
		\begin{center}
			\begin{tikzpicture}[scale=2.1]
			\pgfmathsetmacro{\cubex}{1}
			\pgfmathsetmacro{\cubey}{1}
			\pgfmathsetmacro{\cubez}{1}
			\draw (0,-\cubey,0) -- (-0.8,-\cubey,0);
			\draw (0,0,0) -- ++(-\cubex,0,0) -- ++(0,-\cubey,0) -- ++(0.32,0,0);
			\draw (0,0,0) -- ++(0,0,-\cubez) -- ++(0,-\cubey,0) -- ++(0,0,\cubez) -- cycle; 
			\draw (0,0,0) -- ++(-\cubex,0,0) -- ++(0,0,-\cubez) -- ++(\cubex,0,0) -- cycle;
			\pgfmathsetmacro{\cubex}{0.2}
			\pgfmathsetmacro{\cubey}{0.2}
			\pgfmathsetmacro{\cubez}{0.2}
			\draw[gray,fill=red!20] (0,-0.8,0) -- ++(-\cubex,0,0) -- ++(0,-\cubey,0) -- ++(\cubex,0,0) -- cycle;
			\draw[gray] (-0.2,-1,0) -- (-0.2,-1,0.2);
			\draw[gray] (-0.2,-1,0) -- (-0.28,-1,0);
			\draw[fill=blue!30,opacity=0.5] (0,-0.8,0.2) -- ++(-\cubex,0,0) -- ++(0,-\cubey,0) -- ++(\cubex,0,0) -- cycle;
			\draw[fill=blue!30,opacity=0.5] (0,-0.8,0.2) -- ++(0,0,-\cubez) -- ++(0,-\cubey,0) -- ++(0,0,\cubez) -- cycle;
			\draw[fill=blue!30,opacity=0.5] (0,-0.8,0.2) -- ++(-\cubex,0,0) -- ++(0,0,-\cubez) -- ++(\cubex,0,0) -- cycle;
			\draw[gray,opacity=0.5] (-1,-1,-1) -- (-1,-1,0);
			\draw[gray,opacity=0.5] (-1,-1,-1) -- (-1,0,-1);
			\draw[gray,opacity=0.5] (-1,-1,-1) -- (0,-1,-1);
			\draw (-0.5,-0.5,-0.5) node{$\Omega_0$};
			\draw (0.5,-0.7,1.1) node{$F_\varepsilon^4$};
			\draw[red] (0.15,-0.45,0.7) node{$\gamma_0$};
			\draw[blue] (0.28,-0.73,1.2) node{$\gamma$};
			\end{tikzpicture}
			\caption{Geometry $\Omega_\varepsilon^4$ with the positive feature $F_\varepsilon^4$.}\label{fig:geomcv3Dposangle}
		\end{center}
	\end{subfigure}
	\caption{$3$D geometries $\Omega_\varepsilon^i$, $i=1,2,3,4$.} \label{fig:geomcv3D}
\end{figure}

\begin{figure}
	\centering
	\begin{subfigure}{0.48\textwidth}
		\begin{center}
			\begin{tikzpicture}[scale=0.7]
			\begin{axis}[xmode=log, ymode=log, legend style={at={(0.5,-0.15)}, legend columns=3, anchor=north, draw=none}, xlabel=$\varepsilon$, ylabel=Error]
			\addplot[mark=x, red, thick] table [x=eps, y=errh1s, col sep=comma] {data/cv3dneg.csv};
			\addplot[mark=x, blue, densely dashed, mark options=solid, thick] table [x=eps, y=est, col sep=comma] {data/cv3dneg.csv};
			\addplot[mark=none, loosely dotted, thick] table [x=eps, y=epspower, col sep=comma] {data/cv3dneg.csv};
			\addplot[mark=o, red, thick] table [x=eps, y=errh1s, col sep=comma] {data/cv3dnegangle.csv};
			\addplot[mark=o, blue, densely dashed, mark options=solid, thick] table [x=eps, y=est, col sep=comma] {data/cv3dnegangle.csv};
			\addplot[mark=none, thick, densely dotted, thick] table [x=eps, y=eps52, col sep=comma] {data/cv3dnegangle.csv};
			\legend{$\left|u-u_0^{(1)}\right|_{1,\Omega_\varepsilon^1}$,$\mathcal E\left(u_0^{(1)}\right)$,$\mathcal O\left(\varepsilon^\frac{3}{2}\right)$, $\left|u-u_\mathrm d^{(2)}\right|_{1,\Omega_\varepsilon^2}$, $\mathcal E\left(u_\mathrm d^{(2)}\right)$, $\mathcal O\left(\varepsilon^\frac{5}{2}\right)$};
			\end{axis}
			\end{tikzpicture}
			\caption{Geometries $\Omega_\varepsilon^1$ and $\Omega_\varepsilon^2$ with a negative feature.} \label{fig:cv3Dneg}
		\end{center}
	\end{subfigure}
	~
	\begin{subfigure}{0.48\textwidth}
		\begin{center}
			\begin{tikzpicture}[scale=0.7]
			\begin{axis}[xmode=log, ymode=log, legend style={at={(0.5,-0.15)}, legend columns=3, anchor=north, draw=none}, xlabel=$\varepsilon$, ylabel=Error]
			\addplot[mark=x, red, thick] table [x=eps, y=errh1s, col sep=comma] {data/cv3dpos.csv};
			\addplot[mark=x, blue, densely dashed, mark options=solid, thick] table [x=eps, y=est, col sep=comma] {data/cv3dpos.csv};
			\addplot[mark=none, loosely dotted, thick] table [x=eps, y=epspower, col sep=comma] {data/cv3dpos.csv};
			\addplot[mark=o, red, thick] table [x=eps, y=err, col sep=comma] {data/cv3dposangle.csv};
			\addplot[mark=o, blue, densely dashed, mark options=solid, thick] table [x=eps, y=est, col sep=comma] {data/cv3dposangle.csv};
			\addplot[mark=none, thick, densely dotted, thick] table [x=eps, y=epspow, col sep=comma] {data/cv3dposangle.csv};
			\legend{$\left|u-u_\mathrm d^{(3)}\right|_{1,\Omega_\varepsilon^3}$, $\mathcal E\left(u_\mathrm d^{(3)}\right)$,$\mathcal O\left(\varepsilon^\frac{3}{2}\right)$, $\left|u-u_\mathrm d^{(4)}\right|_{1,\Omega_\varepsilon^4}$, $\mathcal E\left(u_\mathrm d^{(4)}\right)$, $\mathcal O\left(\varepsilon^\frac{5}{2}\right)$};
			\end{axis}
			\end{tikzpicture}
			\caption{Geometries $\Omega_\varepsilon^3$ and $\Omega_\varepsilon^4$ with a positive feature.}\label{fig:cv3Dpos}
		\end{center}
	\end{subfigure}
	\caption{Convergence of the error and of the estimator in $3$D domains with one feature.} \label{fig:cv3D}
\end{figure}

The results are presented in Fig.~\ref{fig:cv3Dneg}. Both the error and the estimator converge with respect to the size of the feature as $\varepsilon^\frac{3}{2}\propto |\gamma_0|^\frac{3}{4}$ in the first geometry $\Omega_\varepsilon^1$, and as $\varepsilon^\frac{5}{2} \propto |\gamma|^\frac{5}{4}$ in the second geometry $\Omega_\varepsilon^2$. Moreover, the effectivity index is indeed independent from the size of the feature since it remains nearly equal to $1.87$ and $1.92$, respectively, for all values of $\varepsilon$. That is, again as predicted by the theory since the estimator is both reliable (Theorem \ref{thm:upperbound}) and efficient up to oscillations (Theorem \ref{thm:lowerbound}), here in dimension three, the dependence of the estimator with respect to the size of the feature is explicit. \\

Let us now consider three-dimensional examples of geometries with a positive feature. Let $\Omega_0$, $\Gamma_D$, $f$, $h$, and $g$ be as before, and let $\Omega_\varepsilon^j := \mathrm{int}\left(\overline{\Omega_0} \cup \overline{F_\varepsilon^j}\right)$ for $j=3,4$ with 
\begin{align*}
F_\varepsilon^3 &:= \left\{ (x,y,z)\in\mathbb R^3: 0.5-\frac{\varepsilon}{2}<x<0.5+\frac{\varepsilon}{2}, 1<y<1+\varepsilon, 0<z<\varepsilon \right\}, \\
F_\varepsilon^4 &:= F_\varepsilon^3 + \left(0.5-\frac{\varepsilon}{2}, 0, 0\right),
\end{align*}
as in Fig.~\ref{fig:geomcv3Dpos} and Fig.~\ref{fig:geomcv3Dposangle}. Let $\Gamma_N := \partial \Omega_\varepsilon^j\setminus \overline{\Gamma_D}$. 
For each $j=3,4$, we consider the same Poisson problem (\ref{eq:originalpb}) as before, but solved in this $\Omega_\varepsilon^j$. We also solve its defeatured version (\ref{eq:simplpb}) in $\Omega_0$ with $g_0 \equiv 0$ on $\partial \Omega_0 \setminus \partial \Omega_\varepsilon$. Then we extend the defeatured solution to $F_\varepsilon^j$ by (\ref{eq:featurepb}) with $\tilde F := F_\varepsilon^j$. 

The results are presented in Fig.~\ref{fig:cv3Dpos}. As for the negative feature case, the error in $\Omega_0$, the error in $F_\varepsilon^j$ and the estimator converge with respect to the size of the feature as $\varepsilon^\frac{3}{2}\propto |\gamma_0|^\frac{3}{4}$ in the first geometry $\Omega_\varepsilon^3$, and as $\varepsilon^\frac{5}{2}\propto |\gamma_0|^\frac{5}{4}$ in the second geometry $\Omega_\varepsilon^4$. Moreover, the effectivity index is indeed almost independent from the size of the feature since it remains nearly equal to $3.10$ and $3.22$, respectively, for all values of $\varepsilon$. That is, as predicted by the theory since the estimator is both reliable (Theorem \ref{thm:genupperbound}) and efficient up to oscillations (Theorem \ref{thm:genlowerbound}), here in dimension three, the dependence of the estimator with respect to the size of the feature is explicit. Finally, and as in the two-dimensional case, we remark that the effectivity indices for the positive features are a little bit larger than the ones for the negative features.

\subsubsection{Effect of the choice of the defeatured problem data }
Let us study the effect of the choice of the defeatured problem data on the convergence of the defeaturing error and estimator. In particular, we will see that in the example of a geometry with one negative feature $F$, the convergence of the error and the estimator crucially depends on the value of $\overline{\left(g+\displaystyle\frac{\partial u_0}{\partial \mathbf n_F}\right)}^\gamma$. As seen in Remark \ref{rmk:compatcondneg}, this value only depends on the Neumann boundary conditions $g$ on $\gamma$ and $g_0$ on $\gamma_0$, and on the extension of the right hand side $f$ in $F$. This means that one can obtain an optimal convergence rate of the defeaturing error by wisely choosing the defeatured data $g_0$ and $f$, considering the original data $g$ (if possible, to satisfy the compatibility condition (\ref{eq:compatcondneg})), so that the second term of the estimator in (\ref{eq:negestimator}) converges faster than the first one. The same observation can be made in the positive feature case. 

To show this, let $\varepsilon=\displaystyle \frac{10^{-2}}{2^k}$ for $k=0,1,\ldots,6$. We consider a 2D geometry with one negative feature. More precisely, let $\Omega_0$ be the disk centered in $(0,0)$ of radius $1$, let $F_\varepsilon$ be the disk centered in $(0,0)$ of radius $\varepsilon$, and let $\Omega_\varepsilon := \Omega_0\setminus F_\varepsilon$, as represented in Fig.~\ref{fig:fouriergeom}. We solve Poisson problem (\ref{eq:originalpb}) in $\Omega_\varepsilon$ with $f\equiv 1$ in $\Omega_\varepsilon$, $h\equiv 0$ on $\Gamma_D := \partial \Omega_0$, and we choose different Neumann data $g=g_i$ on $\partial F_\varepsilon$ for $i=1,\ldots,4$, where $g_1 \equiv 0$, $g_2 \equiv 1$, $g_3 \equiv \varepsilon^{-1}$, and  $g_4 \equiv \varepsilon^{-3}$. Then we solve the defeatured problem (\ref{eq:simplpb}) in $\Omega_0$, for which we need to choose an extension of $f$ in $F_\varepsilon$, that we still call $f$. This extension should somehow mimic the behavior of the Neumann data $g$, as required by the compatibility condition, but instead of that, we choose the trivial extension $f\equiv 1$ in all four cases, and we will verify whether this is always a good choice or not. For $i=1,\ldots,4$, we call $u^{(i)}$ and $u_0^{(i)}$ the solutions of (\ref{eq:originalpb}) and (\ref{eq:simplpb}), respectively. 

\begin{figure}
	\centering
	\begin{tikzpicture}[scale=0.85]
	\begin{axis}[xmode=log, ymode=log, legend style={at={(1.4,0.67)}, legend columns=2, anchor=north, draw=none}, xlabel=$\varepsilon$, ylabel=Error]
	\addplot[mark=x, red, thick] table [x=eps, y=errh1s0, col sep=comma] {data/cvdiffNeumann.csv};
	\addplot[mark=x, blue, densely dashed, mark options=solid, thick] table [x=eps, y=est0, col sep=comma] {data/cvdiffNeumann.csv};
	\addplot[mark=o, red, thick] table [x=eps, y=errh1s1, col sep=comma] {data/cvdiffNeumann.csv};
	\addplot[mark=o, blue, densely dashed, mark options=solid, thick] table [x=eps, y=est1, col sep=comma] {data/cvdiffNeumann.csv};
	\addplot[mark=+, red, thick] table [x=eps, y=errh1sepsmp1, col sep=comma] {data/cvdiffNeumann.csv};
	\addplot[mark=+, blue, densely dashed, mark options=solid, thick] table [x=eps, y=estepsmp1, col sep=comma] {data/cvdiffNeumann.csv};
	\addplot[mark=*, red, thick] table [x=eps, y=errh1sepsmp3, col sep=comma] {data/cvdiffNeumann.csv};
	\addplot[mark=*, blue, densely dashed, mark options=solid, thick] table [x=eps, y=estepsmp3, col sep=comma] {data/cvdiffNeumann.csv};
	\legend{$\left|u^{(1)}-u_0^{(1)}\right|_{1,\Omega_\varepsilon}$,$\mathcal E\left(u_0^{(1)}\right)$,
		$\left|u^{(2)}-u_0^{(2)}\right|_{1,\Omega_\varepsilon}$, $\mathcal E\left(u_0^{(2)}\right)$,
		$\left|u^{(3)}-u_0^{(3)}\right|_{1,\Omega_\varepsilon}$,$\mathcal E\left(u_0^{(3)}\right)$,
		$\left|u^{(4)}-u_0^{(4)}\right|_{1,\Omega_\varepsilon}$,$\mathcal E\left(u_0^{(4)}\right)$};
	\end{axis}
	\end{tikzpicture}
	\caption{Convergence of the error and of the estimator with different Neumann boundary conditions.} \label{fig:choicedefeatured}
\end{figure}
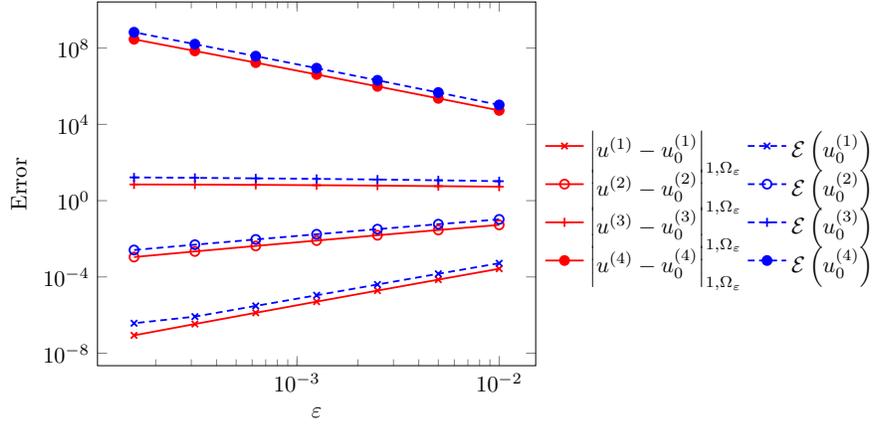

The results are presented in Fig.~\ref{fig:choicedefeatured}. As we can see and as expected, the proposed estimator follows the convergence of the defeaturing error in all four cases. Moreover, the effectivity index always remains the same, as we were also expecting since the shape of the geometry never changes. However, we see that the trivial extension of $f$ in $F$ is not always a good choice since it slows down the convergence when $g=g_2$, it does not permit the error to decrease with $\varepsilon$ when $g=g_3$, and it even implies the error to explode with $\varepsilon$ when $g=g_4$. The explanation is present in the expression of the estimator in (\ref{eq:negestimator}): indeed, in the case $g=g_1$, the first term of (\ref{eq:negestimator}) is dominant, while in the other cases, the second term dominates because of the value of $\overline{\left(g+\displaystyle\frac{\partial u_0}{\partial \mathbf n_F}\right)}^\gamma$ due to the bad choice of $f$ in $F_\varepsilon$. 

Consequently, the estimator not only tells us whether a feature is important for the given problem at hand, but it also tells us whether the choice of the defeaturing problem data is right or should be reconsidered. 

\subsection{Non-Lipschitz features: fillets and rounds} \label{sec:nonlipschitz}
Classical features one finds in design for manufacturing are fillets and rounds, that allow for example the use of round-tipped end mills to cut out some material. However, when considered as features isolated from the rest of the domain, fillets and rounds are non-Lipschitz feature domains. The following numerical examples analyze these types of features, and show that our estimator manages to capture the behavior of the defeaturing error even if the domains are not Lipschitz.

\subsubsection{Round: a negative non-Lipschitz feature}
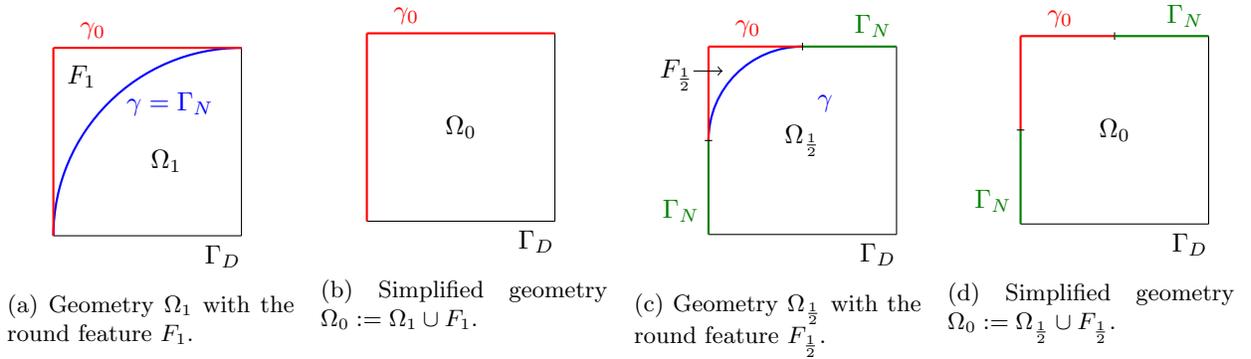
\begin{figure}[b]
	\centering
	\begin{subfigure}{0.23\textwidth}
		\begin{center}
			\begin{tikzpicture}[scale=2.5]
			\draw (0,0) -- (1,0) ;
			\draw (1,0) -- (1,1) ;
			\draw[blue,thick] (1, 1) arc (90:180:1cm);
			\draw[red,thick] (0,0) -- (0,1) ;
			\draw[red,thick] (0,1) -- (1,1) ;
			\draw[red] (0.21,1) node[above]{$\gamma_0$};
			\draw[blue] (0.62,0.7) node{$\gamma=\Gamma_N$};
			\draw (0.6,0.4) node{$\Omega_1$} ;
			\draw (0.15,0.85) node{$F_1$} ;
			\draw (0.9,0) node[below]{$\Gamma_D$} ; 
			\end{tikzpicture}
			\caption{Geometry $\Omega_1$ with the round feature $F_1$.}
		\end{center}
	\end{subfigure}
	~
	\begin{subfigure}{0.23\textwidth}
		\begin{center}
			\begin{tikzpicture}[scale=2.5]
			\draw (0,0) -- (1,0) ;
			\draw (1,0) -- (1,1) ;
			\draw[red,thick] (0,0) -- (0,1) ;
			\draw[red,thick] (0,1) -- (1,1) ;
			\draw[red] (0.21,1) node[above]{$\gamma_0$};
			\draw (0.5,0.5) node{$\Omega_0$} ;
			\draw (0.9,0) node[below]{$\Gamma_D$} ; 
			\end{tikzpicture}
			\caption{Simplified geometry $\Omega_0 := \Omega_1 \cup F_1$.\\} 
		\end{center}
	\end{subfigure}
	~
	\begin{subfigure}{0.23\textwidth}
		\begin{center}
			\begin{tikzpicture}[scale=2.5]
			\draw (0,0) -- (1,0) ;
			\draw (1,0) -- (1,1) ;
			\draw[blue,thick] (0.5, 1) arc (90:180:0.5cm);
			\draw[red,thick] (0,0.5) -- (0,1) ;
			\draw[red,thick] (0,1) -- (0.5,1) ;
			\draw[thick, green!50!black] (0.5,1) -- (1,1);
			\draw[thick, green!50!black] (0,0) -- (0,0.5) ;
			\draw[red] (0.21,1) node[above]{$\gamma_0$};
			\draw[blue] (0.62,0.7) node{$\gamma$};
			\draw (0.5,0.5) node{$\Omega_{\frac{1}{2}}$} ;
			\draw (-0.03,0.85) node[left]{$F_{\frac{1}{2}}$} ;
			\draw[->] (-0.08,0.87) -- (0.075,0.87) ;
			\draw (0.9,0) node[below]{$\Gamma_D$} ; 
			\draw[green!50!black] (0,0.125) node[left]{$\Gamma_N$} ; 
			\draw[green!50!black] (0.875,1) node[above]{$\Gamma_N$} ; 
			\draw (0.5,0.98) -- (0.5,1.02) ;
			\draw (-0.02,0.5) -- (0.02,0.5) ;
			\end{tikzpicture}
			\caption{Geometry $\Omega_{\frac{1}{2}}$ with the round feature $F_{\frac{1}{2}}$.}
		\end{center}
	\end{subfigure}
	~
	\begin{subfigure}{0.23\textwidth}
		\begin{center}
			\begin{tikzpicture}[scale=2.5]
			\draw (0,0) -- (1,0) ;
			\draw (1,0) -- (1,1) ;
			\draw[thick, green!50!black] (0,0) -- (0,0.5) ;
			\draw[red,thick] (0,0.5) -- (0,1);
			\draw[red,thick] (0,1) -- (0.5,1) ;
			\draw (0.5,0.98) -- (0.5,1.02) ;
			\draw (-0.02,0.5) -- (0.02,0.5) ;
			\draw[thick, green!50!black] (0.5,1) -- (1,1);
			\draw[red] (0.21,1) node[above]{$\gamma_0$};
			\draw (0.5,0.5) node{$\Omega_0$} ;
			\draw (0.9,0) node[below]{$\Gamma_D$} ; 
			\draw[green!50!black] (0,0.125) node[left]{$\Gamma_N$} ; 
			\draw[green!50!black] (0.875,1) node[above]{$\Gamma_N$} ; 
			\end{tikzpicture}
			\caption{Simplified geometry $\Omega_0 := \Omega_{\frac{1}{2}} \cup F_{\frac{1}{2}}$.\\} 
		\end{center}
	\end{subfigure}
	\caption{Geometries with a round.} \label{fig:negfillet}
\end{figure}

Let us first consider the case of a round, that is, the rounding process creates a convex domain. For $R\in (0,1]$, and as represented in Fig.~\ref{fig:negfillet}, let 
\begin{align*}
\Omega_R := &(R,1-R)+\left\{\big(r\cos\left(\theta\right),r\sin\left(\theta\right)\big) \in \mathbb R^2: 0\leq r < R, \,\frac{\pi}{2}< \theta< \pi\right\} \\
& \cup (0,1)\times(0,R] \cup [R,1)\times[1-R,1),
\end{align*}
$\Omega_0 := (0,1)^2$, and $F_R:= \Omega_0\setminus\overline{\Omega_R}$. We remark that $F_R$ is not a Lipschitz domain, that is, this case is not covered by the presented theory. We consider Poisson problem (\ref{eq:originalpb}) with $f\equiv 0$ in $\Omega_R$, $h(x,y):= x^2\left(1-x\right)^2+y^2(1-y)^2$ on 
$$\Gamma_D := \big\{ (x,0), (1,y)\in \mathbb R^2: 0\leq x,y < 1 \big\}.$$
and $g\equiv 0$ on $\Gamma_N:= \partial \Omega_R \setminus \overline{\Gamma_D}$. 
We solve the defeatured Poisson problem (\ref{eq:simplpb}) with the same data and $g_0\equiv 0$ on $\gamma_0:= \partial F_R \setminus \overline{\Gamma_N}.$

The results are presented in Table \ref{table:negfillet}, and for all considered values of $R$, we indeed have $\big|u-u_0\big|_{1,\Omega_R}\lesssim \mathcal E(u_0)$ with a low effectivity index. In particular, the effectivity index is almost the same for all considered values of $R$ in $(0,0.5)$ while it is slightly larger for $R\in(0.5,1)$, since the geometries for $R\in(0,0.5)$ are almost an homothety of one another, while it is not when $R>0.5$ because of the closeness of the boundary $\Gamma_D$ from the boundary $\gamma$. This example shows that our estimator estimates well the defeaturing error even if the feature is not a Lipschitz domain, and it confirms the fact that we can indeed have a feature that is attached to the Dirichlet boundary, that is $\overline{\gamma}\cap\overline{\Gamma_D}\neq \emptyset$ but ${\gamma}\cap{\Gamma_D}= \emptyset$, as in the case $R=1$. 

\begin{table}
	\centering
	{\def\arraystretch{1.2}
		\begin{tabular}{@{}cccc@{}}
			\hline\\ [-1em]
			$R$ & $\mathcal{E}(u_0)$ & $\big|u-u_0\big|_{1,\Omega_R}$ & Effectivity index\\ [-1em] \\
			\hline
			$1$ & $6.83\cdot 10^{-3}$ & $2.37\cdot 10^{-3}$ & $2.88$\\
			$0.99$ & $6.48\cdot 10^{-3}$ & $2.27\cdot 10^{-3}$ & $2.85$\\
			$0.5$ & $3.36\cdot 10^{-4}$ & $1.26\cdot 10^{-4}$ & $2.67$\\
			$0.25$ & $2.08\cdot 10^{-5}$ & $7.77\cdot 10^{-6}$ & $2.67$\\
			$0.125$ & $1.30\cdot 10^{-6}$ & $4.86\cdot 10^{-7}$ & $2.67$\\
			\hline
	\end{tabular}}
	\caption{Results for the geometry with a round.\label{table:negfillet} }
\end{table}

\subsubsection{Fillet: a positive non-Lipschitz feature} \label{sec:fillet}
Now, let us consider the case of a fillet, that is, the filleting process creates a non-convex domain. Since the fillet $F$ is a complex positive feature we possibly do not want to mesh, we will consider two different feature extensions $\tilde F^1$ and $\tilde F^2$ containing $F$ to solve the extension problem (\ref{eq:featurepb}). We will compare them, and we will also compare the result with the one obtained without feature extension, that is for $\tilde F = F$. In particular, we remark again that $F$ is not a Lipschitz domain, that is, this example is not covered by the presented theory. 
As illustrated in Fig.~\ref{fig:posfillet}, let 
\begin{align*}
& \Omega_0 := (0,1)^2\setminus \left[\frac{1}{2},1\right]^2, \hspace{1cm}
\tilde F^1:= \left(\frac{1}{2}, 1\right)^2,\\
& \tilde F^2:= \tilde F^1 \setminus \left\{  \big(1+r\cos(\theta), 1+r\sin(\theta)\big)\in\mathbb{R}^2: 0\leq r\leq\frac{1}{4}, \pi\leq\theta\leq \frac{3\pi}{2}\right\},\\
& F:= \tilde F^1 \setminus \left\{  \big(1+r\cos(\theta), 1+r\sin(\theta)\big)\in\mathbb{R}^2: 0\leq r\leq\frac{1}{2}, \pi\leq \theta\leq \frac{3\pi}{2}\right\}, \\
& \Omega := \mathrm{int}\left(\overline{\Omega_0}\cup \overline{F}\right).
\end{align*}
$\tilde F^1$ is the bounding box of $F$, it is therefore a very simple geometry but $\left|\tilde F^1\right| \gg |F|$. At the contrary, $\tilde F^2$ is a little bit more complex, but $\left|\tilde F^2\right| \approx |F|$. 
\begin{figure}
	\centering
	\begin{subfigure}[b]{0.3\textwidth}
		\begin{center}
			\begin{tikzpicture}[scale=3]
			\draw (0,0) -- (1,0);
			\draw (1,0) -- (1,0.5);
			\draw[blue,thick] (0.5, 1) arc (180:270:0.5cm);
			\draw (0,0) -- (0,1);
			\draw (0,1) -- (0.5,1);
			\draw[blue,thick] (0.58,0.6) node{$\gamma$};
			\draw (0.4,0.4) node{$\Omega$} ;
			\end{tikzpicture}
			\caption{Exact original geometry {$\Omega~=~\mathrm{int}\left(\overline{\Omega_0}\cup\overline{F}\right)$.}}
		\end{center}
	\end{subfigure}
	~
	\begin{subfigure}[b]{0.3\textwidth}
		\begin{center}
			\begin{tikzpicture}[scale=3]
			\draw (0,0) -- (1,0);
			\draw[red,thick] (0.5,0.5) -- (0.5,1);
			\draw (1,0) -- (1,0.5);
			\draw[red,thick] (0.5,0.5) -- (1,0.5);
			\draw[green!60!black!,thick] (0.5,1) -- (1,1);
			\draw[green!60!black!,thick] (1,0.5) -- (1,1);
			\draw[blue,thick] (0.5, 1) arc (180:270:0.5cm);
			\draw (0,0) -- (0,1);
			\draw (0,1) -- (0.5,1);
			\draw[red] (0.75,0.5) node[below]{$\gamma_0$};
			\draw[blue] (0.58,0.6) node{$\gamma$};
			\draw[green!60!black!] (1.1,0.65) node[above]{$\tilde\gamma$};
			\draw (0.3,0.3) node{$\Omega_0$} ;
			\draw[gray] (0.85,0.85) node{$\tilde F^1$} ;
			\fill[gray, fill, opacity=0.3] (0.5,0.5) rectangle (1,1);
			\end{tikzpicture}
			\caption{Simplified geometry $\Omega_0$ and feature extension $\tilde F^1$.}
		\end{center}
	\end{subfigure}
	~
	\begin{subfigure}[b]{0.3\textwidth}
		\begin{center}
			\begin{tikzpicture}[scale=3]
			\fill[gray, fill, opacity=0.3] (0.5,0.5) rectangle (1,1);
			\fill[white] (0.7,1) arc (180:270:0.3cm) -- (1.01,1.01) -- cycle;
			\draw (0,0) -- (1,0);
			\draw[red,thick] (0.5,0.5) -- (0.5,1);
			\draw (1,0) -- (1,0.5);
			\draw[red,thick] (0.5,0.5) -- (1,0.5);
			\draw[green!60!black!,thick] (0.5,1) -- (0.7,1);
			\draw[green!60!black!,thick] (1,0.5) -- (1,0.7);
			\draw[blue,thick] (0.5, 1) arc (180:270:0.5cm);
			\draw[green!60!black!,thick] (0.7, 1) arc (180:270:0.3cm);
			\draw (0,0) -- (0,1);
			\draw (0,1) -- (0.5,1);
			\draw[red] (0.75,0.5) node[below]{$\gamma_0$};
			\draw[blue] (0.58,0.6) node{$\gamma$};
			\draw[green!60!black!] (0.85,0.85) node{$\tilde\gamma$};
			\draw (0.3,0.3) node{$\Omega_0$} ;
			\draw[gray] (0.615,0.9) node{$\tilde F^2$} ;
			\end{tikzpicture}
			\caption{Simplified geometry $\Omega_0$ and feature extension $\tilde F^2$.}
		\end{center}
	\end{subfigure}
	\caption{Geometry $\Omega = \mathrm{int}\left(\overline{\Omega_0} \cup \overline{F}\right)$ with a fillet $F$, and two possible extended features.} \label{fig:posfillet}
\end{figure}
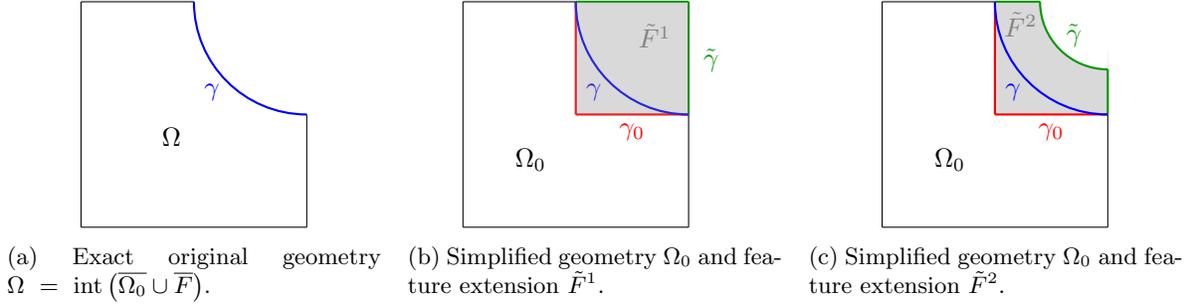
\begin{table}
	\centering
	{\def\arraystretch{1.2}
		\begin{tabular}{@{}cccccc@{}} \hline\\ [-1em]
			Extension & $\mathcal{E}(u_\mathrm d)$ & $\big|u-u_\mathrm d\big|_{1,\Omega}$ & $\big|u-u_0\big|_{1,\Omega_0}$ & $\big|u-\tilde u_0\big|_{1,F}$ & Effectivity index\\ [-1em] \\
			\hline
			$\tilde F^1$ & $1.78\cdot 10^0$ & $2.92\cdot 10^{-1}$ & $1.69\cdot 10^{-1}$ & $2.39\cdot 10^{-1}$ & $6.11$\\
			$\tilde F^2$ & $1.71\cdot 10^0$ & $2.89\cdot 10^{-1}$ & $1.69\cdot 10^{-1}$ & $2.34\cdot 10^{-1}$ & $5.93$\\
			$F$ & $1.33\cdot 10^0$ & $2.69\cdot 10^{-1}$ & $1.69\cdot 10^{-1}$ & $2.01\cdot 10^{-1}$ & $4.94$\\
			\hline
	\end{tabular}}
	\caption{Results for the geometry with a fillet. \label{table:posfillet}}
\end{table}

We consider Poisson problem (\ref{eq:originalpb}) with $f\equiv 0$ in $\Omega$, $$h(x,y):= \cos\left(\pi x\right)+10\cos(5\pi x)$$ on 
$\Gamma_D := \big\{ (x,0), \in \mathbb R^2: 0\leq x\leq 1 \big\},$
and $g\equiv 0$ on $\Gamma_N=\partial \Omega\setminus \overline{\Gamma_D}.$
We solve the defeatured Poisson problem (\ref{eq:simplpb}) with the same data and with $g_0\equiv 0$ on $\gamma_0:= \partial \Omega_0\cap\partial F$. Then, we solve the Dirichlet extension problem (\ref{eq:featurepb}) first in $\tilde F^1$, and secondly in $\tilde F^2$, with $\tilde g \equiv 0$ on $\tilde \gamma:= \partial \tilde F^1 \setminus \overline{\gamma_0}$ and $\tilde \gamma:= \partial \tilde F^2 \setminus \overline{\gamma_0}$, respectively. Finally, we also solve (\ref{eq:featurepb}) by taking $\tilde F:=F$ itself. 

The results are presented in Table \ref{table:posfillet}, and we indeed have $\big|u-u_\mathrm d\big|_{1,\Omega}\lesssim \mathcal E(u_\mathrm d)$ with a reasonable effectivity index in all three cases. Note that the effectivity index is higher in this case than in the case of a round since not only the geometry $\Omega$ but also the feature $F$ are simplified, respectively by $\Omega_0$ and by $\tilde F^1$ or $\tilde F^2$. 
Moreover, $F$ contains the extension $\tilde F^1$ that itself contains the extension $\tilde F^2$, and this is reflected both on the defeaturing error and on the estimator. Indeed, both the error and the estimator are larger when the considered extension is $\tilde F^1$ instead of $\tilde F^2$, and smaller when $\tilde F=F$, but the effectivity index is not affected: it is different because the shapes of $F$, $\tilde F^1$ and $\tilde F^2$ are different, not because an extension is bigger than the other one, as we have seen in the numerical examples of Sec.~\ref{sec:impactfeat} and Sec.~\ref{sec:errcv}. 
Furthermore, the effectivity index on the fillet is larger than the one on the round: as already remarked in Sec.~\ref{sec:errcv}, the effectivity index is in general larger for positive features than for negative ones.  {Finally, the effectivity indices for both the round and the fillet are larger than for the other negative and positive features, respectively, and this can come from the fact that rounds and fillets are non-Lipschitz features.}

{\begin{remark}[Effectivity indices] Let us summarize the observations made on the behavior of the effectivity indices. 
		\begin{itemize}
			\item The effectivity index of every test case is independent from the measure of the feature (see Sec.~\ref{sec:cv2d} and \ref{sec:cv3d}), but it depends on their shape (see Sec.~\ref{sec:featshape}). 
			\item The observed effectivity indices are small when Lipschitz features are considered: in both two and three dimensions, the value of the effectivity index ranges between $1$ and $4$ (see Sec.~\ref{sec:impactfeat} and \ref{sec:errcv}), in general with smaller values for negative features than for positive ones. While we have no evidence of it, the difference possibly comes from the smoothness of the defeatured solution $u_0$. Indeed, in the positive feature case, a $\gamma_0$-Dirichlet extension of $u_0$ is necessary to define the error in the whole geometry $\Omega$. This extension is in $H^1(\Omega)$ by definition, but its gradient jumps at the boundary $\gamma_0$, making it possibly less regular than the defeatured solution that one can have in the negative feature case. 
			\item The observed effectivity indices are larger when non-Lipschitz features are considered, but such geometries are not considered in the presented theory. The special cases of rounds and fillets are analyzed in Sec.~\ref{sec:nonlipschitz}, and the observed effectivity indices in those cases are smaller than $3$ and $5$, respectively. 
			\item The observed effectivity indices in the case of an extended positive feature, that is, a feature for which $\tilde F \supsetneq F$, are larger than in the case of a positive feature, while still remaining relatively small (see Sec. \ref{sec:fillet}). Indeed, in the former case, not only the geometry $\Omega$ but also the feature $F$ are simplified, respectively by $\Omega_0$ and by $\tilde F$. 
		\end{itemize}
\end{remark}}

\section{Conclusions} \label{sec:ccl}
We have introduced a novel \textit{a posteriori} error estimator for analysis-aware geometric defeaturing in the context of {Poisson} equation on geometries of arbitrary dimension. We have demonstrated its reliability and efficiency up to oscillations, and tested it on an extensive set of numerical experiments: in all of them, we have observed that the proposed estimator acts as an excellent approximation of the true error. We have considered geometries with either a negative, a positive or a general complex feature, and we have verified that our estimator is not only driven by geometrical considerations, but also by the differential problem at hand. The proposed estimator is able to weight the impact of defeaturing in energy norm, and it is explicit with respect to the size of the geometrical features.
Finally, our estimator is simple and computationally cheap: once the solution of the defeatured problem is computed, it only requires the computation of the solution of a local extension problem if the feature has a positive component, and of boundary integrals. 

In this paper, the analysis is performed in continuous spaces, and for one feature only. The extension to a few features is not hard as the indicator is merely additive, but the extension to several features as well as the development of a fully adaptive scheme taking into account the discretization and the defeaturing errors will be the object of our subsequent work. {Finally, this work focuses on the global energy norm of the error, which is an important first step to understand the impact of defeaturing in analysis. Studying defeaturing using a local goal-oriented error measure would be a further step attracting a broader industrial interest.}

\appendix
\section{Appendix} \label{sec:appendix}
In this section, we state lemmas that are used throughout the paper, and the symbol $\lesssim$ will be used to mean any inequality which does not depend on the size of the considered domains, but which can depend on their shape. {We assume that all considered domains are Lipschitz.}

\begin{mylemma}[Poincar\'e I] \label{lemma:poincarebd}
	Let $\omega$ be an $(n-1)$-dimensional manifold in $\mathbb{R}^{n}$ that is isotropic according to Definition \ref{as:isotropy}. 
	Then for all $v\in H^{\frac{1}{2}}(\omega)$,
	$$\|v-\overline{v}\|_{0,\omega} \lesssim \left|\omega\right|^\frac{1}{2(n-1)} |v|_{\frac{1}{2}, \omega},$$
	where $\overline{v} := \displaystyle\frac{1}{\left|\omega\right|}\int_{\omega} v\,\mathrm ds$ is the average of $v$ on $\omega$. 
\end{mylemma}

\begin{proof}
	Let $v\in H^\frac{1}{2}(\omega)$. Recall that since $\omega$ is an $(n-1)$-dimensional manifold in $\mathbb{R}^n$, then
	$$|v|^2_{\frac{1}{2},\omega} = \int_{\omega}\int_{\omega}\frac{\big(v(x)-v(y)\big)^2}{|x-y|^n}\,\mathrm{d}x\,\mathrm{d}y.$$
	Moreover, let $$\omega_{\mathrm{max}} := \arg\displaystyle\max_{\omega_c \in \text{conn}(\omega)} \big(\mathrm{diam}\left(\omega_c\right)\big).$$
	Then since $\omega$ is isotropic,
	$$\mathrm{diam}\left(\omega\right) \lesssim \mathrm{diam}\left(\omega_{\mathrm{max}}\right) \lesssim \left| \omega_{\mathrm{max}} \right|^\frac{1}{n-1} \leq |\omega|^\frac{1}{n-1}.$$
	Therefore, 
	\begin{align*}
	\|v-\bar{v}\|_{0,\omega}^2 &= \int_{\omega}\left(v(x)-\frac{1}{\left|\omega\right|} \int_{\omega}v(y)\,\mathrm{d}y \right)^2\,\mathrm{d}x \\
	& = \frac{1}{|\omega|^2} \int_{\omega}\left[\int_{\omega}\big(v(x)-v(y)\big)\,\mathrm{d}y \right]^2\,\mathrm{d}x \\
	& \leq \frac{1}{|\omega|^2} \int_{\omega} \left[  |\omega| \int_{\omega} \big(v(x)-v(y)\big)^2\,\mathrm{d}y\right] \,\mathrm{d}x \\
	& = \frac{1}{|\omega|} \int_{\omega} \int_{\omega} \frac{\big(v(x)-v(y)\big)^2}{|x-y|^n}|x-y|^n\,\mathrm{d}y \,\mathrm{d}x \\
	& \leq \frac{\mathrm{diam}\left(\omega\right)^n}{|\omega|} \int_{\omega} \int_{\omega} \frac{\big(v(x)-v(y)\big)^2}{|x-y|^n}\,\mathrm{d}y \,\mathrm{d}x 
	\lesssim |\omega|^\frac{1}{n-1} |v|^2_{\frac{1}{2}, \omega}.
	\end{align*}
\end{proof}

\begin{mylemma}[] \label{lemma:savare}
	Let $D\subset \mathbb{R}^n$, and let $\omega\subset \partial D$ be a $(n-1)$-dimensional manifold in $\mathbb{R}^n$. Then for all $v\in H^\frac{1}{2}(\partial D)$,  if we define $\eta\in\mathbb{R}$ as the unique solution of $\eta = -\log(\eta)$, 
	$$\| v \|_{0,\omega} \lesssim c_\omega |\omega|^\frac{1}{2(n-1)} \|v\|_{\frac{1}{2}, \partial D}, \quad \text{where } \quad c_{\omega} := \begin{cases}
	\max\big(\hspace{-0.05cm}\left|\log\left(|\omega|\right)\right|, \eta \big)^\frac{1}{2} & \text{ if } n = 2; \\
	1 & \text{ if } n = 3. 
	\end{cases}$$
	The hidden constant is independent from the measure of $\omega$. 
\end{mylemma}
\begin{proof}
	Let $v\in H^\frac{1}{2}(\partial D)$. By Sobolev embedding, it is well known that $H^\frac{1}{2}(\partial D)$ can be continuously embedded in $L^{2p}(\partial D)$ for every $1\leq p<\infty$ if $n=2$, or for every $1\leq p\leq 2$ if $n=3$. Therefore, by H\"older inequality, 
	\begin{equation} \label{eq:averageerror}
	\|v \|^2_{0,\omega} = \sum_{\omega_c\in \text{conn}(\omega)} \|v \|^2_{0,\omega_c} \leq \sum_{\omega_c\in \text{conn}(\omega)}  \left|\omega_c\right|^{1-\frac{1}{p}} \|v\|^2_{L^{2p}\left(\omega_c\right)} \lesssim |\omega|^{1-\frac{1}{p}} \|v\|^2_{L^{2p}(\partial D)}.
	\end{equation}
	If $n=3$, by taking $p=2$ in (\ref{eq:averageerror}) and by Sobolev embedding, 
	\begin{equation*} 
	\|v \|_{0,\omega}^2 \lesssim |\omega|^{\frac{1}{2}} \|v\|^2_{\frac{1}{2},\partial D} = c_{\omega}^2|\omega|^\frac{1}{n-1} \|v\|^2_{\frac{1}{2},\partial D}.
	\end{equation*}
	Let us now consider the case $n=2$. Thanks to Ref.~\cite[Lemma~5.1]{benbelgacembuffamaday}, it is known that for all $q\in [2,\infty)$ and all $v\in H^\frac{1}{2}(0,2\pi)$, 
	$$\|v\|_{L^q(0,2\pi)} \leq c \sqrt{q} \|v\|_{\frac{1}{2}, (0,2\pi)},$$
	where $c$ is a constant independent from $q$. 
	Then by definition of the norms $L^q$ and $H^\frac{1}{2}$ on a manifold (see Ref.~\cite[Sec.~1.3.3]{grisvard}), we obtain
	$\|v\|_{L^{2p}(\partial D)} \leq \tilde c \sqrt{p} \|v\|_{\frac{1}{2},\partial D}$, where $\tilde c$ is a constant independent from $p$. So by taking $p=\max\big(\hspace{-0.05cm}\left|\log\left(|\omega|\right)\right|, \eta \big) = c_\omega^2$ in (\ref{eq:averageerror}), then $|\omega|^{-\frac{1}{p}} \leq e$ and thus
	\begin{align*} 
	\|v \|_{0,\omega}^2 &\lesssim |\omega|^{1-\frac{1}{p}} p\, \|v\|_{\frac{1}{2},\partial D}^2 \lesssim |\omega| c_\omega^2 \|v\|^2_{\frac{1}{2},\partial D} = c_{\omega}^2|\omega|^\frac{1}{n-1} \|v\|^2_{\frac{1}{2},\partial D}. 
	\end{align*}
\end{proof}

\begin{mylemma} [Poincar\'e II]\label{lemma:poincarebd2}
	Let $\omega$ be an $(n-1)$-dimensional manifold in $\mathbb{R}^{n}$ that is isotropic according to Definition \ref{as:isotropy}. 
	Then for all $v\in H^{\frac{1}{2}}_{00}(\omega)$, 
	$$\|v\|_{0,\omega} \lesssim \left|\omega\right|^\frac{1}{2(n-1)} \|v\|_{H^{1/2}_{00}(\omega)}.$$
\end{mylemma}
\begin{proof}
	Let $D\subset \mathbb R^n$ and $\varphi \subset \mathbb R^n$ such that $\partial D = \overline\omega\cup\overline\varphi$ and $\omega\cap \varphi = \emptyset$. 
	Let $v\in H_{00}^\frac{1}{2}(\omega)$. First, suppose that $\omega$ is connected. Then from Ref.~\cite[Proposition~2.4]{poincare2}, since $\omega$ is isotropic and if we let $v^\star \in H^\frac{1}{2}(\partial D)$ be the extension of $v$ by $0$,
	\begin{equation} \label{eq:poincare2connected}
	\left\| v \right\|_{0,\omega}^2 \lesssim \left|\omega\right|^\frac{1}{n-1} \big|v^\star\big|^2_{\frac{1}{2},\partial D} \leq \left|\omega\right|^\frac{1}{n-1} \left\|v^\star\right\|^2_{\frac{1}{2},\partial D} = \left|\omega\right|^\frac{1}{n-1} \|v\|^2_{H_{00}^{1/2}(\omega)}.
	\end{equation}
	Now, if $\omega$ is not connected, then dist$\left(s,\partial \omega_c\right)$ = dist$(s,\partial \omega)$ for all $\omega_c\in \text{conn}(\omega)$ and for all $s\in \omega_c$.
	Thus 
	\begin{align*}
	\sum_{\omega_c\in \text{conn}(\omega)} \left\|v\vert_{\omega_c}\right\|^2_{H_{00}^{1/2}\left(\omega_c\right)} &= \sum_{\omega_c\in \text{conn}(\omega)} \left( \big\|v\vert_{\omega_c}\big\|^2_{\frac{1}{2},\omega_c} + \big|v\vert_{\omega_c}\big|^2_{H_{00}^{1/2}(\omega_c)} \right) \\
	& \lesssim \|v\|_{\frac{1}{2},\omega}^2 + \sum_{\omega_c\in \text{conn}(\omega)} \int_{\omega_c} \frac{v^2(s)}{\text{dist}\big(s,\partial\omega_c\big)}\,\mathrm ds\\
	&= \|v\|_{\frac{1}{2},\omega}^2 + \int_{\omega} \frac{v^2(s)}{\text{dist}\big(s,\partial\omega\big)}\,\mathrm ds
	= \|v\|^2_{H_{00}^{1/2}(\omega)}.
	\end{align*}
	Therefore, from (\ref{eq:poincare2connected}), 
	\begin{equation*}
	\|v\|_{0,\omega}^2 = \sum_{\omega_c\in \text{conn}(\omega)} \|v\|^2_{0,\omega_c} \lesssim \sum_{\omega_c\in \text{conn}(\omega)} \left|\omega_c\right|^\frac{1}{n-1} \|v\|^2_{H_{00}^{1/2}\left(\omega_c\right)} \lesssim \left|\omega\right|^\frac{1}{n-1} \|v\|_{H^{1/2}_{00}(\omega)}^2.
	\end{equation*}
\end{proof}

\begin{mylemma}[] \label{lemma:inverseineq2}
	Let $\omega$ be an $(n-1)$-dimensional manifold in $\mathbb{R}^{n}$ that is isotropic according to Definition \ref{as:isotropy}. Then for all $v\in L^2(\omega)$, 
	$$\|v\|_{H_{00}^{-1/2}(\omega)} \lesssim \left|\omega\right|^\frac{1}{2(n-1)} \|v\|_{0,\omega}.$$
\end{mylemma}
\begin{proof}
	Since $H_{00}^{-\frac{1}{2}}(\omega)$ is the dual space of $H_{00}^\frac{1}{2}(\omega)$,
	then by Appendix \ref{lemma:poincarebd2}, we obtain
	\begin{align*}
	\|v\|_{H_{00}^{-1/2}(\omega)} &= \sup_{\substack{z\in H^{1/2}_{00}(\omega)\\ z\neq 0}} \frac{\displaystyle\int_{\omega}vz \,\mathrm ds}{\|z\|_{H_{00}^{1/2}(\omega)}} \\
	&\leq \sup_{\substack{z\in H^{1/2}_{00}(\omega)\\ z\neq 0}} \frac{\|v\|_{0,\omega}\,\|z\|_{0,\omega}}{\|z\|_{H_{00}^{1/2}(\omega)}} \\
	& \lesssim \sup_{\substack{z\in H^{1/2}_{00}(\omega)\\ z\neq 0}} \frac{\|v\|_{0,\omega} \left|\omega\right|^\frac{1}{2(n-1)} \|z\|_{H_{00}^{1/2}(\omega)}}{\|z\|_{H_{00}^{1/2}(\omega)}} = \left|\omega\right|^\frac{1}{2(n-1)} \|v\|_{0,\omega}.
	\end{align*}
\end{proof}

\begin{mylemma}[Inverse inequality I] \label{lemma:inverseineq}
	Let $\omega$ be an open $(n-1)$-dimensional manifold in $\mathbb{R}^{n}$ that is isotropic and regular according to Definitions \ref{as:isotropy} and \ref{as:pwsmoothshapereg}, and let ${m} \in \mathbb N$. 
	Then for all $p\in \mathbb Q_{{m},0}^\mathrm{pw}(\omega)$, 
	$$\|p\|_{0,\omega} \lesssim |\omega|^{-\frac{1}{2(n-1)}} \|p\|_{H_{00}^{-1/2}(\omega)},$$
	where the hidden constant increases with ${m}$. 
\end{mylemma}
\begin{proof}
	For all $q\in \mathbb Q_{{m},0}^\text{pw}(\omega)\subset H_0^1(\omega)$, the following inverse estimate is well known (see Ref.~\cite[Theorem~3.2]{graham} for example): with the notation of Definition \ref{as:pwsmoothshapereg}, for all $\ell=1,\ldots,L_\omega$,
	$$ \big| q\vert_{\omega_\ell} \big|_{1,\omega_\ell} \lesssim \left|\omega_\ell\right|^{-\frac{1}{n-1}} \big\| q\vert_{\omega_\ell} \big\|_{0,\omega_\ell},$$
	and the hidden constant increases with $m$. Therefore, since $\omega$ is isotropic and shape regular,
	$$\left| q \right|_{1,\omega} \lesssim \max_{\ell=1,\ldots,L_\omega}\left(\left|\omega_\ell\right|^{-\frac{1}{n-1}}\right) \left\| q \right\|_{0,\omega}\lesssim \left|\omega\right|^{-\frac{1}{n-1}} \left\| q \right\|_{0,\omega}.$$
	Moreover, from Ref.~\cite{triebel}, we know that the interpolation space $$\left[ H_0^1\left(\omega\right), L^2\left(\omega\right)\right]_\frac{1}{2} = H_{00}^\frac{1}{2}\left(\omega\right)$$
	(see also Ref.~\cite[Theorem~11.7]{lionsmagenes}). Therefore, from Ref.~\cite[Proposition~2.3]{lionsmagenes}, for all $q\in \mathbb Q_{{m},0}^\text{pw}(\omega)$, 
	\begin{align} \label{eq:inverseH0012}
	\|q\|_{H_{00}^{1/2}(\omega)} &\lesssim |q|_{1,\omega}^\frac{1}{2} \,\|q\|_{0,\omega}^\frac{1}{2} \lesssim \left|\omega\right|^{-\frac{1}{2(n-1)}} \|q\|_{0,\omega}.
	\end{align}
	Consequently, for all $p\in \mathbb Q_{{m},0}^\text{pw}(\omega)\subset H_{00}^{-\frac{1}{2}}(\omega)$, since $\mathbb Q_{{m},0}^\text{pw}(\omega)\subset H_{00}^{\frac{1}{2}}(\omega)$,
	\begin{align}
	\|p\|_{0,\omega} = \frac{\displaystyle \int_\omega p^2 \,\mathrm ds}{\|p\|_{0,\omega}} \leq \sup_{\substack{q\in \mathbb Q_{{m},0}^\text{pw}(\omega)\\ q\neq 0}} \frac{\displaystyle \int_\omega pq \,\mathrm ds}{\|q\|_{0,\omega}} &\lesssim \left|\omega\right|^{-\frac{1}{2(n-1)}} \sup_{\substack{q\in \mathbb Q_{{m},0}^\text{pw}(\omega)\\ q\neq 0}} \frac{\displaystyle \int_\omega pq \,\mathrm ds}{\|q\|_{H_{00}^{1/2}(\omega)}} \nonumber \\
	& \leq \left|\omega\right|^{-\frac{1}{2(n-1)}} \sup_{\substack{v\in H_{00}^{1/2}(\omega)\\ v\neq 0}} \frac{\displaystyle \int_\omega pv \,\mathrm ds}{\|v\|_{H_{00}^{1/2}(\omega)}} \nonumber \\
	&= \left|\omega\right|^{-\frac{1}{2(n-1)}} \|p\|_{H_{00}^{-1/2}(\omega)}. \label{eq:stepsproofinversedual}
	\end{align}\\
	
\end{proof}

For the following lemmas, let $D\subset \mathbb R^n$ be an open bounded domain, and let $\partial D = \bigcup_{k=1}^{K+1}\overline{\omega_k}$ for some $K\in\mathbb{N}$, such that $\omega_i \cap \omega_j = \emptyset$, for all $i,j=1,\ldots,K+1$, and let $\omega = \text{int}\left(\bigcup_{k=1}^K\overline{\omega_k}\right)$. Moreover, let
$$H:= \left\{ v\in H_{00}^\frac{1}{2}(\omega) : v|_{\omega_k} \in H_{00}^\frac{1}{2}(\omega_k), \forall k=1,\ldots,K \right\} \subset H_{00}^\frac{1}{2}(\omega)$$
equipped with the norm $\|\cdot\|_H := \displaystyle \left(\sum_{k=1}^K \big\| \cdot\vert_{\omega_k} \big\|^2_{H_{00}^\frac{1}{2}(\omega_k)}\right)^\frac{1}{2}$, and let $H^*$ be its dual space, equipped with the dual norm $\|\cdot\|_{H^*}$. 

\begin{mylemma}[] \label{lemma:sumH0012}
	For all $v\in H$,
	$$\left\| v \right\|_{H_{00}^{1/2}\left( \omega\right)} \leq \sqrt{K} \|v\|_H,$$
	and for all $w \in H_{00}^{-\frac{1}{2}}(\omega)$, 
	$$\left\| w \right\|_{H^*} \leq \sqrt{K}\|w\|_{H_{00}^{-1/2}(\omega)}.$$
\end{mylemma}
\begin{proof}
	Let $v\in H\subset H_{00}^\frac{1}{2}(\omega)$, and let $v\vert_{\omega_k}^\star$ be the extension of $v\vert_{\omega_k}$ by $0$ on $\partial D$. 
	Then by triangular inequality,
	\begin{align}
	\| v\|_{H_{00}^{1/2}(\omega)} = \left\| \sum_{k=1}^K v|_{\omega_k}^\star \right\|_{H_{00}^{1/2}(\omega)} &\leq \sum_{k=1}^K \left\| v|_{\omega_k}^\star \right\|_{H_{00}^{1/2}(\omega)} = \sum_{k=1}^K \left\| v|_{\omega_k}^\star \right\|_{\frac{1}{2}, \partial D} \nonumber \\
	&= \sum_{k=1}^K \left\| v|_{\omega_k}^\star \right\|_{H_{00}^{1/2}(\omega_k)} \leq \sqrt{K} \|v\|_H. \label{eq:justproved}
	\end{align}
	Moreover, for all $w\in H_{00}^{-\frac{1}{2}}(\omega) \subset H^*,$ using (\ref{eq:justproved}), 
	\begin{align*}
	\|w\|_{H^*} = \sup_{\substack{v\in H\\ v\neq 0}} \frac{\displaystyle\int_\omega wv \,\mathrm ds}{\|v\|_H} 
	&\leq \sqrt{K} \sup_{\substack{v\in H\\ v\neq 0}} \frac{\displaystyle\int_\omega wv \,\mathrm ds}{\|v\|_{H_{00}^{1/2}(\omega)}} \\
	&\leq\sqrt{K}\sup_{\substack{v\in H_{00}^\frac{1}{2}(\omega)\\ v\neq 0}} \frac{\displaystyle\int_\omega wv \,\mathrm ds}{\|v\|_{H_{00}^{1/2}(\omega)}} = \sqrt{K} \|w\|_{H_{00}^{-1/2}(\omega)}.
	\end{align*}
\end{proof}

\begin{mylemma}[Inverse inequality II] \label{lemma:inverseineqH}
	Assume that $\omega$ is isotropic and regular according to Definitions \ref{as:isotropy} and \ref{as:pwsmoothshapereg}, and let ${m} \in \mathbb N$. 
	Then for all piecewise polynomial $p\in \mathbb Q_m^0$, where $$\mathbb Q_m^0 := \left\{ q \in \mathbb Q_{{m},0}^\mathrm{pw}(\omega) : q\vert_{\omega_k}\in \mathbb Q_{{m},0}^\mathrm{pw}(\omega_k), \forall k=1,\ldots,K \right\}\subset H,$$
	then
	$$\|p\|_{0,\omega} \lesssim |\omega|^{-\frac{1}{2(n-1)}} \|p\|_{H^*},$$
	where the hidden constant increases with ${m}$. 
\end{mylemma}
\begin{proof}
	For all $q\in \mathbb Q_m^0$ and all $k=1,\ldots,K$, $q|_{\omega_k}\in \mathbb Q_{m,0}^\text{pw}(\omega_k)$, and thus from (\ref{eq:inverseH0012}) and since $\omega$ is regular, 
	$$\big\|q|_{\omega_k}\big\|_{H_{00}^{1/2}(\omega_k)} \lesssim \left|\omega_k\right|^{-\frac{1}{2(n-1)}} \big\|q|_{\omega_k}\big\|_{0,\omega_k} \lesssim |\omega|^{-\frac{1}{2(n-1)}}\big\|q|_{\omega_k}\big\|_{0,\omega_k}.$$
	Therefore, 
	$$\|q\|_H = \left(\sum_{k=1}^K \big\| \cdot\vert_{\omega_k} \big\|^2_{H_{00}^\frac{1}{2}(\omega_k)}\right)^\frac{1}{2} \lesssim |\omega|^{-\frac{1}{2(n-1)}} \|q\|_{0,\omega}.$$
	Consequently, for all $p\in \mathbb Q_m^0\subset H$, following the same steps as in (\ref{eq:stepsproofinversedual}) of the proof of Appendix \ref{lemma:inverseineq}, replacing $H_{00}^{-\frac{1}{2}}(\omega)$ by $H^*$, $H_{00}^{\frac{1}{2}}(\omega)$ by $H$, and $\mathbb Q_{{m},0}^\text{pw}(\omega)$ by $\mathbb Q_m^0$, then
	$$\|p\|_{0,\omega} \lesssim |\omega|^{-\frac{1}{2(n-1)}} \|p\|_{H^*}.$$
\end{proof}

\section*{Acknowledgment}
{The authors would like to thank Stefan Sauter for the fruitful discussions on the subject. Moreover,} the authors gratefully acknowledge the support of the European Research Council, via the ERC AdG project CHANGE n.694515. 
R. V\'azquez also thanks the support of the Swiss National Science Foundation via the project HOGAEMS n.200021\_188589.

\addcontentsline{toc}{section}{Bibliography}
\bibliography{bib}
\bibliographystyle{ieeetr}








\end{document}